\documentclass[sts,
	reqno,
]{imsart}

\newcommand{\declarecommand}[1]{\providecommand{#1}{}\renewcommand{#1}} 
\usepackage{amsmath, amssymb, color}
\usepackage{amsthm} 
\usepackage[colorlinks,citecolor=blue,urlcolor=blue,pdfusetitle=false]{hyperref}
\usepackage{xcolor,colortbl}
\usepackage{color}
\usepackage{amsfonts, fancybox}
\usepackage{amssymb}
\usepackage[american]{babel}
\usepackage{enumitem}
\setenumerate{label = (\roman*)}
\usepackage{psfrag,color}
\usepackage{dcolumn}
\usepackage{caption}
\usepackage[labelformat=simple]{subcaption}

\usepackage{flafter, bm, dsfont}
\usepackage[section]{placeins}

\usepackage{multirow}
\usepackage{color}
\usepackage{amsmath}
\usepackage{mathrsfs}
\usepackage{amsfonts}
\usepackage{dsfont}
\usepackage{amssymb}
\usepackage{algorithmicx, algorithm, algpseudocode}
\usepackage{graphicx}

\usepackage[noadjust]{cite}
\usepackage{tikz}
\usetikzlibrary{positioning}

\newcommand{\bc}{\begin{center}}
\newcommand{\ec}{\end{center}}
\newcommand{\ba}{\begin{array}}
\newcommand{\ea}{\end{array}}
\newcommand{\be}{\begin{eqnarray}}
\newcommand{\ee}{\end{eqnarray}}
\newcommand{\bel}{\begin{eqnarray}\label}
\newcommand{\eel}{\end{eqnarray}}
\newcommand{\bes}{\begin{eqnarray*}}
\newcommand{\ees}{\end{eqnarray*}}
\newcommand{\bn}{\begin{enumerate}}
\newcommand{\en}{\end{enumerate}}

\newcommand{\KL}{\mathsf{KL}}
\newcommand{\sg}{\mathsf{subG}}
\newcommand{\bern}{\mathsf{Bern}}

\newcommand{\subE}{\mathsf{subE}}

\definecolor{MIT}{cmyk}{.24, 1.00, .78, .17} 
\definecolor{pink}{cmyk}{0, 1, 0, 0} 
\definecolor{darkgreen}{cmyk}{1,0, 1, 0}

\makeatletter
\newcommand{\iid}{\@ifnextchar{.}
	{{\it i.i.d}}
	{\@ifnextchar{,}
		{{\it i.i.d.}}
		{{\it i.i.d.\ }}
	}
}
\makeatother

\newtheorem{theorem}{Theorem}
\newtheorem*{theorem*}{Theorem}
\newtheorem{lemma}[theorem]{Lemma}
\newtheorem{definition}[theorem]{Definition}
\newtheorem{proposition}[theorem]{Proposition}
\newtheorem*{proposition*}{Proposition}

\newtheorem{remark}[theorem]{Remark}
\newcommand{\tr}{\mathop{\mathsf{Tr}}}

\newcommand{\diag}{\mathop{\mathsf{diag}}}

\makeatletter
\newtheorem*{rep@theorem}{\rep@title}
\newcommand{\newreptheorem}[2]{%
\newenvironment{rep#1}[1]{%
 \def\rep@title{#2 \ref{##1}}%
 \begin{rep@theorem}}%
 {\end{rep@theorem}}}
\makeatother
\newreptheorem{lemma}{Lemma}
\newreptheorem{theorem}{Theorem}
\newreptheorem{proposition}{Proposition}

\numberwithin{equation}{section}
\usepackage[utf8]{inputenc}
\newcommand{\leadeq}[2][4]{\MoveEqLeft[#1] #2 \nonumber}

\usepackage{stmaryrd}
\usepackage{bbm}
\usepackage[abbreviations]{foreign}

\usepackage{mathtools}
\mathtoolsset{showonlyrefs}
\makeatletter
\newcommand{\subalign}[1]{%
  \vcenter{%
    \Let@ \restore@math@cr \default@tag
    \baselineskip\fontdimen10 \scriptfont\tw@
    \advance\baselineskip\fontdimen12 \scriptfont\tw@
    \lineskip\thr@@\fontdimen8 \scriptfont\thr@@
    \lineskiplimit\lineskip
    \ialign{\hfil$\m@th\scriptstyle##$&$\m@th\scriptstyle{}##$\crcr
      #1\crcr
    }%
  }
}
\makeatother
\newcommand{\cA}{\mathcal{A}}
\newcommand{\cB}{\mathcal{B}}
\newcommand{\cC}{\mathcal{C}}

\newcommand{\cE}{\mathcal{E}}
\newcommand{\cF}{\mathcal{F}}
\newcommand{\cG}{\mathcal{G}}

\newcommand{\cJ}{\mathcal{J}}

\newcommand{\cL}{\mathcal{L}}

\newcommand{\cN}{\mathcal{N}}

\newcommand{\cT}{\mathcal{T}}
\newcommand{\cU}{\mathcal{U}}

\newcommand{\hB}{\hat{B}}
\newcommand{\hBi}{\hat{B}_{\mathrm{init}}}
\newcommand{\hT}{\hat{T}}
\newcommand{\hb}{\hat{b}}
\newcommand{\hatt}{\hat{t}}
\newcommand{\init}{\mathrm{init}}
\newcommand{\loc}{\mathrm{loc}}
\newcommand{\llc}{\mathrm{llc}}

\newcommand{\R}{{\rm I}\kern-0.18em{\rm R}}
\newcommand{\h}{{\rm I}\kern-0.18em{\rm H}}
\newcommand{\K}{{\rm I}\kern-0.18em{\rm K}}
\newcommand{\p}{{\rm I}\kern-0.18em{\rm P}}
\newcommand{\E}{{\rm I}\kern-0.18em{\rm E}}
\newcommand{\Z}{{\rm Z}\kern-0.18em{\rm Z}}
\newcommand{\1}{{\rm 1}\kern-0.24em{\rm I}}
\newcommand{\N}{{\rm I}\kern-0.18em{\rm N}}

\newcommand{\pn}{\p_{\kern-0.25em n}}
\newcommand{\pnm}{\p_{\kern-0.25em n,m}}
\newcommand{\psubm}{\p_{\kern-0.25em m}}
\newcommand{\psubp}{\p_{\kern-0.25em p}}
\newcommand{\cfi}{\cF_{\kern-0.25em \infty}}

\newcommand{\argmin}{\mathop{\mathrm{argmin}}}

\newlength{\minipagewidth}
\begin{document}

\declarecommand{\supp}{\operatorname{supp}}
\declarecommand{\nullspace}{\operatorname{null}}
\declarecommand{\rank}{\operatorname{rank}}
\declarecommand{\normop}[1]{\|#1\|_\mathsf{op}}
\declarecommand{\normF}[1]{\|#1\|_\mathsf{F}}
\declarecommand{\lamphat}{\hat{\lambda}_k}
\declarecommand{\thetaphat}{\hat{\theta}_k}
\declarecommand{\simiid}{\overset{\text\small\rm{iid}}{\sim}}
\declarecommand{\mockalph}[1]{}
\declarecommand{\doublesup}[4]{\mathrlap{{#1_{\mathrlap{#2}}^{#3}}^{#4}}\phantom{#1_{#2}^{#3 #4}}}
\declarecommand{\rmdolap}[1]{\mathrlap{\mathrm{do, }\, #1}}
\declarecommand{\rmdo}[1]{\mathrm{do, }\, #1}
\declarecommand{\pen}{\mathrm{pen}}
\declarecommand{\op}{\mathrm{op}}
\declarecommand{\constop}{\eta}
\declarecommand{\localradius}{R_\loc}
\declarecommand{\constre}{\rho}
\declarecommand{\constremax}{R}
\declarecommand{\constcollone}{\tilde R}
\declarecommand{\eventllc}{\cA}
\declarecommand{\stocherrllc}{\phi_n}
\declarecommand{\matrixclass}{\cB}
\declarecommand{\classgen}{\cG}
\newtheorem{assumption}{} %
\renewcommand{\theassumption}{A\arabic{assumption}}


\begin{frontmatter}

	\title{Estimation Rates for Sparse Linear Cyclic Causal Models}
	\runtitle{Linear cyclic rates}

	\author{ \fnms{Jan-Christian} \snm{H\"utter}\thanksref{}\ead[label=huetter]{huetter@math.mit.edu} 
		~and~
		\fnms{Philippe} \snm{Rigollet}\thanksref{t2}\ead[label=rigollet]{rigollet@math.mit.edu}
	}
	\affiliation{Massachusetts Institute of Technology}

	\thankstext{t2}{Supported by NSF awards IIS-1838071, DMS-1712596 and DMS-TRIPODS-1740751; ONR grant N00014-17- 1-2147 and grant 2018-182642 from the Chan Zuckerberg Initiative DAF.}

	\address{{Jan-Christian Hütter}\\
		{Department of Mathematics} \\
		{Massachusetts Institute of Technology}\\
		{77 Massachusetts Avenue,}\\
		{Cambridge, MA 02139-4307, USA}\\
		\printead{huetter}
	}

	\address{{Philippe Rigollet}\\
		{Department of Mathematics} \\
		{Massachusetts Institute of Technology}\\
		{77 Massachusetts Avenue,}\\
		{Cambridge, MA 02139-4307, USA}\\
	}

	\runauthor{H\"utter, Rigollet}

	\begin{abstract}
		Causal models are important tools to understand complex phenomena and predict the outcome of controlled experiments, also known as interventions.
		In this work, we present statistical rates of estimation for linear cyclic causal models under the assumption of homoscedastic Gaussian noise by analyzing both the LLC estimator introduced by Hyttinen, Eberhardt and Hoyer and a novel two-step penalized maximum likelihood estimator.
		We establish asymptotic near minimax optimality for the maximum likelihood estimator over a class of sparse causal graphs in the case of near-optimally chosen interventions.
		Moreover, we find evidence for practical advantages of this estimator compared to LLC in synthetic numerical experiments.
	\end{abstract}


\end{frontmatter}

\section{Introduction} 
\label{sec:introduction}

Directed graphical models \cite{Pea09, SpiGlySch00} provide a useful framework for interpretation, inference and decision making in many areas of science such as biology, sociology, and environmental sciences \cite{FriLinNac00, Dun66, KeaHit88}.
Unlike their undirected counterparts that merely encode the structure of probabilistic dependence between random variables directed graphical models reveal causal effects that are the basis of scientific discovery~ \cite{Pea09}.

Most frequently, the model is assumed to be governed by a directed acyclic graph (DAG) \( G = (V, E) \), where \( V = \{ X_1, \dots, X_p \} \) are the variables of an observed system and \( E \) is a set of edges such that there is no directed cycle in \( G \). In such models, known as \emph{Bayes networks} \cite{Pea09}, the variables follow a joint distribution that factorizes according to the graph \( G \) in the sense that node $i$ is independent of other nodes conditionally on its parents.
The absence of cycles allows for a direct interpretation of the causal structure between the variables $X_1, \ldots, X_p$ whereby a directed edge corresponds to a causal effect. At the same time, most complex systems showcase feedback loops that can be both positive and negative and the need to extend Bayes networks to allow for cycles was recognized long ago.

A large body of work focuses on learning Bayes from \emph{observational} data, that is, data drawn independently from the joint distribution of $(X_1, \dots, X_p)$. Observational data is rather abundant but even in the cyclic cases, it is known to lead to a severe lack of identifiability: Such data, even in infinite abundance, can only yield an equivalence class---the Markov equivalence class---of DAGs that are all compatible with the conditional independence relation in the given data. While a DAG in the Markov equivalence class can already yield decisive scientific insight~\cite{MaaKalBuh09}, searching over the space of DAGS is often computationally hard. Many algorithms have been proposed over the years such as the PC algorithm \cite{SpiGlySch00} and Greedy Equivalence search \cite{Chi02a} and max-min hill climbing \cite{TsaBroAli06}, but all of them rely on the notion of faithfulness of the distribution, \ie the assumption that all conditional dependence relations that could be compatible with the DAG $G$ are actually fulfilled by the distribution of \( X \).
In fact, for consistency of these algorithms, one needs to assume that these dependencies observe a signal-to-noise ratio that allows to detect them with high probability~\cite{KalBuh07, LohBuh14, GeeBuh13}.
Extensions that allow certain kinds of cycles, \cite{Ric96a, RicSpi96, SchMur09, ItaOhaSac10, LacSpiRam12} have been proposed but at the expense of having an increased number of graphs in each equivalence class.

Recent breakneck advances in data collection processes such as the spread of A/B testing for online marketing or targeted gene editing with CRISPR-Cas9 are contributing to the proliferation of \emph{interventional} data, the gold standard for causal inference. With unlimited interventions on any combination of nodes, learning a directed graphical model becomes a trivial task. However, exhaustively performing all interventions is a daunting and costly task and recent work has focused on finding a small number of interventions for several classes of DAGs~\cite{ShaKocDim15, KocDimVis17}. For graphs with cycles, \cite{HytEbeHoy12} have characterized the system of interventions necessary to learn a parametric linear structural equation model (SEM) \cite{BieHau77, Bol83}, in which all variables are real valued and the causal relationships given by the edges \( E \) are linear. Formally, this model postulates that the following equation holds (in distribution) for observational samples from \( X \):
\begin{equation}
	\label{eq:gy}
	X = B^\ast X + Z, \quad Z \sim \cN(0, I),
\end{equation}
where we exclude explicit self-loops by assuming that the diagonal of \( B^\ast \in \R^{p \times p} \) is zero.
By writing \( X = (I - B^\ast)^{-1} Z \) and assuming that the corresponding inverse matrix exists, this allows us to handle underlying graphs that are cyclic.
This model has been extensively studied in \cite{HytEbeHoy12}, where it is shown that if we have access to data from a sufficiently rich system of interventions, \ie if enough variables are randomized and are thus made independent of the influence of their parents encoded in \( B^\ast \), then on a population level, \( B^\ast \) is identifiable by a method of moments type estimator that the authors call LLC (for linear, latent, causal).

In this paper, we present upper and lower bounds for the reconstruction of \( B^\ast \) in Frobenius norm for classes of sparse \( B^\ast \), corresponding to graphs with bounded in-degree, using multiple observations for each intervention setup.
We also provide upper bounds  for the original LLC estimator with \( \ell_1 \)-penalization term as well as an \( \ell_1 \)-penalized maximum likelihood estimator, all under the simplifying assumption that the noise or disturbance variables \( Z \) are Gaussian, independent of each other, and have unit variance.
Moreover, we provide numerical evidence that a non-convex ADMM type algorithm can be used to find a solution to this maximum likelihood problem, albeit without convergence guarantees.

\subsection{Related work}

It is known that several variants of the model~\eqref{eq:gy} are identifiable from observational data,  including non-linear SEMs \cite{HoyJanMoo09} or non-Gaussian noise~\cite{ShiHoyHyv06}. Linear SEMs with Gaussian noise can be identifiable under additional assumptions, for example when the components of the noise have equal variances and the underlying graph is a DAG~\cite{LohBuh14, PetBuh14} or when the underlying graph is random and sparse~\cite{AbrRig18}.

As for assumptions on the noise that guarantee identifiability from observational data, one example is the recovery of a linear structural equation model under non-Gaussian noise via Independent Component Analysis, \cite{ShiHoyHyv06}, which under additional restrictive assumptions on the structure of the underlying graph can be extended to the Gaussian case, \cite{AbrRig18}.

Moreover, many more approaches to dealing with cycles and/or interventions are known, such as convex regularizers in an exponential family model \cite{SchNicMur07, SchMur09}, independence testing \cite{ItaOhaSac10}, Independent Component Analysis \cite{LacSpiRam12}, and adapting Greedy Equivalence Search to handle interventional data \cite{HauBuh12, WanSegUhl18}.
From the above, it seems that the linear Gaussian case is somewhat of a worst case example for identifiability of the ground truth matrix, especially when allowing cycles, and thus warrants the investigation of controlled interventions to eliminate ambiguity, which is the main contribution of \cite{HytEbeHoy12}.
Similar models have been considered for applications, for example in computational biology, see \cite{CaiBazGia13}, where identifiability is not provided by controlled experiments on the variance, but rather by a mean shift of some variables.

Our work extends the results in \cite{HytEbeHoy12} by providing explicit upper bounds for their suggested method, as well as presenting an alternative estimator that leads to upper bounds independent of the conditioning of the experiments as explained in Section \ref{sec:mle}.
In spirit, our results are similar to consistency guarantees obtained in \cite{GeeBuh13} and \cite{WanSegUhl18}, but we focus on the case where enough interventions are performed to identify the ground truth structure matrix \( B^\ast \), alleviating the need for additional assumptions on \( B^\ast \).

\subsection{Structure of the paper}

The rest of the paper is structured as follows:
In Section \ref{sec:model}, we give an overview of the linear structural equation model we consider and the main assumptions we make.
In Section \ref{sec:results}, we present lower bounds, upper bounds for LLC, and upper bounds for a two-step maximum likelihood estimator.
In Section \ref{sec:numerics}, we derive a non-convex variant of ADMM to solve part of the numerical optimization problem for the penalized maximum likelihood estimator and explore its performance on synthetic and semi-synthetic data.
The proofs of the main results are deferred to Sections \ref{sec:lower-bound-proof} -- \ref{sec:2-step-rates-proof} in the appendix, and we collect general lemmas used in all the proofs in Section \ref{sec:general-lemmas}.
Section \ref{sec:idfblty} contains a short argument for why experimental data is necessary given our assumptions, and Section \ref{sec:low-rank-update} provides a way of speeding up our numerical calculations.

\subsection{Notation}

We write \( a \lesssim b \) for two quantities \( a \) and \( b \) if there exists an absolute constant \( C > 0 \) such that \( a \le C b \), and similarly for \( a \gtrsim b \).
For a natural number \( p \), we denote by \( [p] = \{1, \dots, p\} \).
Given a set \( S \), we write \( |S| \) for its cardinality.

Let \( x, y \in \R^p \).
We write \( \supp x \) for the indices of non-zero elements of \( x \),
\begin{equation}
	\label{eq:jz}
	d_H(x, y) = \{ i \in [p] : x \neq y \}
\end{equation}
for the Hamming distance between \( x \) and \( y \), and \( \| x \|_p \) for the \( \ell^p \) norm of \( x \).

For two matrices \( A, B \in \R^{p_1 \times p_2} \), we abbreviate the \( i \)th row by \( B_{i,:} \) and the \( i \)th column by \( B_{:,i} \).
Similarly, \( B_{i, -j} \) denotes the \( i \)th row of \( B \) where the \( j \)th element is omitted.
Further, \( \| B \|_{F} \) denotes the Frobenius norm, \( \| B \|_{\op} \) the operator norm,
\begin{equation}
	\label{eq:ka}
	\| B \|_\infty = \max_{i,j} | B_{i, j} |, \quad \| B \|_1 = \sum_{i, j} | B_{i, j} |,
\end{equation}
and \( \| B \|_{\infty, \infty} \) the operator norm of \( B \) with respect to the \( \ell^\infty \) norm, which is
\begin{equation}
	\label{eq:jy}
	\| B \|_{\infty, \infty} = \max_{i \in [p_1]} \| B_{i,:} \|_1.
\end{equation}
If $A$ is a square invertible matrix, we denote by $A^{-1}$ its inverse and by $A^{-\top}$ the transpose of $A^{-1}$.  We denote the smallest and largest singular value of \( A \) by \( \sigma_{\min}(A) \) and \( \sigma_{\max}(A) \), respectively. If \( A \) and \( B \) are symmetric, we write \( A \prec B \) if \( B - A \) is positive definite, and similarly for \( A \succ B \). 
By \( I \in \R^{p \times p} \), we denote the identity matrix.

For a function \( f : \R^{p_1} \to \R^{p_2} \), we denote its derivative at a point \( x \in \R^{p_1} \) applied to a vector \( h \in \R^{p_1} \) by \( Df(x)[h] \).
We write \( \sg \) and \( \subE \) to denote sub-Gaussian and sub-Exponential distributions as defined in Definition \ref{def:subg-sube}.

\section{Model and assumptions}
\label{sec:model}

Before summarizing our explicit assumptions, we give a definition of observations under a linear cyclic structural equation model with and without interventions.
We assume that a linear SEM on a random vector \( X = (X_1, \dots, X_p) \) is given by a matrix \( B^\ast \in \R^{p \times p} \) without self-cycles, \ie \( B^\ast \in \cB \) with
\begin{equation*}
	\cB := \{ B \in \R^{p \times p} : B_{i,i} = 0, \, \text{ for all } i = 1, \dots, p \}.
\end{equation*}
Without any intervention, each observation is an independent copy of $X = (I - B^\ast)^{-1} Z$,
where \( Z \) can in principle be any noise variable.
Since non-Gaussian noise can lead to identifiability from observational data through exploiting this particular property \cite{HoyJanMoo09,LacSpiRam12}, we focus on Gaussian noise, and make the simplifying assumption that \( Z \sim \cN(0, I) \).
In order to guarantee that \( (I - B^\ast)^{-1} \) exists, we assume $\| B^\ast \|_{\mathrm{op}} < 1$
which in particular allows us to write
\begin{equation}
	\label{eq:gf}
	X = \sum_{k = 0}^{\infty} (B^\ast)^k Z,
\end{equation}
and \( X \) can be interpreted as the steady state distribution of an auto-regressive process $\{x_t\}_{t\ge 0}$ governed by the dynamics 
\begin{equation}
\label{EQ:markovprocess}
x_{t + 1} = B^\ast x_t, \quad x_0 = Z.
\end{equation}
Hence, \( X \) is distributed according to \( X \sim \cN(0, \Sigma^\ast) \) with
\begin{equation*}
	\Sigma^\ast = (I-B^\ast)^{-1}(I - B^\ast)^{-\top}.
\end{equation*}

In order to obtain results in the high-dimensional regime \( p \asymp n \), we additionally assume that the in-degree of \( B^\ast \) is bounded, resulting in a sparse matrix \( B^\ast \).
That is, if we denote the maximum in-degree of a matrix \( B \in \R^{p \times p} \) by
\begin{equation}
	\label{eq:is}
	d(B) = \max_{i \in [p]} | \{ j : B_{i,j} \neq 0 \} |,
\end{equation}
then we assume \( d(B^\ast) \ll p \).

Moreover, we assume that we have access to interventional, {\it a.k.a.} experimental, data, which is modeled as follows, keeping in line with the definition from \cite{HytEbeHoy12}.
An experiment \( e \) is given by a partition
\begin{equation}
	\label{eq:iu}
	[p] = \cU_e \mathbin{\dot\cup} \cJ_e,
\end{equation}
with associated projection matrices
\begin{equation}
	\label{eq:iv}
	(U_e)_{i, j} = \left\{
	\begin{aligned}
		1, \quad & i = j \text{ and } i \in \cU_e\\
		0, \quad & \text{otherwise},
	\end{aligned}
	\right.
	\quad
	(J_e)_{i, j} = \left\{
		\begin{aligned}
			1, \quad & i = j \text{ and } i \in \cJ_e\\
			0, \quad & \text{otherwise}.
		\end{aligned}
	\right.
\end{equation}
In effect, all nodes in \( \cJ_e \) are intervened on, \ie, they are not influenced by their parents anymore.
We assume that they follow a standard Gaussian distribution $\cN(0,1)$, leading to a random variable \( X^e \sim \cN(0, \Sigma^{\ast, e}) \) corresponding to experiment \( e \) with covariance matrix
\begin{equation}
	\label{eq:iw}
	\Sigma^{\ast, e} = (I - U_e B^\ast)^{-1} (I - U_e B^\ast)^{-\top},
\end{equation}
and inverse covariance matrix (concentration matrix)
\begin{equation*}
	\Theta^{\ast, e} = (\Sigma^{\ast, e})^{-1} = (I - U_e B^\ast)^\top (I - U_e B^\ast).
\end{equation*}

\cite{HytEbeHoy12} provide the following criterion to identify \( B^\ast \) from interventional data associated with \( \cE \).
\begin{definition}
	[Completely separating system]
	\label{def:sep}
	The set of experiments \( \cE \) is a \emph{completey separating system} if for every \( i \neq j \in [p] \), there exists \( e \in \cE \) such that \( i \in \cJ_e \) and \( j \in \cU_e \).
\end{definition}
Note that \cite{HytEbeHoy12} call the separation condition for a pair \( (i, j) \in [p]^2 \) the \emph{pair condition}.
They show that Definition \ref{def:sep} guarantees identifiability of \( B^\ast \) from observational data.
Conversely, they show that if \( \cE \) is not separating, there exists a ground truth system that is not satisfied, albeit allowing a more general covariance structure on the error terms \( Z^e_k \) for the latter construction than we do.



We are now in a position to state our assumptions.

\begin{assumption}[Structure matrix]
	\label{assump:matrix}
	For any two positive integers $d \le p$ and \( \eta \in (0, 1/2] \), let $\matrixclass(p, d, \eta)$ denote the set of sparse  matrices defined by
	\begin{equation}
		\label{eq:it}
		\matrixclass(p, d, \eta) := \{ B \in \R^{p \times p} : B_{i,i} = 0 \text{ for } i \in [p], \, \| B \|_{\op} \le 1 - \eta, \, d(B) \le d \},
	\end{equation}
	and assume \( B^\ast \in \matrixclass(p, d, \eta) \).
\end{assumption}

\begin{assumption}[Interventions]
	\label{assump:experiments}
	Let \( \cE \) be a set of experiments with associated partitions \( \{ (\cU_e, \cJ_e) \}_{e \in \cE} \) and projection matrices \( \{ (U_e, J_e) \}_{e \in \cE} \) as in \eqref{eq:iu} and \eqref{eq:iv}, respectively.  Assume that \( \cE \) is \emph{separating} in the sense of Definition \ref{def:sep}.
\end{assumption}

\begin{assumption}[Noise]
	\label{assump:noise}
	Assume \( n \in \mathbb{N} \) is divisible by \( E := | \cE | \), set \( n_e = n/E \) for \( e \in \cE \), and for \( k \in [n_e], e \in \cE \), denote by \( Z_k^e \sim \cN(0, I) \) \iid Gaussian random vectors.
	Then, we assume that we have access to observations of the form $X_k^e = (I - U_e B^\ast)^{-1} Z_k^e$.
\end{assumption}

A few remarks are in order.

\ref{assump:matrix}. The bound \( \| B^\ast \|_{\op} \le 1 - \eta \) guarantees invertibility of \( I - U B^\ast \) for any projection matrix \( U \) and  stationarity of the process~\eqref{EQ:markovprocess}.

\ref{assump:experiments}. As mentioned, this is the same assumption under which \cite{HytEbeHoy12} show identifiability of \( B^\ast \) under more general assumptions than the ones presented here, in particular allowing more general noise variances and hidden variables.
Note that their proof of necessity of this assumption does not exactly match our assumption because our noise variances are restricted, so in principle, identifiability from observational data could be possible under a weaker condition.
However, we give evidence in Section \ref{sec:idfblty} that at least observational data alone is not sufficient to recover a general \( B^\ast \).

Intuitively, the fact that \( \cE \) is separating guarantees that \( B^\ast \) can be recovered from submatrices of \( \{ \Sigma^{\ast, e} \}_{e \in \cE} \) via solving a system of linear equations, a fact that is made more precise in Section \ref{sec:upper-bounds-llc}.
Since we are interested in recovering \( B^\ast \) under otherwise minimal assumptions on \( B^\ast \), this is the case we consider for the theoretical contributions of this paper.
We do however investigate the behavior of the two estimators considered in Section \ref{sec:results} with respect to a violation of this assumption numerically in Section \ref{sec:numerics}.

\ref{assump:noise}. The assumption of Gaussian noise is not critical for our analysis, and in fact all our proofs extend readily to sub-Gaussian noise.
Similarly, the assumption \( n_e = n/E \) can be replaced by \( n_e \asymp n/E \), that is, the number of observations in all experiments is comparable.
On the other hand, the assumption that \( \E[Z_k^e] = 0 \), \( \E[(Z_k^e)^2] = 1 \) might be restrictive in practice.
We conjecture that it might be relaxed while maintaining many of the guarantees we give in Section \ref{sec:results}, but due to the notational burden associated with incorporating these additional factors into the estimation, we chose to leave this topic as the subject of future research.
Note that while the assumption of equal variances implies identifiability from observational data in the case where \( B^\ast \) is assumed to be acyclic \cite{LohBuh14, PetBuh14}, it does not in the cyclic case, see Section \ref{sec:idfblty}.
Hence, the assumptions as presented still lead to a class rich enough to require controlled experiments to estimate \( B^\ast \).

\begin{remark}
	\label{rem:opt-system}
	It was shown in \cite{HytEbeHoy13} that the minimum number of experiments necessary to obtain a completely separating system is of the order \( \log(p) \), which can be seen by a simple binary coding argument.
	Hence, if we are able to pick the experiments in the most parsimonious way possible, \( E = O(\log(p)) \) only contributes a logarithmic factor to any of the rates presented in Section \ref{sec:results}.
\end{remark}


\section{Main results}
\label{sec:results}

\subsection{Lower bounds}

First, we give lower bounds for the estimation of matrices \( B^\ast \in \matrixclass(p, d, \eta) \). To that end, let $\kappa$ denote the \emph{redundancy} of the experiments $\cE$. It is defined as the maximum number of experiments that separate two variables,
\begin{equation}
	\label{eq:dj}
	\kappa = \kappa(\cE) = \max_{i \neq j \in [p]} |\{e \in \cE : i \in \cU_e, j \in \cJ_e \}|.
\end{equation}

\begin{theorem}
	\label{THM:LOWERBOUNDS}
	There exists a constant \( c > 0 \) such that if \( d \leq p/4 \) and
	\begin{equation}
		\label{eq:cr}
		n \geq p d E^2 \log\left(1 + \frac{p}{4d}\right),
	\end{equation}
	then,
	\begin{equation}
		\label{eq:gt}
		\inf_{\hB} \sup_{B^\ast \in \matrixclass(p, d, \eta)} \p \big( \| \hB - B^\ast \|_F^2 \gtrsim \frac{pdE}{\kappa n} \log\big(1 + \frac{p}{4d} \big) \big) \ge c,
	\end{equation}
	where the infimum is taken over all measurable functions of the data \( \{ X_k^e \}_{e \in \cE, \, k \in [n_e]} \).
\end{theorem}

The proof of Theorem \ref{THM:LOWERBOUNDS} is deferred to Section \ref{sec:lower-bound-proof}.
We remark that there is a mismatch in the lower bound and the range of \( n \) for which it is effective that is of order \( E \).
In the case of a minimal system of completely separating interventions, by Remark \ref{rem:opt-system}, this mismatch is of order \( \log(p) \).

\subsection{Upper bounds for the LLC estimator}
\label{sec:upper-bounds-llc}

Next, we give bounds on the performance of the LLC estimator introduced in \cite{HytEbeHoy12}.
We briefly summarize the algorithm below, which can be seen as a moment estimator for \( B^\ast \).

\subsubsection{The LLC estimator}
Denote by \( b^\ast_i \in \R^{p-1}\) the \( i \)th row of \( B^\ast \), where we omit the \( i \)th entry, which is assumed to be zero since \( B^\ast \in \matrixclass \).
Formally, $b^\ast_i = (P_i B_{i, :}^\top) = B_{i, -i}^\top$, 
where \( P_i \colon \R^p \to \R^{p-1} \) denotes the projection operator that omits the \( i \)th coordinate.

LLC is motivated by the observation that on the population level, each \( b^\ast_i \) satisfies a linear system \( T^\ast_i b^\ast_i = t^\ast_i \), where \( T^\ast_i \in \R^{m_i\times (p-1)}\) and \( t^\ast_i \in \R^{m_i}\) for some $m_i\ge 1$ are defined as follows.
For \( i = 1, \dots, p \), define the matrix \( T^\ast_i \) and the column vector $t^\ast_i$ row by row.
For each experiment \( e \) such that \( i \in \cU_e \) and each \( j \in \cJ_e \), add a row to \( T^\ast_i \) and to $t^\ast_i$, say with index \( \ell =\ell(e,j)\), that is of the form
\begin{equation}
	\label{eq:eq}
	(T^\ast_i)_{\ell, :} =
	\mathfrak{e}_j^\top \Sigma^{\ast, e} P_i^\top \,,\qquad (t^\ast_i)_\ell = \Sigma^{\ast, e}_{j, i}
\end{equation}
where $\mathfrak{e}_j$ is the $j$th canonical vector of $\R^p$. To better visualize $(T^\ast_i)_{\ell, :}$, one may  rearrange the indices so that \( \cJ_e = \{1, \dots, | \cJ_e |\} \), in which case we have
$$
(T^\ast_i)_{\ell, :}=
	\begin{bmatrix}
		0 & \dots & 1 & \dots & 0 & \Sigma^{\ast, e}_{j, \cU_e \setminus \{ i \}}
	\end{bmatrix},
$$
where ``$1$" appears in the $j$th coordinate.
Let $m_i$ denote the total number of such rows obtained by scanning through all experiments $e$ such that \( i \in \cU_e \) and $j$ such that  \( j \in \cJ_e \).

When \( \cE \) is a completely separating system, \( T^\ast_i b_i = t^\ast_i \) has the unique solution \( b^\ast_i = (B^\ast_{i, -i})^\top \), \cite{HytEbeHoy12}.
The LLC estimator is obtained by substituting \( \Sigma^{\ast, e} \) in the above definitions with its empirical counterpart $\hat \Sigma^e$ defined by
\begin{equation}
	\label{eq:ja}
	\hat \Sigma^e = \frac{1}{n_e} \sum_{k=1}^{n_e} X_k^e (X_k^e)^\top,
\end{equation}
except for where the variances are known exactly due to the fact that an intervention is performed. This leads to a linear system of the form $\hat T_i b_i = \hat t_i$.  Rather than solving the linear system exactly, the LLC estimator is obtained by minimizing a penalized least squares problem to promote sparsity in the resulting estimate:
\begin{equation}
	\label{eq:es}
	\hb_i = \argmin_{b \in \R^{p-1}} \| \hT_i b - \hatt_i \|_2^2 + \lambda \| b \|_1, \quad i = 1, \dots, p,
\end{equation}
where $\lambda>0$ is a tuning parameter. The solutions to the above problems are assembled into the LLC estimator \( \hB_{\llc} \) by setting
\begin{equation}
	\label{eq:gb}
	(\hB_{\llc})_{i, -i} = \hb_i^\top, \quad (\hB_{\llc})_{i, i} = 0, \quad i \in [p].
\end{equation}

\subsubsection{Statistical performance}

The upper bounds we give for the performance of LLC depend on additional constants that are not directly controlled for an arbitrary \( B^\ast \in \matrixclass(p, d, \constop) \).
Loosely speaking, they pertain to the conditioning of the \( \ell^1 \)-regularized least squares problems that are solved to obtain \( \hB_{\llc} \).
These constants are defined as follows.
Denote by
\begin{equation}
	\label{eq:jj}
	\cC(d) := \{ v \in \R^p : \text{for all } S \subseteq [p] \text{ with } |S| \le d, \| v_{S^c} \|_1 \le 3 \| v_{S} \|_1 \}.
\end{equation}
Then, define
\begin{align*}
	\constre(d) = {} & \min_{i \in [p]} \inf_{v \in \cC(d), v \neq 0} \frac{\| T_i^\ast v \|_2}{\| v \|_2},\quad
	\constremax(d) = {} \max_{i \in [p]} \sup_{\substack{v \in \R^p, v \neq 0,\\ |\supp(v)| \le d}} \frac{\| T_i^\ast v \|_2}{\| v \|_2}, \quad
	\constcollone = 
	\max_{i \in [p]} \max_{j \in [p]} \sum_{k \in [p]} | (T^\ast_i)_{k, j} |.
\end{align*}
%
We are now in a position to state the first rate of convergence for the LLC estimator.

\begin{theorem}[Rates for LLC estimator]
	\label{THM:LLC-RATES}
	Let   assumptions \ref{assump:matrix} -- \ref{assump:noise} hold and fix $\delta \in (0,1)$. Assume further that
	\begin{align}
		n \gtrsim {} & \left( 1 \vee \frac{p^2}{\constcollone^2 \eta^4} \vee \frac{pd}{(\constremax(d)+1)^2 \eta^4 \constre(d)^4} \right)E \log(e \kappa p / \delta), \label{eq:jk}
	\end{align}	
	Then LLC estimator \( \hB_{\llc} \) defined in \eqref{eq:gb} with $\lambda$ chosen such that
	\begin{align}
		\lambda \asymp {} & \constcollone \sqrt{\frac{E \log(e \kappa p/\delta)}{n}},
	\end{align}
satisfies 
	\begin{align}
		\label{eq:jl}
		\| \hB_\llc - B^\ast \|_F^2
		\lesssim {} & \frac{\constcollone^2}{\constre(d)^4 \eta^4} \frac{p d E \log (e \kappa p/\delta)}{n}\,,
	\end{align}
	with probability at least $1-\delta$.
\end{theorem}


The proof is deferred to Section \ref{sec:llc-proof}.
It uses standard arguments for the LASSO, together with perturbation results for regression with noisy design from \cite{LohWai11a} in Lemma \ref{lem:matrix-dev-relax} to handle the presence of noise in the matrices \( \hT_i \).


\begin{remark}
	\label{rem:dimension-dep}
	Unfortunately, it is not clear whether the factors \( \constre(d), \constremax(d), \constcollone \) stay bounded with increasing \( p, \, d, \) and \( E \), uniformly over all possible ground truth matrices \( B^\ast \in \matrixclass(p, d, \eta) \).
	Hence, even though the explicit dependence on \( p, \, d, \) and \( E \) in the upper bounds \eqref{eq:jl} matches the lower bounds \eqref{eq:gt}, we can not claim this rate to be (near) minimax optimal.





\end{remark}

\begin{remark}
	Comparing the definitions of \( \constre(d) \) and \( \constremax(d) \), one might prefer an alternative definition of the former of the form
	\begin{equation}
		\label{eq:je}
		\tilde \constre(d) := \min_{i \in [p]} \inf_{\substack{v \in \R^p, \, v \neq 0,\\|\supp v| \le d}} \frac{\| T_i^\ast v \|_2}{\| v \|_2}.
	\end{equation}
	In fact, these two quantities are related, albeit for different values of \( d \), see \cite[Section 8]{BelLecTsy18}.
	We choose \( \constre(d) \) instead of \( \tilde \constre(d) \) for the sake of a simpler presentation.
\end{remark}

In order to address the issues raised in the previous remark, we next give a penalized maximum likelihood estimator.

\subsection{Upper bounds for two-step penalized likelihood}
\label{sec:mle}

\subsubsection{Two-step maximum likelihood estimator}

One shortcoming in the rate for LLC for large~\( n \) in Theorem \ref{THM:LLC-RATES} are the constants \( \constre(d) \) and \( \constcollone \) which might actually grow with \( p \), see Remark \ref{rem:dimension-dep}.
Moreover, as a moment estimator, it does not naturally behave well with respect to model misspecification.
This motivates a different estimator based on a penalized maximum likelihood approach.

Recall that the negative log-likelihood of a multivariate Gaussian with empirical covariance matrix \( \hat \Sigma \) and precision matrix $\Theta$ is given by.
\begin{equation*}
	\ell(\Theta, \hat \Sigma) = \tr(\hat \Sigma \Theta) - \log \det (\Theta)
\end{equation*}
Thus, the negative log-likelihood for the whole model is proportional to
\begin{align}
	\label{eq:he}
	\mathcal{L}(B) = \cL(B, \hat\Sigma^{1}, \dots, \hat\Sigma^{E}) = \sum_{e \in \cE} \ell(\Theta^e(B), \hat\Sigma^e),
\end{align}
where $\Theta^e(B) = (I - U_e B)^\top (I - U_e B),$
and
\begin{equation*}
	\hat \Sigma^e = \frac{1}{n_e} \sum_{k = 1}^{n_e} X^e_{k} (X^e_{k})^\top
	= \frac{E}{n} \sum_{k = 1}^{n_e} X^e_{k} (X^e_{k})^\top.
\end{equation*}

In order to exploit sparsity in the underlying matrix \( B^\ast \), we need to penalize $\cL(B)$ before minimizing it.
However, due to the non-linear dependence of \( \Sigma^e \) on \( B \), a vanilla \( \ell_1 \)-penalization term might be acting at the wrong scale globally.
To overcome this limitation, we propose a two-step estimation procedure. First an initial guess \( \hBi \) is produced using a penalization acting on the scale of the concentration matrices. This initial guess is subsequently refined to \( \hB \) as the solution to the \( \ell_1 \)-penalized log-likelihood restricted to a small ball around \( \hBi \).

In the first step, we employ penalization with a term resembling a graphical lasso penalty for each experiment,
\begin{equation*}
	\pen_\mathrm{init}(B) = \pen_{\mathrm{init}, \lambda_\mathrm{init}}(B) = \lambda_\mathrm{init} \sum_{e \in \cE} \| \Theta^e(B) \|_1,
\end{equation*}
leading to the penalized log-likelihood
\begin{equation}
	\label{eq:eb}
	\cT_\mathrm{init}(B) = \cT_{\mathrm{init}, \lambda_\mathrm{init}}(B, \hat \Sigma^1, \dots, \hat \Sigma^{E}) = \cL(B, \hat \Sigma^1, \dots, \hat \Sigma^E) + \pen_{\mathrm{init}, \lambda_\mathrm{init}}(B).
\end{equation}
The initialization estimator is then given by
\begin{equation}
	\label{eq:dk}
	\hBi \in \argmin_{B \in \cB} \cT_{\mathrm{init}}(B).
\end{equation}
Note that this is not a convex optimization problem and it is hard to solve in general.
However, we do give a local optimization algorithm in Section \ref{sec:numerics} that attempts to find a local minimum for \eqref{eq:eb}.

In the second step, this estimator is refined by employing a different regularization term,
\begin{equation*}
	\pen_{\mathrm{loc}}(B) = \pen_{\mathrm{loc}, \lambda_{\mathrm{loc}}}(B) = \lambda_{\mathrm{loc}} \| B \|_1,
\end{equation*}
\begin{equation}
	\label{eq:ec}
	\cT_\mathrm{loc}(B) = \cT_{\mathrm{loc}, \lambda_\mathrm{loc}}(B, \hat \Sigma^1, \dots, \hat \Sigma^{E}) = \cL(B, \hat \Sigma^1, \dots, \hat \Sigma^E) + \pen_{\mathrm{loc}, \lambda_{\mathrm{loc}}}(B),
\end{equation}
and the estimator is given by
\begin{equation}
	\label{eq:dl}
	\hB_\loc \in \argmin_{\substack{B \in \cB\\ \| B - \hBi \|_F \leq \localradius}} \cT_{\loc}(B),
\end{equation}
with a suitably chosen localization parameter \( \localradius > 0 \).

The loss function \eqref{eq:ec} is again non-convex and hence hard to optimize, but local optimization algorithms seem to produce good results, see Section \ref{sec:numerics}.

\subsubsection{Statistical performance}
Assuming we have access to the global minima \( \hB_\init \) and \( \hB_\loc \), we show the following rates for \( \hB_\loc \):

\begin{theorem}
	\label{THM:2-STEP-RATES}
	Under assumptions \ref{assump:matrix} -- \ref{assump:noise}, if
	\begin{align}
		\label{eq:jn}
		n \gtrsim {} & \left( E^2 \vee \frac{1}{\constop^4} \vee p^2 \right) \frac{p^2 (d+1)^2 E^3}{\constop^4} \log(e p E/\delta)
	\end{align}
	and the parameters for the estimators \( \hB_\init \) and \( \hB_\loc \) are chosen such that
	\begin{align}
		\localradius \asymp {} & \frac{1}{\sqrt{E}} \wedge \constop \wedge \frac{1}{\sqrt{p}}, \quad
		\lambda_\init \asymp  \sqrt{\frac{E \log(e p E/\delta)}{n}}, \quad
		\text{ and } \quad \lambda_\loc \asymp \sqrt{\frac{E^2 \log(e p E/\delta)}{n}}
	\end{align}
	then
	\begin{align}
		\label{eq:jo}
		\| \hB_\loc - B^\ast \|_F^2 \lesssim \frac{p (d + 1) E^2}{\eta^8 \,  n} \log(pE/\delta),
	\end{align}
	with probability at least \( 1 - \delta \).
\end{theorem}
The proof is deferred to Section \ref{sec:2-step-rates-proof}.
It is based on the one hand on convexity properties of the Gaussian log-likelihood function that were developed in the context of convex optimization problems for estimation of sparse concentration matrices in \cite{RotBicLev08} and \cite{LohWai13}, and on the other to new structural results on the difference \( \Theta^e(B) - \Theta^{\ast, e} \) between concentration matrices expressed in terms of  \( B - B^\ast \); see Lemma \ref{lem:relate-h}.


Note that the upper bound \eqref{eq:jo} is worse by a factor of \( E \) and a log factor than the lower bound \eqref{eq:gt} in Theorem \ref{THM:LOWERBOUNDS}.
However, the completely separating system \( \cE \) can be chosen to be as small as \( E \asymp \log(p) \), see \cite{HytEbeHoy13} and Remark \ref{rem:opt-system}, in which case this eventual rate is almost minimax optimal up to logarithmic terms.

\section{Numerical experiments}
\label{sec:numerics}


Recall that $\ell(\Theta, \hat \Sigma) = \tr(\hat \Sigma \Theta) - \log \det (\Theta)$ and that we want to find solutions to the two regularized maximum likelihood problems,
\begin{align}
	\label{eq:dn}
	\hB_\init \in {} & \argmin_{B \in \cB} \left\{ \sum_{e \in \cE} \ell(\Theta^e(B), \hat \Sigma^e) + \lambda_\init\sum_{e \in \cE} \| \Theta^e(B) \|_1 \right\}\\
	\hB_\loc \in {} & \argmin_{\substack{B \in \cB\\ \| B - \hBi \|_F \leq \localradius}} \left\{ \sum_{e \in \cE} \ell(\Theta^e(B), \hat \Sigma^e) + \lambda_\loc \| B \|_1 \right\}
	\label{eq:do}
\end{align}

Both problems are non-convex and there is no obvious strategy for how to find global minima.
However, since they are continuous, we can empirically study the performance of optimization algorithms designed for convex problems, hoping to obtain at least local minima.
In the following, we describe how candidate solutions for both \eqref{eq:dn} and \eqref{eq:do} can be found efficiently and demonstrate their performance based on experiments with synthetic data.
Additionally, we give a low-rank update approach in Appendix \ref{sec:low-rank-update} that can be used to speed up calculations when the number of experiments \( E \) is large, but for each experiment, the number of controlled variables \( |\cJ_e| \) is small.

\subsection{Solving the initialization problem by non-convex ADMM}

The difficulty in solving problem \eqref{eq:dn} is to handle the non-smooth penalty terms of non-linear transformations of \( B \), \( \| \Theta^e(B) \|_1 \).
We use a non-linear version of the Alternating Direction Method of Multipliers (ADMM) algorithm, which allows us to introduce additional variables \( \Theta^e \), constrain them to fulfill \( \Theta^e = \Theta^e(B) \), and keep the resulting dimensionality blowup manageable.

The ADMM algorithm \cite{GabMer76, GloMar75, BoyParChu11} is a splitting algorithm intended to solve convex optimization problems of the form
\begin{equation}
	\label{eq:dq}
	\begin{aligned}
		\min {} & f(x) + g(y)\\
		\text{s. t. } {} & F x + G y = b,
	\end{aligned}
\end{equation}
where \( x \in \R^m, y \in \R^\ell \), \( f \) and \( g \) are convex functions on \( \R^m \) and \( \R^\ell \), respectively, and \( F \in \R^{m \times k} \), \( G \in \R^{\ell \times k} \), \( b \in \R^k \).
Introducing the dual variable \( u \), a step size \( \rho > 0 \), and starting with an initialization \( x^0, y^0, u^0 \), its iterations are given by
\begin{align*}
	 \ x^{k + 1} = {} & \argmin_x f(x) + \frac{\rho}{2} \| Fx + G y^k - b + u^k \|_2^2\\
	 \ y^{k + 1} = {} & \argmin_{y} g(y) + \frac{\rho}{2} \| Fx^{k+1} + G y - b + u^k \|_2^2 \\
	 \ u^{k + 1} = {} & u^k + F x^{k + 1} + G y^{k + 1} - b,
\end{align*}
which is the so called \emph{scaled form} of ADMM.

Note that while in the case of convex objective functions and linear constraints, there are well-established convergence results for ADMM, \cite{Gab83, EckBer92}, results about convergence to a stationary point for non-convex variants are scarce, requiring either linear constraints \cite{WanYinZen15} or further modifications and additional assumptions \cite{BenKnoSch15}.

In order to apply a non-convex ADMM variant, we rewrite problem \eqref{eq:dn} as
\begin{align}
	\label{eq:dp}
	\min_{B \in \cB} {} & \sum_{e \in \cE} \left( \ell(\Theta^e, \hat \Sigma^e) + \lambda_\init \| \Theta^e \|_1 \right) \\
	\text{s. t. } {} & \Theta^e = (I - U_e B)^\top (I- U_e B) \quad \text{for } e \in \cE.
\end{align}

Then, introducing dual variables \( \Lambda^{e} \in \R^{p \times p }\), \( e = 1, \dots, E \), the outer iteration of our algorithm is given by
\begin{align}
	\Theta^{e, k+1} = {} & \argmin_{\Theta^e} \tr(\hat \Sigma^e \Theta^e) - \log \det \Theta^e + \lambda_\init \| \Theta^e \|_1 \nonumber\\
	{} & \qquad \qquad+ \frac{\rho}{2} \| \Theta^e - (I - U_e B^k)^\top (I - U_e B^k) + \Lambda^{e, k} \|_F^2,
	\quad (e = 1, \dots, E) \label{eq:dr}\\
	\ B^{k+1} = {} & \argmin_B \sum_{e} \| \Theta^{e, k+1} - (I - U_e B)^\top (I - U_e B) + \Lambda^{e, k} \|_F^2 \label{eq:ds}\\
	\ \Lambda^{e, k + 1} = {} & \Lambda^{e, k} + \Theta^{e, k+1} - (I - U_e B^{k + 1})^\top (I - U_e B^{k + 1}),
	\quad (e = 1, \dots, E).
\end{align}

Note that \eqref{eq:dr} is a convex problem, resembling the graphical LASSO \cite{FriHasTib08} or SPICE \cite{RotBicLev08} but with an additional quadratic penalty term.
We can solve these subproblems with an extension of the QUIC algorithm \cite{HsiDhiRav11} that employs coordinate descent to iteratively find Newton directions.

Problem \eqref{eq:ds} on the other hand is a non-convex problem, albeit without constraints.
Hence, we can use any local optimization algorithm.
For our experiments, we choose L-BFGS \cite{LiuNoc89} to perform this approximate minimization, yielding a stationary point of the objective function.

In order to find a suitable step size parameter \( \rho \), we allow varying \( \rho^k \) and employ the dual-balancing strategy from \cite{HeYanWan00, WanLia01}.


\subsection{Solving local problem by Augmented Lagrangian Method}

In order to find a local minimum of \eqref{eq:do}, we employ the Augmented Lagrangian Method \cite{NocWri06} that transforms the inequality constraint into a box constraint and iteratively solves for the associated dual variable.
It leads to the following iteration, where \( u \) is a slack variable for the \( \ell_2 \) constraint and \( \lambda \) is the associated dual variable.
\begin{align}
	B^{k + 1} = {} & \argmin_{B, \, u \leq \localradius^2} \sum_{e \in \cE} \ell(\Theta^e(B), \hat \Sigma^e) + \lambda_\loc \| B \|_1 + \frac{\rho}{2} \left( u^{k} - \| B - \hB_1 \|_2^2 + \frac{\lambda^k}{\rho} \right)^2 \label{eq:dv}\\
	u^{k+1} = {} & u^k + \rho (u - \| B^{k+1} - \hB_1 \|_2^2)
\end{align}
To solve \eqref{eq:dv}, we use L-BFGS-B \cite{ZhuByrLu97}, transforming the \( \ell^1 \)-regularization into a linear term with additional non-negativity constraints,
\begin{align}
	\label{eq:dw}
	B = B_+ - B_-, \quad B_+ \geq 0, \quad B_- \geq 0, \quad \| B \|_1 = \sum_{i, j}  ((B_+)_{ij} + (B_-)_{ij}).
\end{align}

\begin{figure}[t]
	\centering
		\includegraphics[width=\textwidth]{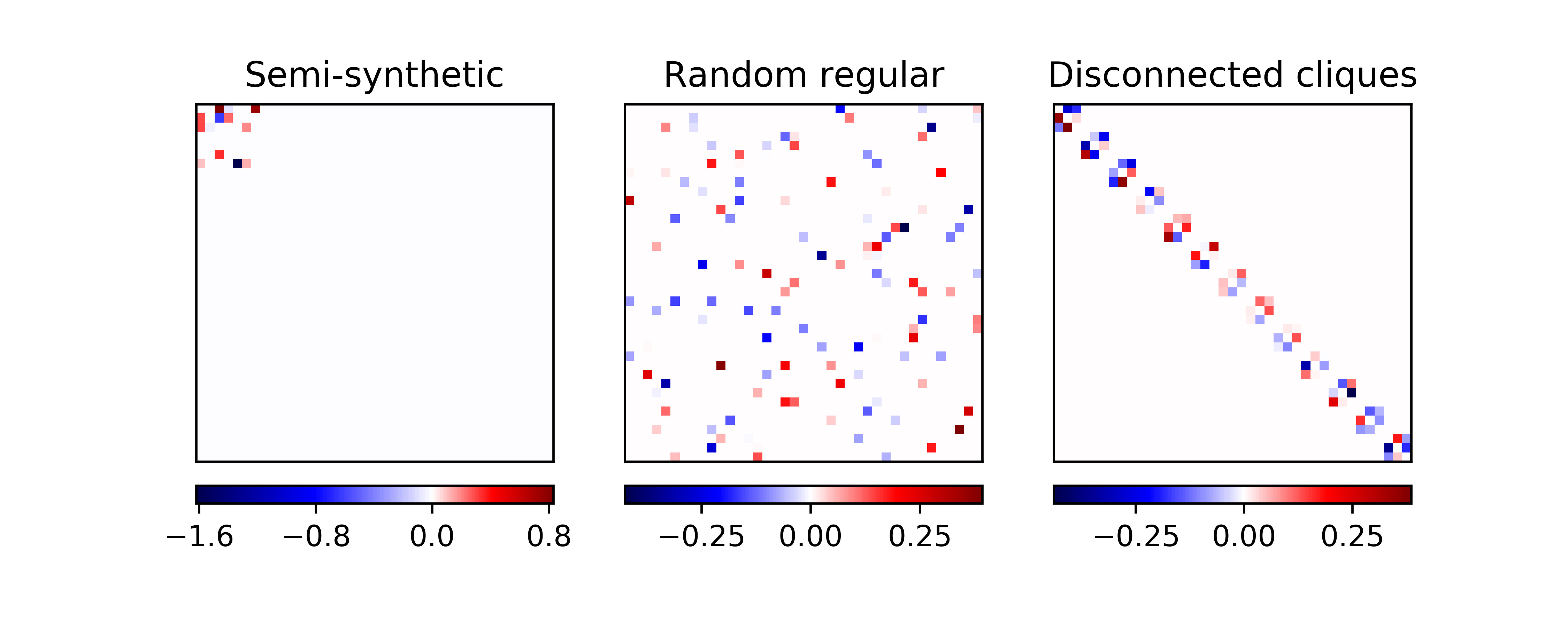}
		\caption{Heatmaps visualizing example matrices \( B^\ast \) for the three studied models. In all examples, \( p = 39 \) and \( d = 3 \). From left to right: Semi-synthetic data from \cite{CaiBazGia13}, Random regular graphs, and disconnected cliques.}
		\label{fig:heatmaps}
\end{figure}

\begin{figure}[t]
	\centering
	\begin{subfigure}[t]{0.45\textwidth}
		\includegraphics[width=\textwidth]{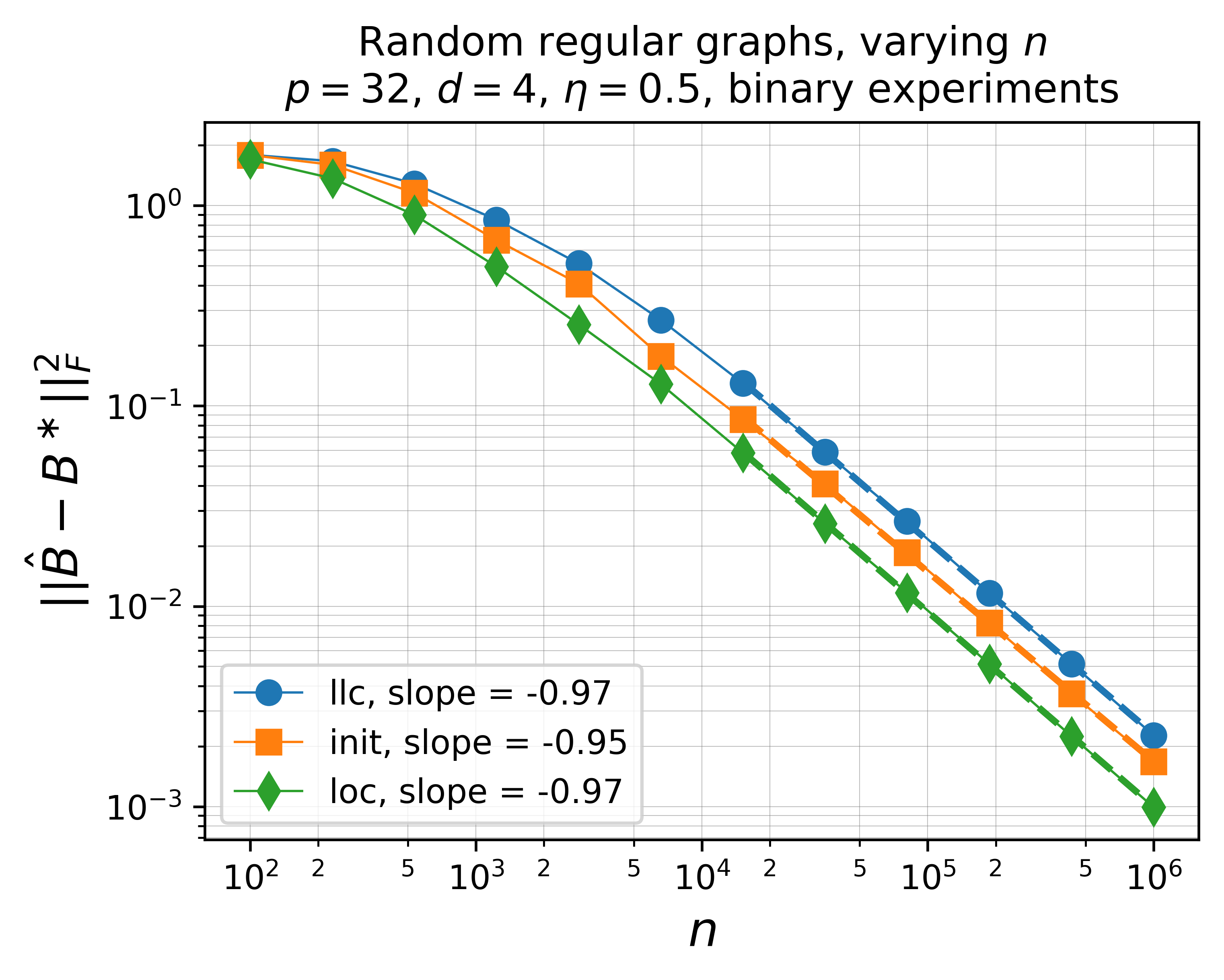}
		\caption{Varying \( n \)}
		\label{fig:randvarn}
	\end{subfigure}
	\hspace{1em}
	\begin{subfigure}[t]{0.45\textwidth}
		\includegraphics[width=\textwidth]{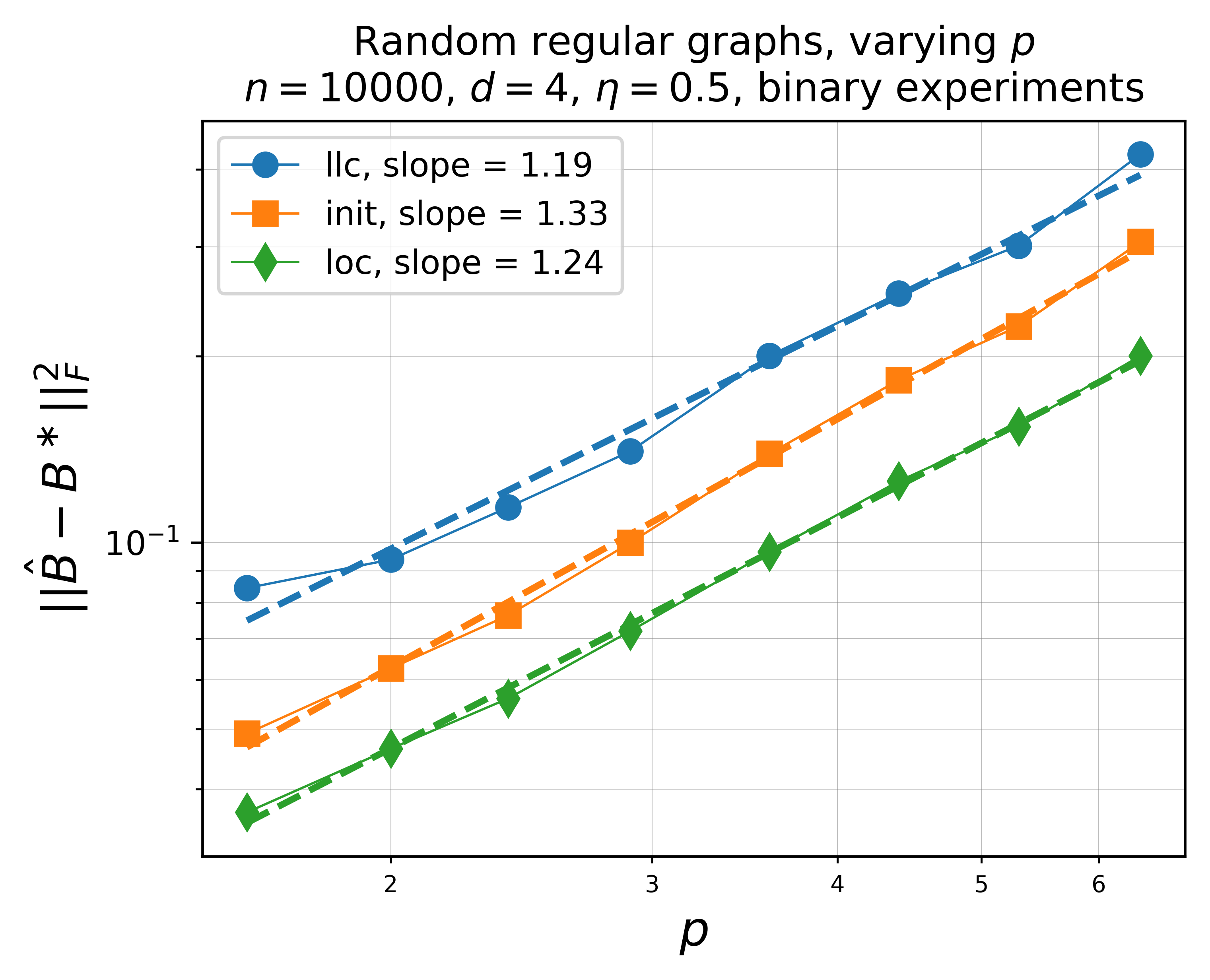}
		\caption{Varying \( p \)}
		\label{fig:randvarp}
	\end{subfigure}\\
	\vspace{1em}
	\begin{subfigure}[t]{0.45\textwidth}
		\includegraphics[width=\textwidth]{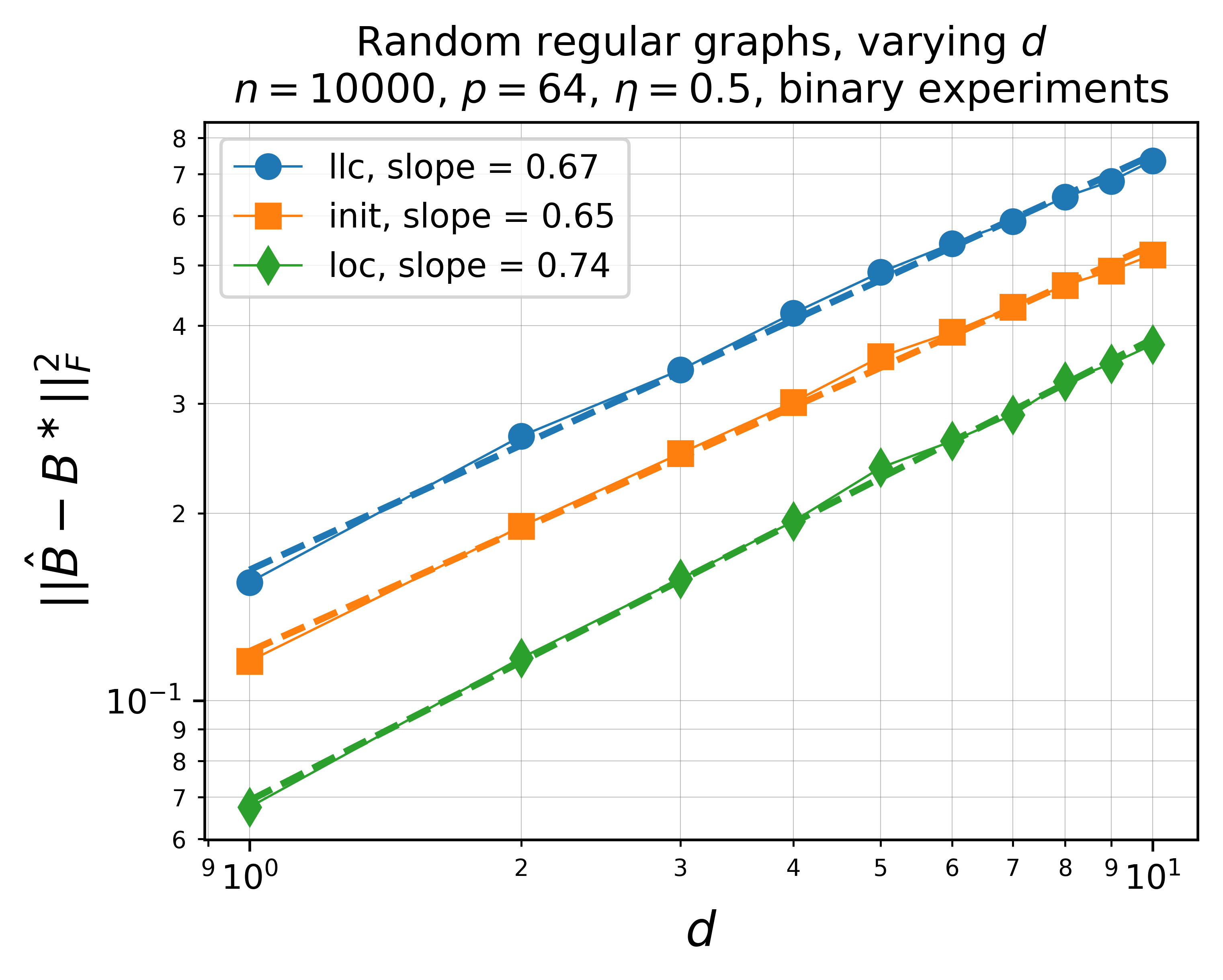}
		\caption{Varying \( d \)}
		\label{fig:randvard}
	\end{subfigure}
	\hspace{1em}
	\begin{subfigure}[t]{0.45\textwidth}
		\includegraphics[width=\textwidth]{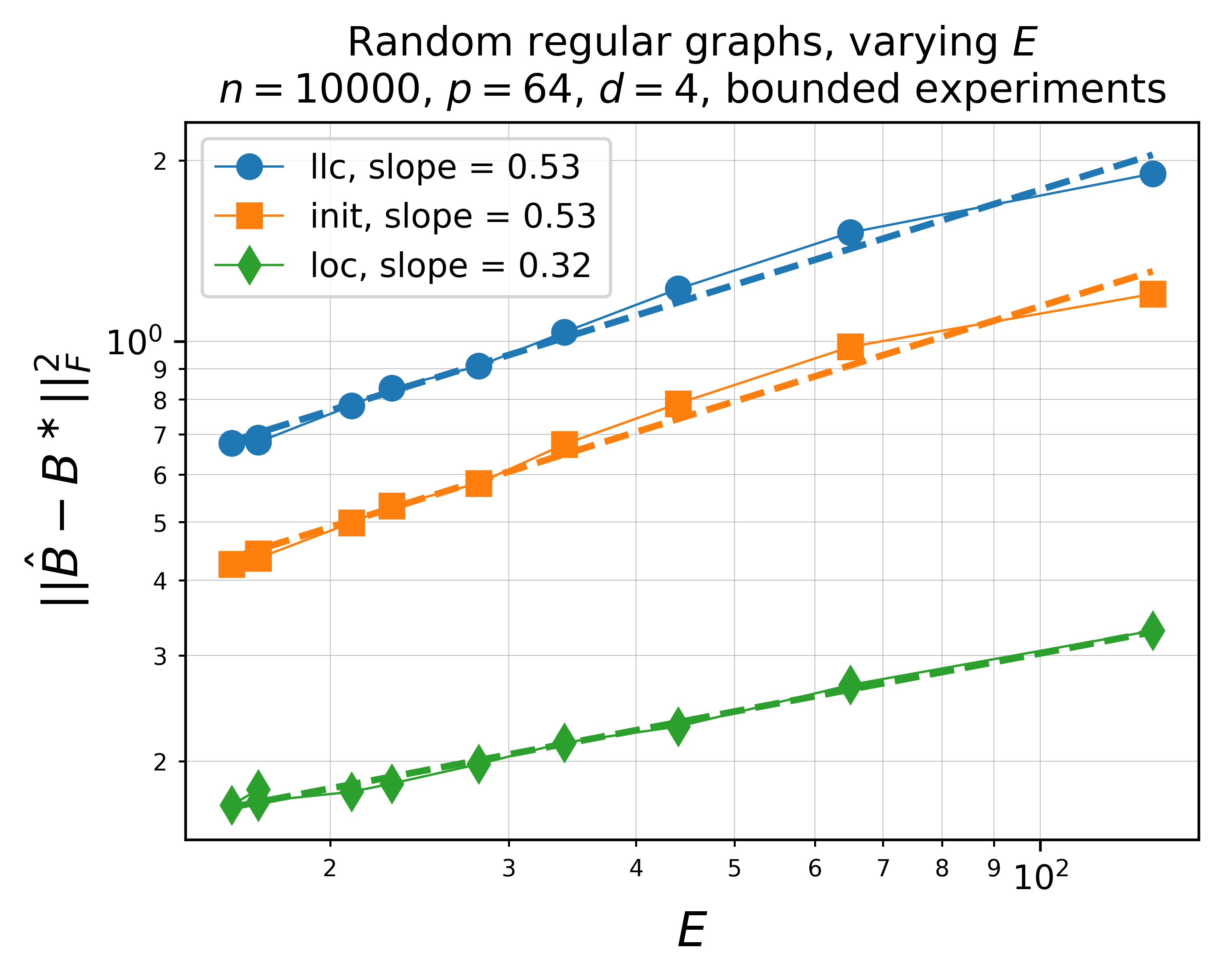}
		\caption{Varying \( E \)}
		\label{fig:randvare}
	\end{subfigure}
	\caption{Experiments for random regular graphs, varying one parameter while keeping the other ones fixed. ``llc'' refers to \( \hB_{\llc} \), ``init'' to \( \hB_{\init} \), ``loc'' to \( \hB_{\loc} \).}
	\label{fig:rand}
\end{figure}

\subsection{Experimental setup}
\label{sec:numerics-setup}

We perform experiments with synthetic and semi-synthetic data to gauge the performance of the maximum likelihood procedure, comparing it to the LLC algorithm \cite{HytEbeHoy12}.

For the synthetic benchmarks, we study two types of graph structures: (directed) random regular graphs and graphs composed of disconnected cliques.
For the semi-synthetic benchmarks, we use a gene-regulatory network from \cite{CaiBazGia13} consisting of 39 genes.
Note that in \cite{CaiBazGia13}, the authors employ a model very similar to ours, but instead of allowing controlled experiments on certain nodes, they consider so-called expression quantitative trait loci (eQTL) as proxies for interventions, which changes their model compared to the one considered here.
Nonetheless, part of the output of their estimator is a linear causal network, which is what we consider as ground truth to simulate data following the Gaussian model introduced in Section \ref{sec:model}.
Example ground truth matrices for the two random models and the semi-synthetic matrix from \cite{CaiBazGia13} are given in Figure \ref{fig:heatmaps}.
There, we set \( p = 39 \) and \( d = 3 \) for the random models to coincide with \( p \) and \( d \) in the semi-synthetic case.

In the following two sections, we give more details about data generation and parameter tuning.

\subsubsection{Models}

\paragraph{Synthetic graphs:} The ground truth graphs are generated by first obtaining the (directed) adjacency matrix \( B_\mathrm{adj} \in \{0, 1\}^{p \times p} \), a matrix \( B_{\mathrm{val}} \in \R^{p \times p} \) containing edge values, and finally setting \( B^\ast \) to be the Hadamard product of the two, normalized to have operator norm \( 1 - \eta =0.5\),
\begin{align}
	\label{eq:dy}
	\tilde B = {} & B_{\mathrm{adj}} \odot B_{\mathrm{val}}, \quad 
	B^\ast = {}  \frac{(1 - \eta)}{\| \tilde B \|_{\op}} \tilde B.
\end{align}
Here, \( B_{\mathrm{val}} \) consists of independent standard Gaussian entries, and \( B_{\mathrm{adj}} \) is the adjacency matrix of either a regular random graph or one composed of disconnected cliques.

\paragraph{Random regular graphs:}

\( \supp ((B_{\mathrm{adj}})_{i,:}) \) is constructed by sampling \( d \) times uniformly at random without replacement from \( \{1, \dots, p\} \setminus \{i\} \) and all elements in the support are assigned \( 1 \).



\paragraph{Disconnected cliques:}
\( B_{\mathrm{adj}} \) is the adjacency matrix of a graph consisting of \( \lfloor p/d \rfloor \) disconnected \( d \)-cliques and an additional disconnected \( p - d \lfloor p/d \rfloor \) clique if \( d \) does not divide \( p \). This model is meant to illustrated clustered variables that operate in modules, akin to the ones arising in gene regulatory networks.

\subsubsection{Tuning parameters}

\paragraph{Choice of $\lambda$:}
To keep the comparison simple, we use an oracle choice of \( \lambda_\init \), \( \lambda_\loc \) and \( \localradius \).
For the first two, this means choosing them such that \( \| \hB_\init - B^\ast \|_F \) and \( \| \hB_\loc - B^\ast \|_F \) is minimal.
For \( \localradius \), we choose \( \localradius = 2 \| \hB_\init - B^\ast \|_F \).
In practice, both parameters could be chosen by cross-validation.


\paragraph{Initialization of optimization algorithm:}
We initialize the calculation of \( \hB_\init \) for the largest value of \( \lambda_\init \) with the all zeros matrix and then warmstart the calculation with the output of the calculation for the next larger value of \( \lambda_\init \).
The calculation of \( \hB_\loc \) is initialized with the output of \( \hB_\init \).
To investigate the dependence on the initialization, we also try initializing the calculation of \( \hB_\init \) with a strict triangular matrix whole upper elements consist of independent \( \cN(0, 10) \) random variables, as well as running the likelihood optimization without constraints as described previously with the same initialization.

\paragraph{Systems of interventions:}
We consider three choices for the experiments $\cE$.
The first one, which we call \emph{binary}, consists of separating the nodes with a bisection approach similar to the construction given in \cite{Dic69} that leads to \( E = O(\log p) \).
The second one, which we call \emph{bounded}, is given by \cite{Mao84} and produces experiments whose sizes \( | \cJ_e | \) are bounded by \( k \).
In this case, \( E = O(n/k) \).
The third kind corresponds to \( k = 1 \), taking \( \cJ_i = \{i\} \) for \( i \in [p] \), which we call \emph{single-node experiments}.

\newpage
\paragraph{Repetitions:}
All errors are averaged over 32 random repetitions of sampling \( B^\ast \) and the observations \( X^e_k \).

\subsection{Results}

\subsubsection{Performance}
\begin{figure}[t]
	\begin{subfigure}[t]{0.45\textwidth}
		\includegraphics[width=\textwidth]{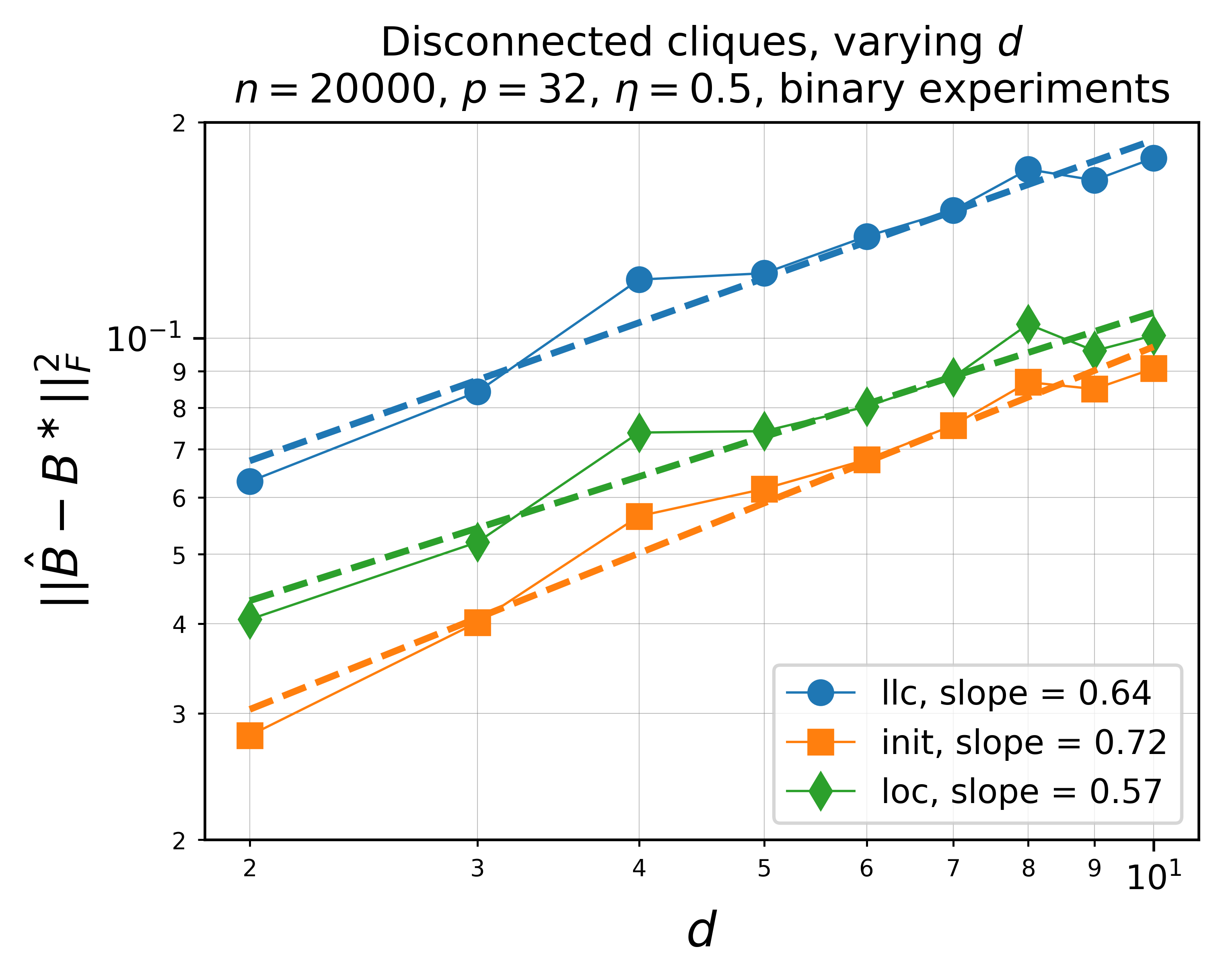}
		\caption{Varying \( d \) for disconnected cliques graph.}
		\label{fig:clustersvard}
	\end{subfigure}
	\hspace{1em}
	\begin{subfigure}[t]{0.45\textwidth}
		\includegraphics[width=\textwidth]{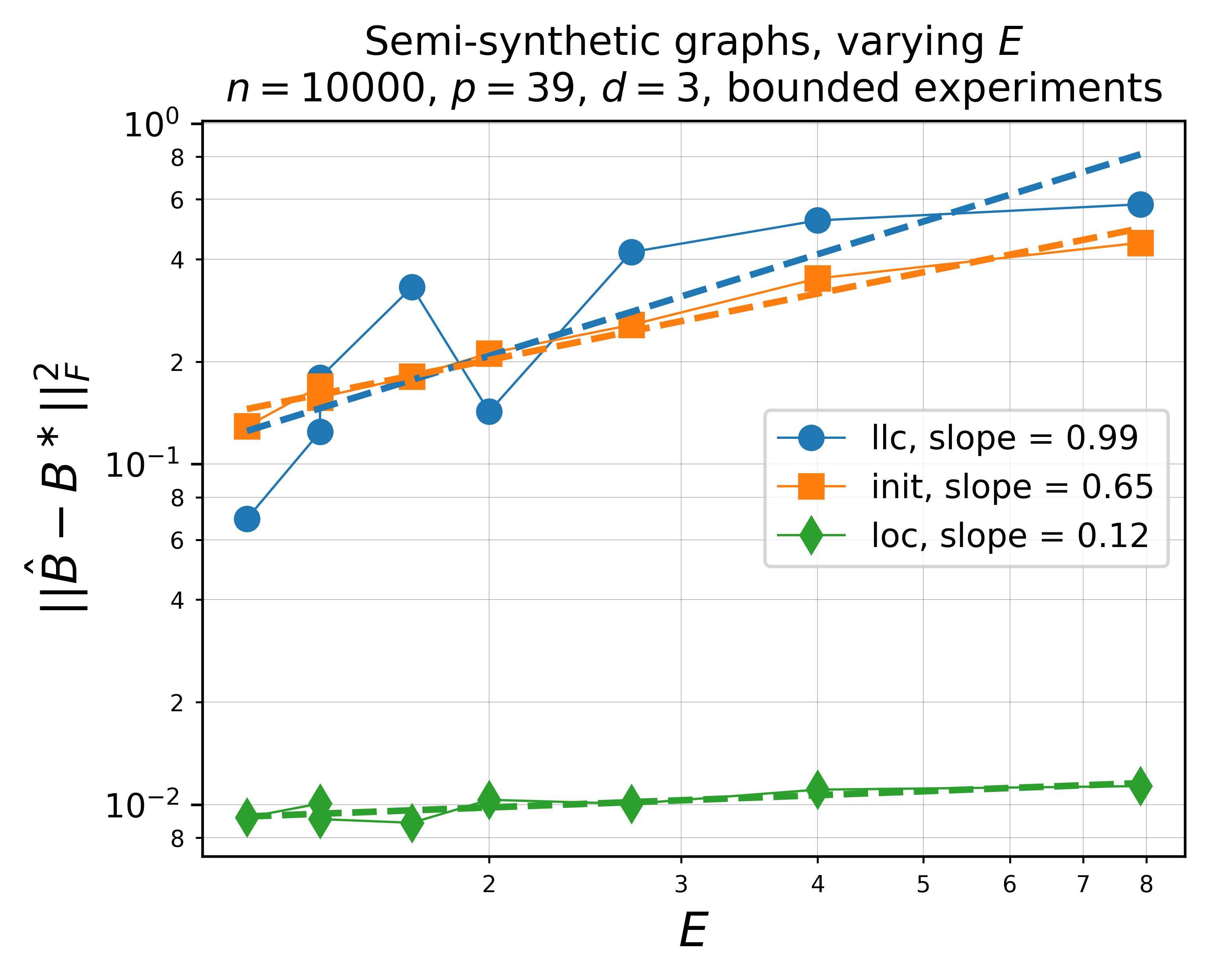}
		\caption{Varying \( E \) for semi-synthetic graph based on~\cite{CaiBazGia13}.}
		\label{fig:semisynthvare}
	\end{subfigure}
	\caption{Experiments for other types of graphs}
	\label{fig:othergraphs}
\end{figure}
\paragraph{Random regular graphs with oracle choice:}

In Figure \ref{fig:rand}, we collect comparisons for the estimation rates of \( \hB_\llc \), \( \hB_\init \), and \( \hB_\loc \), varying \( n, p, d, \) and \( E \), respectively, where the varying \( E \) case is given by bounded experiments with a varying bound on the size $k$ of the experiments which, of course, governs the total number $E$ of experiments needed for separation.
In all other cases, we consider binary experiments.

Figure \ref{fig:randvarn} indicates that all three estimators exhibit a risk that scales as  \( 1/n \) and displays a clear ordering in the performance of the three candidates where \( \hB_\llc \) performs worse than \( \hB_\init \), which in turn is worse than \( \hB_\loc \).

In Figure \ref{fig:randvarp}, we observe a scaling with respect to \( p \) that is slightly worse than guaranteed by our theorems and could be due to the presence of log factors.
In Figure \ref{fig:randvard}, we in turn see that the scaling with respect to \( d \) is slightly better than expected, hinting at good adaptation to the sparsity parameter \( d \).
Most interestingly, in Figure \ref{fig:randvare}, we observe that the scaling with respect to \( E \) when increasing the number of experiments appears to be better than predicted by our theory: about \( E^{1/2} \) for \( \hB_{\llc} \) and \( \hB_{\init} \), about \( E^{1/3} \)  for \( \hB_{\loc} \). This different behavior is even more striking in Figure~\ref{fig:semisynthvare} where the performance of \( \hB_{\loc} \) appears to decay at most logarithmically in $E$.



\paragraph{Disconnected clique graphs:}
In Figure \ref{fig:clustersvard}, we plot the same experiment as in Figure \ref{fig:randvarn}, only this time with disconnected clusters instead of random regular graphs. 
We notice a similar behavior, with the key difference of the performance of \( \hB_{\init} \) surpassing that of \( \hB_{\loc} \).
This could be explained by the fact that the penalization in the objective \( \hB_{\init} \) is particularly suited for the estimation of this kind of graphs since the sparsity of \( (I - B^\ast)^\top (I - B^\ast) \) in this case almost coincides with the one of \( B^\ast \), which can be seen from the argument that led to \eqref{eq:br} in the proof of Theorem \ref{THM:2-STEP-RATES}.

\paragraph{Semi-synthetic graph:}
The performance of the three estimators on the semi-synthetic data built from the graph taken in \cite{CaiBazGia13} appears in Figure \ref{fig:semisynthvare}. The LLC estimator  \( \hB_{\llc} \) performs similarly to   \( \hB_{\init} \) and both suffer in comparison to \( \hB_{\loc} \) either in terms of absolute performance and in terms of scaling with \( E \).

\subsubsection{Stability}
\begin{figure}[t]
	\begin{subfigure}[t]{0.45\textwidth}
		\includegraphics[width=\textwidth]{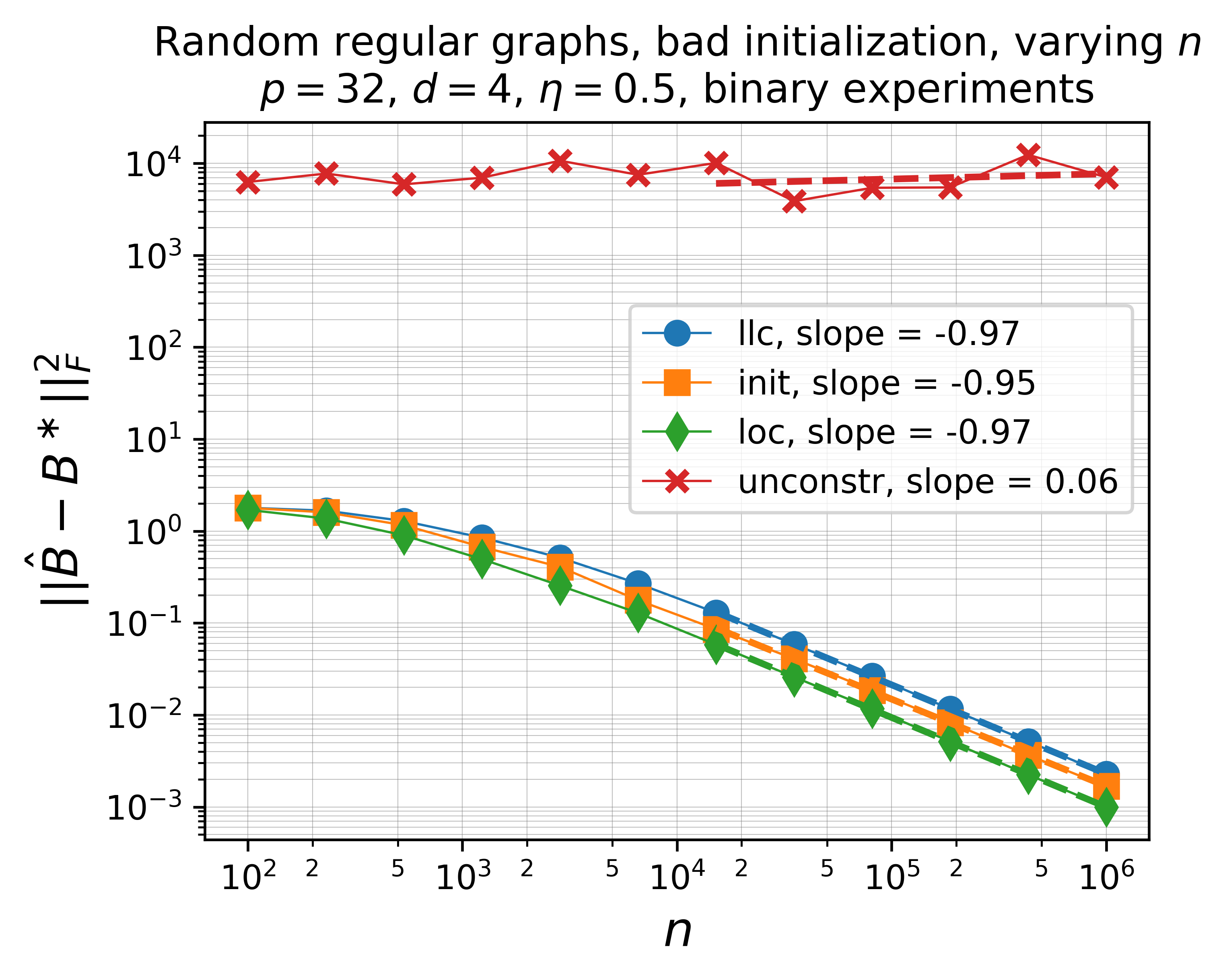}
		\caption{Same case as Figure \ref{fig:randvarn}, but with bad initialization. ``unconstr'' refers to the case where $\localradius=\infty$ and illustrates the need for localization.}
		\label{fig:randbadinit}
	\end{subfigure}
	\hspace{1em}
	\begin{subfigure}[t]{0.45\textwidth}
		\includegraphics[width=\textwidth]{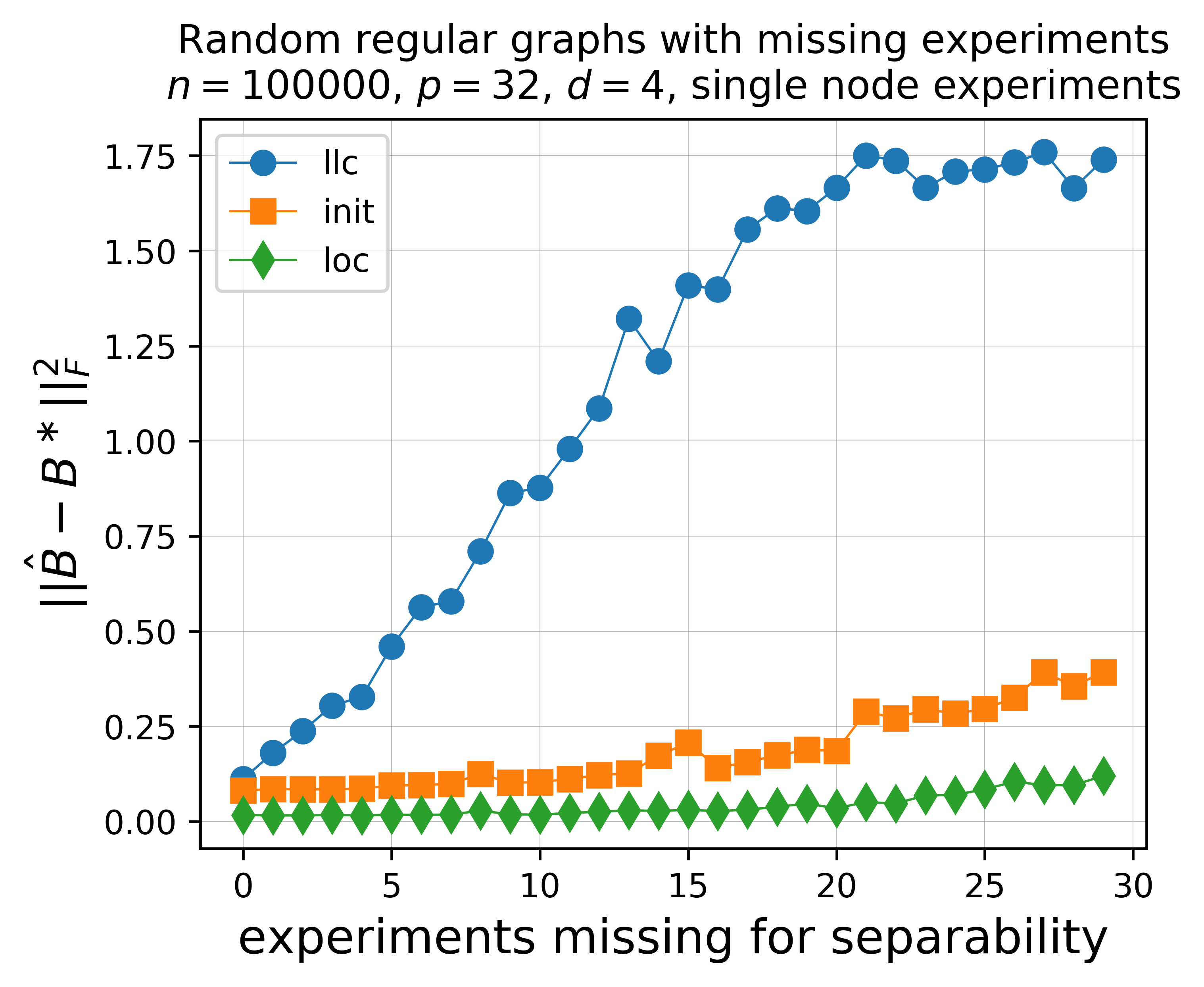}
		\caption{Random regular graphs case where Definiton \ref{def:sep} is violated.}
		\label{fig:randmissing}
	\end{subfigure}
	\label{fig:missingbadinit}
	\caption{Additional computational experiments, running the algorithm with bad initializations and with experiments not satisfying complete separability.}
\end{figure}

\paragraph{Role of initialization:}
In Figure \ref{fig:randbadinit}, we show the same setup as in Figure \ref{fig:randvarn}, only this time, the calculation for \( \hB_\init \) is initialized with a random matrix as outlined in Section \ref{sec:numerics-setup}.
Additionally, we plot the result of optimizing an unconstrained version of \( \hB_\loc \) with the same bad initialization, denoted by \( \hB_{\mathrm{unconstr}} \).
We observe that the performance of the latter is very bad due to the non-convex nature of the objective together with the fact that a bad initialization point is chosen.
However, even though \( \hB_\init \) is found through solving a non-convex objective as well, it seems to be robust enough to yield comparable performance and hence serve as a good initialization for calculating \( \hB_\loc \) even with a poor initial choice of \( B \).

\paragraph{Missing experiments:}
In Figure \ref{fig:randmissing} we investigate the robustness to systems of interventions that do not fulfill the separability condition in Definition \ref{def:sep}.
For this, we consider single-node experiments and plot the number of experiments that are missing from a completely separating set of such experiments (in which case we would have \( E = p \)).
The likelihood based approaches are much more robust in this case, and to a larger degree than degree of freedom calculations as in Appendix \ref{sec:idfblty} would suggest.

\appendix 

\section{Non-identifiability in the cyclic case for equal variances}
\label{sec:idfblty}

In this section, we give a brief argument to show that for generic matrices \( B^\ast \), unlike the acyclic case considered in \cite{LohBuh14, PetBuh14}, having equal noise variance as required in Assumption \ref{assump:noise} does not lead to identifiability from observational data.

The argument is based on counting dimensions of the null space of the non-linear maps
\begin{equation}
	\label{eq:jv}
		\Theta^e(B) = (I - U_e B)^\top (I - U_e B).
\end{equation}
One limitation of our argument is that it does not cover the potential identifiability of \( B^\ast \) from observational data under the additional assumption of bounded in-degree \( d(B^\ast) \le d \).

\begin{proposition}
	\label{prp:idfblty}
	Define the integer
\begin{equation}
\label{EQ:defm}
m=| \{ (i, j) | \exists e: i \in \cU_e, \, j \in \cJ_e, \, i \ne j \} |+ | \{ (i, j) | i < j, \exists e: i, j \in \cU_e\}| + p\,.
\end{equation}
Then the matrix \( B^\ast \in \matrixclass \) is not uniquely determined by \( \Theta^{\ast, e} = \Theta^e(B^\ast), e \in \cE \) whenever $m<p^2-p$. In particular, without interventions, this condition holds as soon as $p \ge 4$.
%
\end{proposition}

\begin{proof}
	Consider the maps 
	\begin{equation}
		\Theta \colon \cB \cong \R^{p^2 - p} \to \R^{E \times p^2} \qquad \text{and}\qquad  \bar \Theta \colon \R^{p^2} \to \R^{E \times p^2},
	\end{equation}
	defined by stacking all \( \Theta^e \) into one vector and accepting respectively matrices with zero diagonal and arbitrary diagonal. Similarly,  denote by \( \bar \Theta^e \) the map \( \Theta^e \) when not restricted to matrices with zero-diagonal.

We show that the derivative of \( \Theta \)  has constant rank bounded above by \( p^2 - p - (m - p) \) at a point \( B^\ast \in \cB \). In turn, whenever $m<p$, this implies the existence of \( \tilde B \neq B^* \) such that \( \Theta(B^\ast) = \Theta(\tilde B) \) by the constant rank theorem

	First, let \( B \in \matrixclass \) be arbitrary and compute the derivative of \( \bar \Theta \) at a point \( B \) by computing the derivative $D \bar \Theta^e(B)$ of the individual maps \( \bar \Theta^e : \R^{p \times p}\to \R \). For any  \( \bar H \in \R^{p \times p} \), it holds
	\begin{align}
		\label{eq:gi}
		D \bar \Theta^e(B)[\bar H]
		= {} & (I - U_e B)^\top (-U_e \bar H) + (-U_e \bar H)^\top (I - U_e B)\\
		= {} & - (I - B)^\top (U_e \bar H) - (U_e \bar H)^\top (I - B)\\
		= {} & - A^{-\top} \left[ (U_e \bar H A) + (U_e \bar H A)^\top \right] A^{-1}, \qquad \quad (A = (I - B)^{-1})
		\label{eq:go}
	\end{align}
where we used the fact that \( U_e^2 = U_e \).

Next, we compute the dimension of the null space of \( D \Theta (B) \).
To that end, observe that for any $H \in \R^{p \times p}$ such that \( D \Theta^e(B)[H] = 0 \), it holds  \( D \bar \Theta^e(B)[\bar H] = 0 \) and \( (\bar H)_{ii} = 0 \) for all \( i \in [p] \).
	To characterize the dimensionality of the subspace of such matrices $H$, we first consider the null space of \( D \bar \Theta^e(B)[\bar H] = 0 \) and then intersect it with the subspace given by \( (\bar H)_{ii} = 0 \).

	Abbreviate \( \bar G = \bar H A \).
	By \eqref{eq:go}, $D \bar \Theta^e(B)[\bar H]=0$ for all $B \in \cB$ whenever \( (U_e \bar G) + (U_e \bar G)^\top = 0 \) for all \( e \). We permute the indices such that \( \cJ_e = \{1, \dots, |\cJ_e| \} \) to write this equality in block form:
	\begin{equation}
		\label{eq:gl}
		U_e \bar G + (U_e \bar G)^\top =
		\begin{bmatrix}
			0 & J_e \bar G^\top U_e\\
			U_e \bar G J_e & U_e (\bar G + \bar G^\top) U_e\,.
		\end{bmatrix}
	\end{equation}
For  each $e \in \cE$ the three nonzero blocks above translate into the following conditions:
	\begin{alignat}{2}
		\label{eq:gm}
		\bar G_{i, j} = {} & 0 \quad &&\text{if }\exists e: i \in \cU_e, \, j \in \cJ_e, \, i \ne j\\
		\bar G_{i, j} = {} & - \bar G_{j, i} \quad &&\text{if }\exists e: i, j \in \cU_e\,.
	\end{alignat}
As a result $\bar H=\bar G(I-B)$ is the image of $(I-B)$ through the linear operator $\bar G$ that lives in the intersections of the orthogonal subspaces defined by the above constraints. Thus, each constraint contributes 1 to the codimension of the null space of \( D \bar \Theta(B) \). Equivalently, each constraint contributes 1 to the rank  \( D \bar \Theta(B) \). 

Next, we discuss how to deal with the fact that we need to compute \( \mathrm{rank}( D \Theta(B) )\) instead of \( \mathrm{rank}( D \bar \Theta(B) ) \), where \( B \) is restricted to lie in the subspace of matrices with zero-diagonal, thus we need to restrict \( \bar H \) above accordingly. Intuitively, we want to say that the rank can increase by at most \( p \), the number of additional linear constraints on the null space, but we need to further establish that there is a \( B^\ast \in \matrixclass \) such that the rank of \( D \Theta(B) \) is constant in a neighborhood of \( B^\ast \).
Adding to that the constraint that $\bar H$ has null diagonal, we get $\mathrm{rank}( D  \Theta(B) )\le m$,
where $m$ is defined in~\eqref{EQ:defm}.

Next, we show that $\mathrm{rank}( D  \Theta(B) )$ is, in fact, constant and equal to some $r^*$ in a neighborhood of $B^*$ to apply the constant rank theorem.

 To that end, let \( B^\ast \) be such that   \(r^\ast:=\mathrm{rank}( D \Theta(B^*) )\ge \mathrm{rank}( D \Theta(B) ) \) for all \( B \in \matrixclass \).
	Considering \( D \Theta(B^\ast) \) a matrix let  \( S^\ast \) be a maximal principal minor and denote the restriction of \( D \Theta(B^\ast) \) to \( S^\ast \) by \( [D \Theta(B^\ast)]_{S^\ast} \).
	By definition, we have 
	$$
	 \rank (D \Theta(B^\ast)) = \rank ([D \Theta(B^\ast)]_{S^\ast}) = r^\ast.
	 $$
	Moreover, the map \( B \mapsto f(B) := \det [D \Theta(B^\ast)]_{S^\ast} \) is a polynomial in the elements of \( B \) such that \( f(B^\ast) \neq 0 \).
	By continuity, it also holds that \( f(B) \neq 0 \) in an open neighborhood of \( B^\ast \) as well, and thus \( \rank ( D \Theta(B)) \ge r^\ast \) in that neighborhood.
	But since   \( r^\ast \) is maximal, \( \rank D \Theta(B) = r^\ast \) for \( B \) in an open neighborhood of \( B^\ast \).



	The above means that we can apply the constant rank theorem \cite[Theorem II.7.1]{Boo86} to obtain diffeomorphisms \( \varphi \colon \R^{p^2 - p} \supset V_1 \to U_1 \subseteq \R^{p^2 - p}\), \( \psi \colon \R^{E \times p^2} \supset U_2 \to V_2 \subseteq \R^{E \times p^2}\), with \( U_j, V_j \) open sets for \( j \in \{1, 2\} \) such that
	\begin{equation}
		\psi \circ \Theta \circ \varphi^{-1}(x) = (x_{1}, \dots, x_{r^\ast}, 0, \dots, 0), \quad \text{and} \quad \varphi^{-1}(0) = B^\ast.
	\end{equation}
	If \( r^\ast < p^2 - p \), we obtain a continuum of pre-images of \( \Theta(B^\ast) \) as
	\begin{equation}
		\tilde B(x) = \varphi^{-1}(0, \dots, 0, x_{p^2 - p - r^\ast + 1}, \dots, x_{p^2 - p}).
	\end{equation}
	for all \( (0, \dots, 0, x_{p^2 - p - r^\ast + 1}, \dots, x_{p^2 - p}) \in V_1 \), which includes points other than \( 0 \) because \( V_1 \) is an open set.

	To conclude, recall that $r^* \le m$ so that $m<p^2-p$ is a sufficient condition for the failure of injectivity of $\Theta$. This completes the first part of the proof.
	
To obtain the conclusion without interventions, note that in this case
$$
m=| \{ (i, j) | i < j, \exists e: i, j \in \cU_e\}| + p ={p \choose 2}+ p 
$$
so that $m<p^2-p$ whenever $p>3$.
%
%
\end{proof}

\section{Numerical speed-up}
\label{sec:low-rank-update}

When many experiments are performed with a small number of nodes that are intervened on, say \( | \cJ_e | \leq k \), calculating the log-likelihood term in the algorithms considered in Section \ref{sec:numerics} in a naive way takes \( O(E p^3) \) operations: both calculating \( \tr(\Theta^e(B) \hat \Sigma^e) \) and performing a Cholesky decomposition for each of the \( E \) matrices \( \Theta^e(B) = L^e (L^e)^\top \) with \( L^e \) lower triangular takes \( O(p^3) \) time.
The Cholesky decomposition in turn is used to compute
\begin{equation}
	\log \det \Theta^e(B) = \sum_{i=1}^{p} 2 \log (L_{ii}).
\end{equation}

The computational complexity can be improved by using a low rank decomposition of \( \Theta^e(B) \), both for computing the trace term and the Cholesky decomposition \( (I - B)^\top (I - B) = LL^\top \).
To see this, write \( J_e = I - U_e \) and decompose
\begin{align}
	\label{eq:dx}
	\leadeq{(I - U_e B)^\top (I - U_e B)}\\
	= {} & (I - B + J_e B)^\top (I - B + J_e B)\\
	= {} & (I - B)^\top (I - B) + (J_e B)^\top (I - B) + (I - B)^\top (J_e B) + (J_e B)^\top (J_e B)\\
	= {} & (I - B)^\top (I-B) + (J_e B)^\top - (J_e B)^\top B + (J_e B) - (J_e B)^\top (J_e B) + (J_e B)^\top (J_e B)\\
	= {} & (I - B)^\top (I-B) + (J_e B)^\top + (J_e B) - (J_e B)^\top (J_e B)\\
	= {} & (I - B)^\top (I - B) - (J_e - J_e B)^\top (J_e - J_e B) + J_e^\top J_e.
\end{align}
Hence, the Cholesky decomposition \( L^e \) can be computed by a rank \( k \) update followed by a rank \( k \) downdate of \( L \), which takes \( O(k p^2) \) \cite{See04}.
Computation of the trace terms \( \tr (\hat \Sigma^e \Theta^e(B)) \) can be sped up analogously, also taking \( O(k p^2) \) time.

Hence, the total time to compute the log-likelihood is \( O(p^3 + E k p^2) \).
In a similar manner, computing the objective function for step \eqref{eq:ds} in the non-convex ADMM procedure can be done in \( O(p^3 + E k p^2) \) time, although one iteration takes \( O(E p^3) \) time due to the complexity of performing step \eqref{eq:dr}.

\section{Proof of lower bounds}
\label{sec:lower-bound-proof}

\subsection{Proof of Theorem \ref{THM:LOWERBOUNDS}}

To begin, recall the definition of the redundancy factor
\begin{equation}
	\kappa = \kappa(\cE) = \max_{(i, j)} |\{e \in \cE : (i, j) \text{ separated in } e \}|.
\end{equation}


The proof of Theorem \ref{THM:LOWERBOUNDS} is based on standard techniques for minimax lower bounds \cite{Tsy09}.

\begin{theorem}[{\cite[Theorem 2.5]{Tsy09}}]
	\label{thm:fano}
	Denote by \( \classgen \subseteq \R^{p \times p} \) a set of possible hypotheses with associated probablity measures \( P_{B} \) for \( B \in \classgen \).

	Fix \( M \geq 2, s > 0, \alpha\in (0,1/8), \) and assume that there exists \( B_0, \dots, B_M \in \classgen \),  such that
	\begin{enumerate}
		\item \label{itm:fano-sep} \( \| B_j - B_k \|_F \geq 2s > 0 \) for all \( 0 \leq j < k \leq M \);
		\item \label{itm:fano-kl}  $\KL(P_j | P_0) \leq \alpha \log M$   for all \( j = 1, \dots, M \), 
			where \( P_j = P_{B_j}\).
	\end{enumerate}
	Then,
	\begin{equation}
		\label{eq:ea}
		\inf_{\hat B} \sup_{B^\ast \in \classgen} P_B( \| \hB - B^\ast \|_F \geq s) \geq \frac{\sqrt{M}}{1 + \sqrt{M}} \left( 1 - 2 \alpha - \sqrt{\frac{2 \alpha}{\log M}} \right),
	\end{equation}
	where the infimum in \eqref{eq:ea} is taken over all measurable functions \( \hB \) on the observations.
\end{theorem}

Before proceeding with the proof, we first present two lemmas.
Lemma \ref{lem:kl-gauss-upper} gives a way to upper bound the Kullback-Leibler divergence between two Gaussian distributions in terms of their concentration matrices, while Lemma \ref{lem:vg} contains a version of the Varshamov-Gilbert lemma adopted to produce candidate matrices with the same row sparsity.
Their proofs can be found in Section \ref{sec:proof-kl-gauss-upper} and Section \ref{sec:proof-vg}, respectively.

In the following, we denote by $d_H(A, B)$ the Hamming distance between two matrices \( A, B \in \R^{p \times p} \). It is defined by
	$d_H(A, B) = | \{ (i, j) \in [p]^2 : A_{i,j} \neq B_{i,j} \} |$.

\begin{lemma}
	\label{lem:kl-gauss-upper}
	Let \( \Theta_1, \Theta_2 \in \R^{p \times p} \) be two positive definite concentration matrices and \( P_1=  \cN(0, \Theta_1^{-1}) \) and \( P_2 = \cN(0, \Theta_2^{-1}) \) the associated Gaussian distributions.
	If
	\begin{equation}
		\label{eq:cw}
		\| \Theta_1 - \Theta_2 \|_{\mathrm{op}} \leq \frac{\lambda_{\mathrm{min}}(\Theta_2)}{2},
	\end{equation}
	then,
	\begin{equation}
		\label{eq:cv}
		\KL(P_1 | P_2) \leq \frac{1}{\lambda_{\mathrm{min}}(\Theta_2)^2} \| \Theta_1 - \Theta_2 \|_F^2.
	\end{equation}
\end{lemma}

\begin{lemma}
	\label{lem:vg}
	Given \( m \geq 1 \) and \( 1 \leq d \leq m/2 \), there is a family \( H_1, \dots, H_M \) of matrices in \( \{ 0, 1 \}^{m \times m} \) such that
	\begin{enumerate}[label=(\roman*), itemsep = 0.5em]
		\item Every row of \( H_i \) is \( d \)-sparse for \( i = 1, \dots, M \);
		\item \( \displaystyle d_H(H_i, H_j) \geq \frac{md}{2}, \quad i \neq j \);
		\item \( \displaystyle 	\log M \geq \frac{md}{16} \log \left( 1 + \frac{m}{2d} \right).
			\)
	\end{enumerate}
\end{lemma}

Taking the above lemmas as given, we proceed to prove Theorem \ref{THM:LOWERBOUNDS}.

	We apply Theorem \ref{thm:fano} by constructing an appropriate set of hypotheses \( B_0, \dots, B_M \).
	Without loss of generality, assume that \( p \) is even.
	Set \( B_0 = 0 \) to be the all zeros matrix, and apply Lemma \ref{lem:vg} with \( m = p/2 \) to obtain \( M \) matrices \( H_1, \dots, H_M \in \{0, 1\}^{p/2 \times p/2} \) with pairwise Hamming distance at least \( pd/4 \) and with
	\begin{equation}
		\label{eq:dd}
		\log M \geq \frac{pd}{32} \log \left( 1 + \frac{p}{4d} \right).
	\end{equation}
	We define \( B_i \), \( i = 1, \dots, M \) as block matrices by setting
	\begin{equation}
		\gamma = \sqrt{\frac{\alpha}{64 \kappa}},
		\quad
		\beta = \gamma \sqrt{\frac{E}{n} \log\left(1 + \frac{p}{4d}\right)},
		\quad
		\text{and}
		\quad
		\label{eq:de}
		B_i =
		\begin{bmatrix}
			0 & \beta H_i\\
			- \beta H_i^\top & 0
		\end{bmatrix}.
	\end{equation}

	By construction, for every \( i \in [M] \), every row of \( B_i \) is \( d \)-sparse, and \( B_i \) has zero-diagonal.
	Moreover, by \( \kappa \geq 1 \), \( \alpha < 1/8 \), and by assumption
	\begin{equation}
		\label{eq:cs}
		n \geq p d E^2 \log\left(1 + \frac{p}{4d} \right) \geq d^2 E \log\left(1 + \frac{p}{4d}\right),
	\end{equation}
	so we get from the Gershgorin circle theorem that
	\begin{align}
		\label{eq:dg}
		\| B_i \|_\mathrm{op} \leq \gamma d \sqrt{\frac{E}{n} \log\left(1 + \frac{p}{4d}\right)} \leq \gamma \leq \frac{1}{5} \le 1 - \eta
	\end{align}
	in light of \( \eta \le 1/2 \).
	Hence, \( B_i \in \matrixclass(p, d, \eta) \) for all \( i \in [M] \).

	Next, we can lower bound the pairwise distances by
	\begin{align}
		\label{eq:df}
		\| B_i - B_j \|_F^2
		\geq {} & 2 \beta^2 d_H(H_i, H_j)
		\geq \gamma^2 \frac{pdE}{4n} \log\left(1 + \frac{p}{4d}\right),\\
		\| B_i - B_0 \|_F^2
		\geq {} & \gamma^2 \frac{pd}{2n} \log\left(1 + \frac{p}{4d}\right)
		\geq \gamma^2 \frac{pdE}{4n} \log\left(1 + \frac{p}{4d}\right),
	\end{align}
	which yields the needed separation in Theorem \ref{thm:fano}\ref{itm:fano-sep}.

	We proceed to estimate the KL divergence between two distributions corresponding to matrices \( B_0 \) and any \( B_i \), \( i \geq 1 \).
	Decompose the difference between the concentration matrices as
	\begin{align}
		\label{eq:ch}
		\Theta^e(B_i) - \Theta^e(B_0)
		= {} & (I - U_e B_i)^\top (I - U_e B_i) - I \\
		= {} & (U_e B_i)^\top + U_e B_i + B_i U_e B_i.
		\label{eq:cx}
	\end{align}
	Because \( \| B_i \|_{\mathrm{op}} \leq 1/5 \), \eqref{eq:ch} together with \( \| U_e \|_\mathrm{op} \leq 1 \) and the sub-multiplicativity of the operator norm implies
	\begin{equation}
		\label{eq:ci}
		\| \Theta^e(B_2) - \Theta^e(B_1) \|_{\mathrm{op}} \leq \frac{1}{5} + \frac{1}{5} + \frac{1}{25} \leq \frac{1}{2} = \frac{\lambda_\mathrm{min}(I)}{2} = \frac{\lambda_{\mathrm{min}}(\Theta^e(B_0))}{2},
	\end{equation}
	so the hypothesis of Lemma \ref{lem:kl-gauss-upper} is satisfied for all pairs \( (\Theta^e(B_i) ,  \Theta^e(B_0)) \).
	By Lemma \ref{lem:kl-gauss-upper}, \eqref{eq:cx}, and the tensorization property of the KL divergence, we obtain
	\begin{align}
		\KL(P_{B_i} | P_{B_0})
		\leq {} & \sum_{e} n_e \| \Theta^e(B_i) - \Theta^e(B_0) \|_F^2
		\leq \frac{2n}{E} \sum_{e} \| U_e B_i + (U_e B_i)^\top \|_F^2 + 2 n \| B_i \|_F^4\\
		\leq {} & \frac{2n}{E} \sum_{e} \Big[ 2 \sum_{\substack{k \in \cU_e\\ \ell \in \cJ_e}} (B_i)_{k,\ell}^2 + \sum_{k, \ell \in \cU_e} ((B_i)_{k,\ell} + (B_i)_{k,\ell})^2 \Big] + 2 n \| B_i \|_F^4.
		\label{eq:cd}
	\end{align}
	Since \( B_i \) is defined to be anti-symmetric, we have
	\begin{equation}
		\label{eq:dh}
		\sum_{k, \ell \in \cU_e} ((B_i)_{k,\ell} + (B_i)_{k,\ell})^2 = 0, \quad i = 1, \dots, M, \, e = 1, \dots, \cE.
	\end{equation}
	Moreover,
	\begin{equation}
		\label{eq:di}
		\sum_e \sum_{\substack{k \in \cU_e\\l \in \cJ_e}} (B_i)_{k,\ell}^2 \leq \kappa(\cE) \| B_i \|_F^2.
	\end{equation}
	Combining \eqref{eq:cd}, \eqref{eq:dh} and \eqref{eq:di}, we arrive at
	\begin{align}
		\label{eq:cg}
		\KL(P_{B_i} | P_{B_0}) \leq \frac{4 n \kappa}{E} \| B_i \|_F^2 + 2 n \| B_i \|_F^4.
	\end{align}
It remains to compute the Frobenius norm of each \( B_i \),
\begin{equation}
	\label{eq:cn}
	\| B_i \|_F^2 = \gamma^2 \frac{pdE}{4 n} \log\left(1 + \frac{p}{4d}\right).
\end{equation}
Hence, because
\begin{equation}
	\label{eq:cq}
	n \geq p d E^2 \log\left(1 + \frac{p}{2d}\right), \quad \gamma < 1,
	\quad
	\text{and}
	\quad
	\gamma^2 = \frac{\alpha}{32 \kappa},
\end{equation}
we obtain
\begin{align}
	\label{eq:cp}
	\KL(P_j | P_0)
	\leq {} &
	\frac{4 n \kappa}{E} \gamma^2 \frac{pdE}{4 n} \log\left(1 + \frac{p}{4d}\right) 
	+ 2 n \gamma^4 \frac{p^2 d^2 E^2}{16 n^2} \left(\log\left(1 + \frac{p}{4d}\right)\right)^2\\
	= {} & \kappa pd \gamma^2 \log\left(1 + \frac{p}{4d}\right) + \gamma^4 \frac{p^2 d^2 E^2}{8n} \left( \log\left(1 + \frac{p}{4d}\right) \right)^2\\
	\leq {} & 2 \kappa pd \gamma^2 \log\left(1 + \frac{p}{4d}\right)\\
	\leq {} & \alpha \frac{pd}{32} \log\left(1 + \frac{p}{4d}\right)
	= \alpha \log M.
\end{align}
Finally, we can pick \( \alpha = \frac{1}{16} \) in Theorem \ref{thm:fano} to conclude that
\begin{equation}
	\label{eq:ct}
\inf_{\hat B} \sup_{B^\ast \in \classgen} P_B \left( \| \hB - B^\ast \|_F \ge \frac{1}{2^{14} \kappa} \frac{pdE}{n} \log\left(1 + \frac{p}{4d}\right) \right) \ge c,
\end{equation}
for some constant \( c > 0 \).

\subsection{Proof of Lemma \ref{lem:kl-gauss-upper}}
\label{sec:proof-kl-gauss-upper}
The Kullback-Leibler divergence between two Gaussians \( P_1 = \cN(0, \Theta_1^{-1}) \) and \( P_2 = \cN(0, \Theta_2^{-1}) \) is given by
\begin{equation}
	\label{eq:cb}
	\KL(P_{1} | P_{2}) = \frac{1}{2} \left( \tr(\Theta_1^{-1} (\Theta_2 - \Theta_1) - \log \det \Theta_2 + \log \det \Theta_1 \right).
\end{equation}

%

Using the fact that the first derivative of \( \Theta \mapsto -\log \det \Theta \) is \( - \Theta^{-1} \) and the second derivative is \( \Theta^{-1} \otimes \Theta^{-1} \), we employ a Taylor expansion about \( \Theta_1 \) (compare \eqref{eq:al} in proof of Lemma \ref{lem:bound-lh}) to obtain
\begin{equation}
	\label{eq:cc}
	\KL(P_{1} | P_{2}) = \frac{1}{4} \tr(\tilde{\Theta}^{-1} (\Theta_2 - \Theta_1) \tilde{\Theta}^{-1} (\Theta_2 - \Theta_1)),
\end{equation}
for some \( \tilde{\Theta} = t \Theta_1 + (1-t) \Theta_2 \), \( t \in [0, 1] \).
By considering the square root of \( \tilde{\Theta}^{-1} \), this can be expressed in terms of the Frobenius norm of the difference \( \Theta_2  - \Theta_1 \),
\begin{equation}
	\label{eq:ed}
	\frac{1}{4} \tr(\tilde{\Theta}^{-1} (\Theta_2 - \Theta_1) \tilde{\Theta}^{-1} (\Theta_2 - \Theta_1))
	= {}  \frac{1}{4} \| \tilde{\Theta}^{-1/2} (\Theta_2 - \Theta_1) \tilde{\Theta}^{-1/2} \|_F^2
	\leq {}  \frac{1}{4 \lambda_{\mathrm{min}}(\tilde{\Theta})^{2}} \| \Theta_2 - \Theta_1 \|_F^2.
\end{equation}
By Weyl's inequality, \cite[Section 6.7, Theorem 2]{Fra12},
\begin{equation}
	\label{eq:cu}
	|\lambda_{\mathrm{min}}(\Theta_2) - \lambda_{\mathrm{min}}(\tilde \Theta)| \leq \| \Theta_2 - \tilde \Theta \|_{\mathrm{op}} \leq \| \Theta_1 - \Theta_2 \|_{\mathrm{op}}.
\end{equation}
Hence, if \( \| \Theta_2 - \Theta_1 \|_{\mathrm{op}} \leq \lambda_{\mathrm{min}}(\Theta_2)/2 \), then
\begin{equation}
	\label{eq:cf}
	\frac{1}{\lambda_\mathrm{min}(\tilde{\Theta})^2} \leq \frac{4}{\lambda_{\mathrm{min}}(\Theta_2)^2},
\end{equation}
which together with \eqref{eq:ed} implies \eqref{eq:cv}.

\subsection{Proof of Lemma \ref{lem:vg}}
\label{sec:proof-vg}

We use the probabilistic method to show the existence of the family \( H_1, \dots, H_M \), modifying a standard argument that can be found in \cite[Lemma 2.9]{Tsy09}.

Let \( H_1, \dots, H_M \) be \( M \) independent random matrices \( H_k \), where each row of \( H_k \) is a zero-one-vector corresponding to a subset of \( [p] \) with cardinality \( d \) drawn uniformly at random.
More precisely, for the \( i \)th row of the matrix \( H_k \), draw \( U_1^i \) uniformly from \( \{1, \dots, m\} \), and \( U_j^i \) conditioned on \( U^i_1, \dots, U^i_{j - 1} \) uniformly from the set \( \{1, \dots, m\} \setminus \{U^i_1, \dots, U^i_{i-1} \} \), \( j = 2, \dots, d \).
Then, set
\begin{equation}
	\label{eq:cz}
	(H_k)_{i, j} = 
	\left\{
		\begin{aligned}
			1, \quad {} & j \in \{U^{i}_1, \dots, U^{i}_d\}\\
			0, \quad {} & \text{otherwise}.
		\end{aligned}
	\right.
\end{equation}

By a union bound, the probability that there exists a pair \( k, \ell \) for which \( d_H(H_k, H_\ell) \leq md/2 \) can be bounded by the probability of this occurring for one draw of \( H_1 \), comparing to a fixed \( H_0 \) with d-sparse rows, say \( (H_0)_{kl} = \1_{\{ k \leq d \}} \),
\begin{align}
	\label{eq:cj}
	P \left( \exists \ell \neq k : d_H(H_\ell, H_k) < \frac{md}{2} \right)
	\leq {} & {M \choose 2} P \left(d_H(H_1, H_0) < \frac{md}{2} \right),
\end{align}
because for each row, every \( d \) sparse pattern is equally likely.

We can lower bound the Hamming distance by the number of elements in \( \mathrm{supp}(H_0) \) on which \( H_1 \) is one, \ie
\begin{align*}
	d_H(H_1, H_0) \geq md - \sum_{i = 1}^{m} \sum_{j = 1}^d Z_{i, j}
\end{align*}
with \( Z_{i, j} = \1 (U^i_j \leq d) \).
Then, \( Z_{i, j} \sim \bern(q_{i, j}) \) with \( q_{i, 1} = \frac{d}{m} \) and, noting that \( d \leq m/2 \),
\begin{equation}
	\label{eq:da}
	q_{i, j} = \frac{d - \sum_{\ell = 1}^{j - 1} Z_{i, j}}{m - (j-1)} \leq \frac{d}{m - d} \leq \frac{2d}{m}, \quad j \geq 2 .
\end{equation}

From there, apply a Chernoff bound, that is, pick \( \lambda > 0 \) and estimate
\begin{align}
	P \left( d_H(H_1, H_0) < \frac{md}{2} \right)
	\leq {} & P \left( \sum_{i = 1}^m \sum_{j = 1}^d Z_{i, j} \geq \frac{md}{2}  \right)\\
	\le {} & \E \left[ \exp \left( \lambda \sum_{i = 1}^m \sum_{j = 1}^d Z_{i, j} \right) \right] \exp \left( -\lambda \frac{md}{2} \right)
	\label{eq:db}
\end{align}
For a Bernoulli distribution \( Z \sim \bern(q) \), the moment generating function is given by
\begin{equation}
	\label{eq:cy}
	\E \left[ \exp( \lambda Z) \right] = q (\exp(\lambda) - 1) + 1.
\end{equation}
Together with the observation that \( Z_{i, j} \) is stochastically dominated by a \( \bern(2d/m) \) distribution, we can estimate
\begin{align}
	\label{eq:ck}
	\E \left[ \exp \left( \lambda \sum_{i = 1}^m \sum_{j = 1}^d Z_{i, j} \right) \right]
	\leq {} & \left( \frac{2d}{m} (\exp(\lambda) - 1) + 1 \right)^{md}.
\end{align}
Setting \( \lambda = \log\left(1 + \frac{m}{2d}\right) \), we get that
\begin{equation}
	\label{eq:cl}
	\E \left[ \exp \left( \lambda \sum_{i = 1}^m \sum_{j = 1}^d Z_{i, j} \right) \right]
	\leq 2^{md}.
\end{equation}

Combining \eqref{eq:cl}, \eqref{eq:db}, \eqref{eq:cj}, and the estimate \( {M \choose 2} \leq M^2 \), we obtain
\begin{align}
	P \left( \exists \ell \neq k : d_H(H_\ell, H_k) \leq \frac{md}{2} \right)
	\leq {} & \exp \left( 2 \log M + md \log 2 - \frac{md}{2} \log \left( 1 + \frac{m}{2d} \right) \right)\\
	\leq {} & \exp \left( 2 \log M - \frac{md}{4} \log \left( 1 + \frac{m}{2d} \right) \right) < 1,
\end{align}
since \( d \leq m/2 \), provided we choose \( M \) such that
\begin{equation}
	\label{eq:cm}
	\log M \leq \frac{md}{8} \log \left( 1 + \frac{m}{2d} \right).
\end{equation}
Setting,
\begin{equation}
	\label{eq:dc}
	\log M = \frac{md}{16} \log \left( 1 + \frac{m}{2d} \right).
\end{equation}
we have thus shown that there exists a family fulfilling the conditions of Lemma \ref{lem:vg}.

\section{Proof of LLC upper bounds}
\label{sec:llc-proof}

\subsection{Notation and lemmas}

We start by recalling the notation from Section \ref{sec:upper-bounds-llc}.
We denote the \( i \)th row of \( B^\ast \) by \( b^\ast_i \in \R^{p-1} \), omitting the diagonal element which is assumed to be zero.
From empirical covariances of the performed experiments, for \( i \in [p] \), we obtain estimators \( \hT_i \) for \( T^\ast_i \) and \( \hatt_i \) for \( t^\ast_i \), where \( T^\ast_i b^\ast_i = t^\ast_i \).
Then, we solve the associated \( \ell^1 \)-regularized least squares problem,
\begin{equation}
	\label{eq:eo}
	\hb_i = \argmin_{b} \| \hT_i b - \hatt_i \|_2^2 + \lambda \| b \|_1, \quad i \in [p],
\end{equation}
and assemble its solutions into \( \hB_{\llc} \) as
\begin{align}
	(\hB_{\llc})_{i, :} = (P_i^{\top} \hb_i)^\top, \quad i \in [p],
\end{align}
where \( P_i \in \R^{(p-1) \times p}\) denotes the projection matrix that omits the \( i \)th coordinate.

In particular, the \( T^\ast_i \) are defined row-wise, adding a row $\mathfrak{e}_j^\top \Sigma^{\ast, e} P_i^\top$.
for each experiment \( e \) such that \( i \in \cU_e \) and each entry \( j \in \cJ_e \).
Similarly, the vector \( t^\ast_i \) is defined by appending the corresponding entries \( \Sigma^{\ast, e}_{j, i} \).
Estimators for \( T^ast_i \) and \( t^\ast_i \) in turn are given by row-wise assembling empirical counterparts of the above quantities, so that a generic \( \ell \)th row of \( T^\ast_i \) and \( \ell \)th entry of \( t^\ast_i \) are given by
\begin{equation}
		\hT_{\ell, :} = 
	e_j^\top (J_e + \hat \Sigma^e U_e) P_i^\top \quad
	\text{and} \quad (\hatt_i)_\ell = \hat \Sigma^e_{j, i},
\end{equation}
respectively.

Next, recall the following quantities that enter the rate:
\begin{align}
	\constre(d) = {} & \min_{i \in [p]} \inf_{v \in \cC(d), v \neq 0} \frac{\| T_i^\ast v \|_2}{\| v \|_2}, \label{eq:fl}\\
	\constremax(d) = {} & \max_{i \in [p]} \sup_{\substack{v \in \R^p, v \neq 0,\\ |\supp(v)| \le d}} \frac{\| T_i^\ast v \|_2}{\| v \|_2}, \label{eq:ij}\\
	\constcollone = {} & \max_{i \in [p]} \| (T^\ast_i)^\top \|_{\infty, \infty}
	= \max_{i \in [p]} \max_{j \in [p]} \sum_{k \in [p]} | (T^\ast_i)_{k, j} |,
\end{align}
where
\begin{equation}
	\label{eq:id}
	\cC(d) = \{ v \in \R^p : \text{for all } S \subseteq [p] \text{ with } |S| \le d, \| v_{S^c} \|_1 \le 3 \| v_{S} \|_1 \}.
\end{equation}
We restate Theorem \ref{THM:LLC-RATES} for convenience.

\begin{reptheorem}{THM:LLC-RATES}
	Let   assumptions \ref{assump:matrix} -- \ref{assump:noise} hold and fix $\delta \in (0,1)$. Assume further that
	\begin{align}
		n \gtrsim {} & \left( 1 \vee \frac{p^2}{\constcollone^2 \eta^4} \vee \frac{pd}{(\constremax(d)+1)^2 \eta^4 \constre(d)^4} \right)E \log(e \kappa p / \delta), \label{eq:jd}
	\end{align}	
	Then, the LLC estimator \( \hB_{\llc} \) defined in \eqref{eq:gb} with $\lambda$ chosen such that
	\begin{align}
		\lambda \asymp {} & \constcollone \sqrt{\frac{E \log(e \kappa p/\delta)}{n}},
	\end{align}
satisfies 
	\begin{align}
		\| \hB_\llc - B^\ast \|_F^2
		\lesssim {} & \frac{\constcollone^2}{\constre(d)^4 \eta^4} \frac{p d E \log (e \kappa p/\delta)}{n}\,,
	\end{align}
	with probability at least $1-\delta$.
\end{reptheorem}

The proof relies on the following key lemmas.
Lemma \ref{lem:llc-trace-term-combined} yields control on the stochastic error, while Lemma \ref{lem:lower-bound-control} ensures that the linear system we solve via \( \ell^1 \)-regularization is well-conditioned for that purpose.

\begin{lemma}
	\label{lem:llc-trace-term-combined}
	Under the assumptions of Theorem \ref{THM:LLC-RATES}, writing
	\begin{equation}
		\label{eq:iq}
		\stocherrllc = C \eta^{-2} \sqrt{\frac{E \log(e \kappa p/\delta)}{n}},
	\end{equation}
	for a fixed \( C > 0 \), there is an event \( \eventllc \) such that \( \p(\eventllc) \ge 1 - \delta \), and on \( \eventllc \),
	\begin{equation}
		\label{eq:ip}
		\| \hT_i^\top (\hT_i b^\ast_i - \hatt_i) \|_\infty
		\leq 4 \constcollone \stocherrllc, \quad \text{for all } i \in [p].
	\end{equation}
\end{lemma}

\begin{lemma}
	\label{lem:lower-bound-control}
	Assume that the same hypotheses as in Theorem \ref{THM:LLC-RATES} hold.
	On the same event \( \eventllc \) as in Lemma \ref{lem:llc-trace-term-combined}, we have
	\begin{equation}
		\label{eq:ft}
		\| \hT_i h \|_2^2 \geq \frac{1}{2} \constre(d)^2 \| h \|_2^2, \quad \text{for all } h \in \cC(d), \, i \in [p],
	\end{equation}
	where \( \cC(d) \) is the set of vectors fulfilling the cone condition in \eqref{eq:id}.
\end{lemma}

\subsection{Proof of Theorem \ref{THM:LLC-RATES}}

Since the experiments are completely separating, it follows from \cite{HytEbeHoy12} that 
\begin{equation}
	\label{eq:in}
	\constre(d) \ge \min_{i \in [p]} \sigma_{\min}(T^\ast_i) > 0.
\end{equation}

Fix \( i \in [p] \) and abbreviate \( T^\ast = T^\ast_i \), \( \hT = \hT_i \), \( t^\ast = t^\ast_i \), \( \hatt = \hatt_i \), \( b^\ast = b^\ast_i \), and \( \hb = \hb_i \).

On the event \( \eventllc \) from Lemma \ref{lem:llc-trace-term-combined} the following holds. By definition of $\hat b$, we have
\begin{align*}
	\| \hT \hb - \hatt \|_2^2 + \lambda \| \hb \|_1 \leq \| \hT b^\ast - \hatt \|_2^2 + \lambda \| b^\ast \|_1.
\end{align*}
Set \( h = \hb - b^\ast \) and rearrange to obtain
\begin{align}
	\label{eq:ew}
	\| \hT h \|_2^2 \leq - 2 h^\top \hT^\top (\hT b^\ast - \hatt) + \lambda ( \| b^\ast \|_1 - \| \hb \|_1).
\end{align}
By H\"older's inequality
\begin{align*}
	| h^\top \hT^\top (\hT b^\ast - \hatt) | \leq \| h \|_1 \| \hT^\top (\hT b^\ast - \hatt) \|_\infty.
\end{align*}
By the assumptions on \( n \) and Lemma \ref{lem:llc-trace-term-combined},
\( \| \hT^\top (\hT b^\ast - \hatt) \|_\infty \leq 4 \constcollone \stocherrllc \).
Denote by \( S \) the support of \( b^\ast \).
By triangle inequality and splitting between \( S \) and \( S^c \), we can bound the regularization term by
\begin{align}
	\label{eq:ev}
	\| b^\ast \|_1 - \| \hb \|_1\le \| h_S \|_1 + \| \hb_S \|_1 - \| \hb_S \|_1 - \| \hb_{S^c} \|_1\le \| h_S \|_1 - \| h_{S^c} \|_1.
\end{align}
Add \( \lambda \| h \|_1/2 \) on both sides of \eqref{eq:ew} to obtain
\begin{align}
	\label{eq:ex}
	\| \hT h \|_2^2 + \frac{\lambda}{2} \| h \|_1 \leq \left(4 \constcollone \stocherrllc + \frac{\lambda}{2}\right) \| h \|_1 + \lambda \| h_S \|_1 - \lambda \| h_{S^c} \|_1
	\leq 2 \lambda \| h_S \|_1.
\end{align}
Now, assume \( \lambda \ge 8 \constcollone \stocherrllc \), which by Lemma \ref{lem:llc-trace-term-combined} matches the assumed scaling of
\begin{equation}
	\label{eq:jg}
	\lambda \asymp  \constcollone \sqrt{\frac{E \log(e \kappa p/\delta)}{n}}.
\end{equation}
Together with \eqref{eq:ex}, we get that \( h \) fulfills the cone condition $\| h_{S^c} \|_1 \leq 3 \| h_{S} \|_1.$
In turn, by Lemma \ref{lem:lower-bound-control}, taking into account that the assumptions on \( n \) and \( \stocherrllc \) are fulfilled by assumption, we obtain
\begin{equation*}
	\| \hT h \|_2^2 \ge \frac{1}{2} \constre(d)^2 \| h \|_2^2.
\end{equation*}
Moreover, by the Cauchy-Schwarz inequality \( \| h_S \|_1 \leq \sqrt{d} \| h \|_2 \),
so combined with \eqref{eq:ex} and \eqref{eq:jg}, we have
\begin{equation}
	\label{eq:fb}
	\| h \|_2^2 \leq \frac{16 d}{\constre(d)^4} \lambda^2
	\lesssim \frac{\constcollone^2}{\constre(d)^4 \eta^4} \frac{d E \log(e \kappa p / \delta)}{n}.
\end{equation}
Re-introducing the index \( i \) and summing the above over all \( i \in [p] \),  we get
\begin{equation}
	\label{eq:fx}
	\| \hB_\llc - B^\ast \|_F^2
	\le \frac{\constcollone^2}{\constre(d)^4 \eta^4} \frac{p d E \log(e \kappa p / \delta)}{n}.
\end{equation}

\subsection{Proof of Lemma \ref{lem:llc-trace-term-combined}}


The proof of Lemma \ref{lem:llc-trace-term-combined} consists of two parts that correspond to Lemma \ref{lem:llc-trace-term} and Lemma \ref{lem:llc_stoch_err} below.

Let \( \phi_n > 0 \) and define the events \( \eventllc_1, \eventllc_2, \eventllc_3 \) as follows:
\begin{align}
	\label{eq:fc}
	\eventllc_1 = {} & \left\{ \max_i \| (\hT_i - T^\ast_i) b^\ast_i \|_\infty \leq \stocherrllc \right\}\\
	\eventllc_2 = {} & \left\{ \max_i \| \hT_i - T^\ast_i \|_\infty \leq \stocherrllc \right\}\\
	\eventllc_3 = {} & \left\{ \max_i \| \hatt_i - t^\ast_i \|_\infty \leq \stocherrllc \right\}
\end{align}
Lemma \ref{lem:llc-trace-term} gives an upper bound on \( \| \hT^\top (\hT b^\ast - \hatt) \|_\infty\) in terms of \( \stocherrllc \), while Lemma \ref{lem:llc_stoch_err} gives a high-probability bound on \( \stocherrllc \).
We give the proofs of both of these lemmas after finishing the proof of Lemma \ref{lem:llc-trace-term-combined}.

\begin{lemma}[Trace term estimate]
	\label{lem:llc-trace-term}
	If 
	\begin{equation}
		\label{eq:fw}
		\stocherrllc \leq \constcollone/p,
	\end{equation}
	then on the event \( \eventllc_1 \cap \eventllc_2 \cap \eventllc_3 \),
	\begin{equation}
		\| \hT_i^\top (\hT_i b^\ast_i - \hatt_i) \|_\infty
		\leq 4 \constcollone \stocherrllc, \quad \text{for all } i \in [p].
	\end{equation}
\end{lemma}

\begin{lemma}[Control on stochastic error]
	\label{lem:llc_stoch_err}
	Let \( \delta \in (0, 1) \).
	If $n \gtrsim E \log(e \kappa p/\delta)$, 
	\begin{equation}
		\label{eq:fr}
		n \gtrsim E \log(e \kappa p/\delta),
	\end{equation}
	and we set
	\begin{equation}
		\label{eq:fg}
		\stocherrllc = C \eta^{-2} \sqrt{\frac{E \log(e \kappa p/\delta)}{n}},
	\end{equation}
	for a fixed constant \( C > 0 \), we have that
	\begin{equation}
		\label{eq:fh}
		\p(\eventllc_1 \cap \eventllc_2 \cap \eventllc_3) \geq 1 - \delta.
	\end{equation}
\end{lemma}

Adjusting the constants in the requirement on \( n \) \eqref{eq:jd} in Theorem \ref{THM:LLC-RATES}, we can ensure the requirements \eqref{eq:fw} and \eqref{eq:fr} and thus Lemma \ref{lem:llc-trace-term-combined} follows by setting $\eventllc = \eventllc_1 \cap \eventllc_2 \cap \eventllc_3$
and combining Lemma \ref{lem:llc-trace-term} and Lemma \ref{lem:llc_stoch_err}.

\begin{proof}
	[Proof of Lemma \ref{lem:llc-trace-term}]
	We fix \( i \in [p] \) and, as before, omit it for notational convenience
It holds
	\begin{align}
		\label{eq:fd}
		\| \hT^\top (\hT b^\ast - \hatt) \|_\infty
		= {} & \| (\hT - T^\ast + T^\ast)^\top ((\hT - T^\ast + T^\ast) b^\ast - (\hatt - t^\ast + t^\ast)) \|_\infty \\
		\leq {} & \| (T^\ast)^\top (\hT - T^\ast) b^\ast \|_\infty + \| (T^\ast)^\top (\hatt - t^\ast) \|_\infty\\
		{} & + \| (\hT - T^\ast)^\top (\hT - T^\ast) b^\ast \|_\infty
		+ \| (\hT - T^\ast)^\top (\hatt - t^\ast) \|_\infty \\
		\leq {} &\left(\| (T^\ast)^\top \|_{\infty, \infty} + \| (\hT - T^\ast)^\top \|_{\infty, \infty}\right) \left( \| (\hT - T^\ast) b^\ast \|_\infty + \| t^\ast - \hatt \|_\infty \right),
	\end{align}
	where we used the fact that \( T^\ast b^\ast = t^\ast \) and that for an arbitrary matrix \( A \in \R^{p \times p} \) and vector \( x \in \R^p \), \( \| A x \|_{\infty} \le \| A \|_{\infty, \infty} \| x \|_\infty \).
	Since by the definition of \( \eventllc_2 \),
	\begin{equation}
		\label{eq:gc}
		\| (\hT - T^\ast)^\top \|_{\infty, \infty} = \max_{j \in [p]} \sum_{i \in [p]} | (\hT - T^\ast)_{i, j} | \leq p \max_{i, j} | (\hT - T^\ast)_{i, j} | \leq p \stocherrllc,
	\end{equation}
	we have that combined with the definitons of \( \eventllc_1 \), \( \eventllc_3 \), and \( \constcollone \),
	\begin{align*}
		\| \hT^\top (\hT b^\ast - \hatt) \|_\infty
		\leq {} &\left(\| (T^\ast)^\top \|_{\infty, \infty} + p \stocherrllc \right) 2 \stocherrllc\leq 4 \constcollone \stocherrllc,
	\end{align*}
	if \( \stocherrllc \le \constcollone/p \).
\end{proof}

\begin{proof}
	[Proof of Lemma \ref{lem:llc_stoch_err}]
	For all three events, we write each element of the associated matrices or vectors as a sum over independent sub-exponential random variables and apply Bernstein's inequality, Lemma \ref{lem:bernstein}.

	We start by controlling \( \max_i \| (\hT_i - T^\ast_i) b^\ast_i \|_\infty \).
	Let \( i \in \{1, \dots, p\} \) and \( \ell \in \{1, \dots, m_i \} \).
	The \( \ell \)th row of \( \hT_i - T^\ast_i \) corresponds to an experiment \( e = e(i, \ell) \) such that \( i \in \cU_e \), and an index \( j = j(\ell) \in \cJ_e \), which means we can write
	\begin{align}
		\label{eq:fi}
			\mathfrak{e}_\ell^\top \hat T_i b^\ast_i
			= {} & \mathfrak{e}_j^\top \left( J_e + \hat \Sigma^e U_e \right) P_i^\top b^\ast_i \mathfrak{e}_j^\top \left( J_e + \hat \Sigma^e U_e \right) (B^\ast_{i, :})^\top
	\end{align}
	where we used that \( B^\ast_{i, i} = 0 \), so that \( P_i^\top b^\ast_i = (B^\ast_{i, :})^\top \).
	Moreover, with independent normal random vectors \( Z^e_k \) for \( e = 1, \dots, E \) and \( k = 1, \dots, n/E \), \( \hat \Sigma^e \) is of the form
	\begin{equation}
		\label{eq:ff}
		\hat \Sigma^e = \frac{E}{n} (I - U_e B)^{-1} \sum_{k = 1}^{n/E} Z_k^e (Z_k^e)^\top (I - U_e B)^{-\top},
	\end{equation}
	so that
	\begin{align}
		\label{eq:ee}
		\mathfrak{e}_\ell^\top \hat T_i b^\ast_i
		= {} & \mathfrak{e}_j^\top \left[ J_e + \frac{E}{n} (I - U_e B)^{-1} \sum_{k = 1}^{n/E} Z_k^e (Z_k^e)^\top (I - U_e B)^{-\top} U_e \right] (B^\ast_{i, :})^\top.
	\end{align}
	We proceed to control the \( \ell^2 \) norm of the vectors that are being multiplied with \( Z_k^e \).
	Lemma \ref{lem:boundedness} yields that
	\begin{align}
		\label{eq:ie}
		\| (I - U_e B^\ast)^{-1} U_e B^\ast_{:,i} \|_2
		\le {} & \| (I - U_e B^\ast)^{-1} \|_{\op} \, \| U_e \|_{\op} \, \| B_{:,i}^\ast \|_2 
		\le \eta^{-1},
	\end{align}
	and
	\begin{equation}
		\label{eq:ih}
		\| (I - U_e B^\ast)^{-\top} \mathfrak{e}_j \|_2
		\le \| (I - U_e B^\ast)^{-\top} \|_\op \, \| \mathfrak{e}_j \|_2
		\le \eta^{-1}.
	\end{equation}
	Hence, by Lemma \ref{lem:subgaussian-vector},
	\begin{align}
		B^\ast_{i,:} U_e (I - U_e B^\ast)^{-1} Z_k^e \sim {} & \sg(\eta^{-2}), \text{ and }\\
		\mathfrak{e}_j^\top (I - U_e B)^{-1} Z_k^e \sim {} & \sg(\eta^{-2}),
	\end{align}
	and by Lemma \ref{lem:product-subgaussian}, \( \mathfrak{e}_\ell^\top \hat T_i b^\ast_i \sim \subE(\eta^{-2}) \).
	Now, Bernstein's inequality in Lemma \ref{lem:bernstein} and \( \E[\hT_i] = T^\ast_i \) allows us conclude that for \( t_1 > 0 \),
	\begin{equation}
		\label{eq:ej}
		\p \left( | \mathfrak{e}_\ell^\top (\hat T_i - T^\ast_i) (B^\ast_{i, :})^\top| > t_1 \right)
		\leq 2 \exp \left[ -c_B \left( \left( \frac{\eta^4 n t_1^2}{E} \right) \wedge \left( \frac{\eta^2 n t_1}{E} \right)\right) \right].
	\end{equation}
	A union bound over all indices \( i \in [p] \) and all \( \ell \), taking into account that there are at most \( \kappa p \) rows in every \( T^\ast_i \), then yields
\begin{align}
	\label{eq:em}
	\p \left( \max_{i \in [p], \ell} | \mathfrak{e}_\ell^\top (\hat T_i - T^\ast_i) (B^\ast_{i, :})^\top| > t_1 \right)
	\leq {} & 2 \kappa p^2 \exp \left[ -c_B \left( \left( \frac{\eta^4 n t_1^2}{E} \right) \wedge \left( \frac{\eta^2 n t_1}{E} \right)\right) \right]\\
	\leq {} & \exp \left[ -c_B \left( \left( \frac{\eta^4 n t_1^2}{E} \right) \wedge \left( \frac{\eta^2 n t_1}{E} \right)\right) + 4 \log(e \kappa p) \right].
\end{align}

Similarly, for \( j \in \{1, \dots, p - 1\} \) and any row index \( \ell \),
\begin{align}
	\label{eq:el}
	\mathfrak{e}_\ell^\top \hat T_i \mathfrak{e}_j
	= {} & \mathfrak{e}_\ell^\top \left[ J_e + \frac{E}{n} (I - U_e B)^{-1} \sum_{k = 1}^{n/E} Z_k^e (Z_k^e)^\top (I - U_e B)^{-\top} U_e \right] P_i^\top \mathfrak{e}_j,
\end{align}
and, as before,
\begin{equation}
	\label{eq:ii}
	\| (I - U_e B^\ast)^{-\top} U_e P_i^\top \mathfrak{e}_j \|_2
	\le \| (I - U_e B^\ast)^{-\top} \|_{\op} \, \| U_e \|_{\op} \| P_i^\top \|_{\op} \| \mathfrak{e}_j \|_2
	\le \eta^{-1},
\end{equation}
so that \( \mathfrak{e}_\ell^\top \hat T_i \mathfrak{e}_j \sim \subE(\tilde \eta^{-2}) \).
Hence, by Bernstein’s inequality, for \( t_2 > 0 \),
\begin{align}
	\label{eq:ek}
	\p \left( | \mathfrak{e}_\ell^\top (\hat T_i - T^\ast_i) \mathfrak{e}_j| > t_2 \right)
	\leq 2 \exp \left[-c_B  \left( \left( \frac{\eta^{4} n t_2^2}{E} \right) \wedge \left( \frac{\eta^{2} n t_2}{E} \right)\right) \right].
\end{align}
A union bound over all \( i \in [p] \), \( j \in [p-1] \), and row indices \( \ell \) yields
\begin{align}
	\label{eq:en}
		\p \left( \max_{i, j, \ell} | \mathfrak{e}_\ell^\top (\hat T_i - T^\ast_i) \mathfrak{e}_j | > t_2 \right)
		\leq {} & 2 \kappa p^3 \exp \left[ -c_B \left( \left( \frac{\eta^4 n t^2}{E} \right) \wedge \left( \frac{\eta^2 n t_2}{E} \right)\right) \right]\\
		\leq {} & \exp \left[ -c_B \left( \left( \frac{\eta^4 n t_2^2}{E} \right) \wedge \left( \frac{\eta^2 n t_2}{E} \right)\right) + 6 \log(e \kappa p) \right].
\end{align}
In particular, the union of the two events in \eqref{eq:em} and \eqref{eq:en} occurs with probability at most \( \delta \) if
\begin{equation}
	\label{eq:ep}
	t_1 \wedge t_2 \gtrsim \eta^{-2} \left[ \sqrt{\frac{E \log(e \kappa p/\delta)}{n}} \vee \frac{E \log(e \kappa p/\delta)}{n} \right].
\end{equation}
Taking into account that all \( \hat T_i  \) and \( \hat t_i \) are of the form we investigated in \eqref{eq:el}, we get the claim of the lemma if we choose
\begin{equation}
	\label{eq:fo}
	\stocherrllc = C \eta^{-2} \sqrt{\frac{E \log(e \kappa p/\delta)}{n}},
\end{equation}
for a suitable constant \( C \) and assume $n \gtrsim E \log(e \kappa p/\delta)$.
\end{proof}

\subsection{Proof of Lemma \ref{lem:lower-bound-control}}


	To obtain the result, we employ the following lemma.

	\begin{lemma}[{\cite[Lemma 12]{LohWai11a}}]
		\label{lem:matrix-dev-relax}
		If for a matrix \( \Gamma \in \R^{k \times k}\), \( k \in \mathbb{N} \) and an integer \( s \geq 1 \), it holds that
		\begin{equation}
			\label{eq:fe}
			| v^\top \Gamma v | \leq \delta \quad \text{for all } v \in \R^k, \, \| v \|_0 \leq 2s, \, \| v \|_2 = 1,
		\end{equation}
		then
		\begin{equation}
			\label{eq:fp}
			| v^\top \Gamma v | \leq 27 \delta ( \| v \|_2^2 + \frac{1}{s} \| v \|_1^2) \quad \text{for all } v \in \R^k.
		\end{equation}
	\end{lemma}

	To this end, let \( v \in \R^{p-1} \) be a \( d \) sparse vector with \( \| v \|_2 = 1 \), as well as \( i \in [p] \), and denote by \( \hat T = \hat T_i \), \( T^\ast = T^\ast_i \).
	Then,
	\begin{align}
		\label{eq:eu}
		| \| \hT v \|_2^2 - \| T^\ast v \|_2^2 |
		= {} & | \| (\hT - T^\ast + T^\ast) v \|_2^2 - \| T^\ast v \|_2^2 |\\
		= {} & | \| (\hT - T^\ast) v \|_2^2 + 2 (T^\ast v)^\top (\hT - T^\ast) v + \| T^\ast v \|_2^2 - \| T^\ast v \|_2^2 |\\
		\leq {} & \| (\hT - T^\ast) v \|_2^2 + 2 \| T^\ast v \|_2 \| (\hT - T^\ast) v \|_2,
	\end{align}
	On the one hand, by the definition of \( \constremax(d) \), \eqref{eq:ij},
	\begin{align}
		\label{eq:ik}
		\| T^\ast v \|_{2}
		\leq \constremax(d) \| v \|_2 = \constremax(d).
	\end{align}
	On the other hand, if the event \( \cA_2 \) occurred, then by definition	
	\begin{equation*}
		\| \hT - T^\ast \|_\infty \leq \stocherrllc.
	\end{equation*}
	Thus, denoting by \( S \) the support of \( v \), we can further estimate
	\begin{align}
		\| (\hT - T^\ast) v \|_{2}^2
		= {} & \sum_{i = 1}^{p} \left( \sum_{j \in S} (\hT_{ij} -T^\ast_{ij}) v_j \right)^2\le \sum_{i = 1}^{p} \| (\hT - T^\ast)_{i, S}\|_2^2 \| v \|_2^2\\
		\leq {} & \sum_{i = 1}^{p} d \| (\hT - T^\ast)_{i, S} \|_\infty^2\le p d \| (\hT - T^\ast) \|_\infty^2\le p d \stocherrllc^2,
		\label{eq:il}
	\end{align}
	Combined, \eqref{eq:ik} and \eqref{eq:il} yield
	\begin{align*}
		| \| \hT v \|_2^2 - \| T v \|_2^2 |
		\le {} & ( \constremax(d) + 2 \sqrt{pd} \stocherrllc ) \sqrt{pd} \stocherrllc.
	\end{align*}

Now, let \( h \in \R^{p-1}  \) be a vector that fulfills the cone condition of order \( d \).
That is, there is a set of indices \( S \subseteq [p-1] \) with \( | S | \le d \) such that $\| h_{S^c} \|_1 \le 3 \| h_{S} \|_1.$
This in turn implies that
\begin{equation}
	\label{eq:io}
	\| h \|_1 = \| h_S \|_1 + \| h_{S^c} \|_1
	\le 4 \| h_S \|_1 \le 4 \sqrt{d} \| h \|_2
\end{equation}
by the Cauchy-Schwarz inequality.
By Lemma \ref{lem:matrix-dev-relax}, \eqref{eq:il}, and the definition of \( \constremax(d) \) in \eqref{eq:fl}, we have
\begin{align*}
	\| \hT h \|_2^2
	\geq {} & \| T^\ast h \|_2^2 - | h^\top (\hT - T^\ast)^\top (\hT -T^\ast) h |\\
	\geq {} &  \| T^\ast h \|_2^2 - 27 \left(( \constremax(d) + 2 \sqrt{pd} \stocherrllc ) \sqrt{pd} \stocherrllc \right) ( \| h \|_2^2 + \frac{2}{d} \| h \|_1^2)\\
	\geq {} &  \left(\constre(d)^2 - 432 \left(( \constremax(d) + 2 \sqrt{pd} \stocherrllc ) \sqrt{pd} \stocherrllc \right) \right) \| h \|_2^2.
\end{align*}
Combined, if
\begin{equation}
	\label{eq:gz}
	\stocherrllc \lesssim \frac{\constre(d)^2}{\sqrt{pd}} (\constremax(d) + 1),
\end{equation}
which is guaranteed from the assumptions of Theorem \ref{lem:llc-trace-term-combined}, we get the claim,
\begin{equation}
	\| \hT h \|_2^2 \geq \frac{1}{2} \constre(d)^2 \| h \|_2^2.
\end{equation}


\section{Proof of upper bounds for penalized maximum likelihood estimator}
\label{sec:2-step-rates-proof}

\subsection{Notation and lemmas}

In the following section, we present the proof of Theorem \ref{THM:2-STEP-RATES}, whereas the proofs of several key lemmas are deferred to later sections.

We begin by recalling the estimators and restating Theorem \ref{THM:2-STEP-RATES}.
The loss functions are given by
\begin{equation}
	\label{eq:jp}
	\ell(\Theta, \hat \Sigma) = \tr(\hat \Sigma \Theta) - \log \det (\Theta), \quad
	\mathcal{L}(B) = \cL(B, \hat\Sigma^{1}, \dots, \hat\Sigma^{E}) = \sum_{e \in \cE} \ell(\Theta^e(B), \hat\Sigma^e),
\end{equation}
where
\begin{equation}
	\label{eq:jr}
	\Theta^e(B) = (I - U_e B)^\top (I - U_e B).
\end{equation}
We consider the penalty terms
\begin{equation}
	\label{eq:js}
	\pen_\mathrm{init}(B) = \pen_{\mathrm{init}, \lambda_\mathrm{init}}(B) = \lambda_\mathrm{init} \sum_{e \in \cE} \| \Theta^e(B) \|_1, \quad
	\pen_{\mathrm{loc}}(B) = \pen_{\mathrm{loc}, \lambda_{\mathrm{loc}}}(B) = \lambda_{\mathrm{loc}} \| B \|_1,
\end{equation}
leading to the objective functions
\begin{equation}
	\label{eq:jt}
	\cT_\mathrm{init}(B) = \cL(B, \hat \Sigma^1, \dots, \hat \Sigma^E) + \pen_{\mathrm{init}, \lambda_\mathrm{init}}(B),
\end{equation}
and
\begin{equation}
	\label{eq:jw}
	\cT_\mathrm{loc}(B) = \cL(B, \hat \Sigma^1, \dots, \hat \Sigma^E) + \pen_{\mathrm{loc}, \lambda_{\mathrm{loc}}}(B).
\end{equation}
Finally, the estimators are defined as
\begin{equation}
	\label{eq:ju}
	\hBi \in \argmin_{B \in \cB} \cT_{\mathrm{init}}(B), \quad
	\hB_\loc \in \argmin_{\substack{B \in \cB\\ \| B - \hBi \|_F \leq \localradius}} \cT_{\loc}(B),
\end{equation}
where \( \lambda_\init, \lambda_\loc \) and \( \localradius \) are tuning parameters that are to be determined.




\begin{reptheorem}{THM:2-STEP-RATES}
	Under assumptions \ref{assump:matrix} -- \ref{assump:noise}, if
	\begin{align}
		\label{eq:jm}
		n \gtrsim {} & \left( E^2 \vee \frac{1}{\constop^4} \vee p^2 \right) \frac{p^2 (d+1)^2 E^3}{\constop^4} \log(e p E/\delta)
	\end{align}
	and the parameters for the estimators \( B_\init \) and \( B_\loc \) are chosen such that
	\begin{align}
		\localradius \asymp {} & \frac{1}{\sqrt{E}} \wedge \constop \wedge \frac{1}{\sqrt{p}}, \quad
		\lambda_\init \asymp  \sqrt{\frac{E \log(e p E/\delta)}{n}}, \quad
		\text{ and } \quad \lambda_\loc \asymp \sqrt{\frac{E^2 \log(e p E/\delta)}{n}}
	\end{align}
	then
	\begin{align*}
		\| \hB_\loc - B^\ast \|_F^2 \lesssim \frac{p (d + 1) E^2}{\eta^8 \,  n} \log(pE/\delta),
	\end{align*}
	with probability at least \( 1 - \delta \).
\end{reptheorem}

First, we  present three key lemmas used in the proof of Theorem \ref{THM:2-STEP-RATES}.
Lemma \ref{lem:bound-lh} yields curvature estimates of the likelihood function in terms of the difference of the concentration matrices associated with a candidate matrix \( B \) while Lemma \ref{lem:relate-h} allows us to relate the difference of the concentration matrices to the difference in the underlying matrices, \( B - B^\ast \).
Finally, Lemma \ref{lem:trace-est-combined} gives bounds on a stochastic error term.

To facilitate the presentation, we present the lemmas with the following set of notations and assumptions.
Let \( B \in \R^{p \times p} \) be an arbitrary matrix and \( \cE \) a set of completely separating experiments as in assumption \ref{assump:experiments} with associated matrices \( \{ J_e, U_e \}_{e \in \cE} \).
Moreover, assume that \( B^\ast \in \cB(p, d, \eta) \).
Then, we denote by
\begin{equation}
	\label{eq:ib}
	\Theta^e = {} \Theta^e(B) = (I - U_e B)^\top (I - U_e B),
	\quad
	\Theta^{\ast, e} = {} \Theta^e(B^\ast),
\end{equation}
the concentration matrices associated with \( B \) and \( B^\ast \), respectively, as well as the associated differences between the structure matrices and the concentration matrices by
\begin{equation}
	\label{eq:ic}
	H = {} B - B^\ast,
	\quad
	\Delta^e = {} \Theta^e - \Theta^{\ast, e},
\end{equation}
respectively.
We also abbreviate
\begin{equation}
	\label{eq:jf}
	\| \Delta \|_F^2 = {} \sum_{e \in \cE} \| \Delta^e \|_F^2.
\end{equation}

The first lemma follows from convexity arguments that also appear in \cite{RotBicLev08, LohWai13}.
\begin{lemma}[Lower bounds on Gaussian log-likelihood function, \cite{RotBicLev08, LohWai13}]
	\label{lem:bound-lh}
	With \( \cL \) defined as in \eqref{eq:he}, it holds for any \( B \in \R^{p \times p} \) that
	\begin{equation}
		\label{eq:bq}
		\cL(B) - \cL(B^\ast) \geq \sum_{e \in \cE} \tr((\hat \Sigma^e - \Sigma^{\ast,e}) (\Theta^e(B) - \Theta^{\ast, e})) + (c_1 \| \Delta \|_F \wedge c_1 \| \Delta \|_F^2),
	\end{equation}
	where \( c_1 = 18^{-1} \).
\end{lemma}

\begin{lemma}
	[Upper and lower bounds on \( \| \Delta \|_F \) in terms of \( \| H \|_F \)]
	\label{lem:relate-h}
	If \( B \in \matrixclass \), that is, \( B \) has zero diagonal, we have
	\begin{align}
		\label{eq:bm}
		\| \Delta \|_F^2
		\gtrsim {} & \frac{\eta^4}{pE} \| H \|_F^4,\\
		\label{eq:bo}
		\| \Delta \|_F^2
		\gtrsim {} & \eta^4 \| H \|_F^2 (1 - 2 \eta^{-4} \| H \|_F^4), \\
		\label{eq:bp}
		\| \Delta \|_F^2
		\lesssim {} & E (\| H \|_F^2 + \| H \|_F^4).
	\end{align}
\end{lemma}

\begin{lemma}
	[Trace term estimates]
	\label{lem:trace-est-combined}
	Let \( \delta \in (0, 1) \).
	Denote by \( \phi_n \) and \( \psi_n \) the rates
	\begin{align}
		\psi_n = {} & C \sqrt{\frac{E \log(e p E/\delta)}{n}},\qquad
		\phi_n = {}  C \sqrt{\frac{E^2 \log(e p/\delta)}{n}},
	\end{align}
	for an appropriately chosen constant \( C > 0 \).
	If $n \gtrsim E \log (e p E/\delta)$,
	then with probability at least \( 1 - \delta \), it holds for any \( B \in \R^{p \times p} \) that
	\begin{equation}
		\sum_{e} \tr \left( (\Sigma^{\ast,e} - \hat \Sigma^e)(\Theta^e - \Theta^{\ast,e})\right)
		\leq \psi_n \sum_{e} \| \Delta^e \|_1
	\end{equation}
	and
	\begin{equation}
		\sum_{e} \tr \left( (\Sigma^{\ast,e} - \hat \Sigma^e)(\Theta^e - \Theta^{\ast,e})\right)
		\leq (2 + \| H \|_{\infty, \infty}) \phi_n \| H \|_1.
	\end{equation}
\end{lemma}

For the proof of Theorem \ref{THM:2-STEP-RATES}, we additionally introduce the following abbreviations.
For \( \diamond \in \{\mathrm{init}, \mathrm{loc}\} \), let
\begin{alignat*}{2}
	\Theta^e_\diamond = {} & \Theta^e(\hB_\diamond),
	\quad&
	H_\diamond = {} & \hB_\diamond - B^\ast,\\
	\Delta^e_\diamond = {} & \Theta^e_{\diamond} - \Theta^{\ast, e},	\quad &
	\| \Delta_\diamond \|_F^2 = {} & \sum_{e \in \cE} \| \Delta^e_\diamond \|_F^2.
\end{alignat*}
With this, we are ready to give the proof of \ref{THM:2-STEP-RATES}.

\subsection{Proof of Theorem \ref{THM:2-STEP-RATES}}

\textbf{Proof sketch:}
The proof of Theorem \ref{THM:2-STEP-RATES} is split into two parts.
First, we show that the initialization estimator \( \hat B_{\init} \) performs well enough to allow us to choose \( \localradius \) sufficiently small, so that the log-likelihood in an \( \localradius \)-neighborhood of \( B_{\init} \) has large enough curvature.
Second, we show that locally, \( \hat B_{\loc} \) achieves the desired rate.

Both proofs are based on re-arranging the optimality condition for the penalized log-likelihood, bounding the occurring trace term with high-probability, and exploiting the curvature of the log-likelihood function.




\textbf{Step 1, basic inequality:}
By definition of the estimator \( \hB_{\init} \),
\begin{equation*}
	\hB_\init \in \argmin_{B} \cT_\init(B).
\end{equation*}
Comparing to the ground truth \( B^\ast \) yields the basic inequality
\begin{equation*}
	\cT_\init(\hB_\init) \leq \cT_\init(B^\ast),
\end{equation*}
which implies
\begin{align}
	\label{eq:ar}
	\cL(\hB_\init) - \cL(B^\ast) \leq \pen_\init(B^\ast) - \pen_\init(\hB).
\end{align}
Applying the lower bound on the negative log-likelihood \eqref{eq:bq} in Lemma \ref{lem:bound-lh} then yields
\begin{equation}
	\label{eq:bn}
	c_1 \| \Delta_\init \|_F \wedge c_1 \| \Delta_\init \|_F^2 \le \sum_{e \in \cE} \tr((\Sigma^{\ast,e} - \hat \Sigma^e) (\Theta^e_\init - \Theta^{\ast, e})) + \pen_\init(B^\ast) - \pen_\init(\hB_\init).
\end{equation}

\textbf{Step 1, estimate error term:}
Next, we bound the trace term
\begin{equation}
	\label{eq:hx}
	\sum_{e \in \cE} \tr((\Sigma^{\ast,e} - \hat \Sigma^e) (\Theta^e_\init - \Theta^{\ast, e})),
\end{equation}
with high probability using Lemma \ref{lem:trace-est-combined}.
For the remainder of the proof, we place ourselves on the event of probability at least $1-\delta$ on which the statement of Lemma \ref{lem:trace-est-combined} holds.
Thus, we can estimate the trace term in \eqref{eq:bn} by
\begin{equation*}
		c_1 \| \Delta_\init \|_F \wedge c_1 \| \Delta_\init \|_F^2 \leq \psi_n \sum_{e} \| \Delta^e_\init \|_1 + \pen_\init(B^\ast) - \pen_\init(\hB_\init).
\end{equation*}
Denoting the support of \( \Theta^{\ast, e} = (I - U_e B^\ast)^\top (I - U_e B^\ast) \) by \( S^e_\init \), we have
\begin{align*}
	\sum_{e} \| \Delta^e_\init \|_1
	= \sum_{e, i, j} | (\Delta^e_\init)_{i, j} |
	= {} & \sum_{e} (\| (\Delta^e_\init)_{S^e_\init} \|_1 + \| (\Delta^e_\init)_{(S_2^e)^c} \|_1).
\end{align*}
Moreover, by triangle inequality,
\begin{align*}
	\| \Theta^{\ast, e} \|_1 - \| \Theta^e_\init \|_1
	\leq {} & \| (\Delta^e_\init)_{S^e_\init} \|_1 - \| (\Delta^e_\init)_{(S_2^e)^c} \|_1.
\end{align*}
Combined with the definition of the penalization term,
\begin{equation*}
	\pen_\init(B) = \lambda_\init \sum_{e \in \cE} \| \Theta^e(B) \|_1.
\end{equation*}
Now, assume \( \lambda_\init \ge \psi_n \), which matches the assumed scaling of \( \lambda_\init \) to obtain
\begin{align*}
	c_1 \| \Delta_\init \|_F \wedge c_1 \| \Delta_\init \|_F^2 
	\leq {} & 2 \lambda_\init \sum_{e} \| (\Delta^e_\init)_{S^e_\init} \|_1.
\end{align*}

Note that we can control the size of the support \( | S^e_\init | \) by the  in-degree of  \( B^\ast \).
Namely, if we decompose
\begin{equation*}
	\Theta^{\ast, e} = (I - U_e B^\ast)^\top (I - U_e B^\ast) = \sum_{k = 1}^{p} (I - U_e B^\ast)_{k, :}^\top (I - U_e B^\ast)_{k, :},
\end{equation*}
which is a sum over the outer product of \( d + 1 \) sparse vectors by the assumption that the in-degree of the underlying graph is bounded by \( d \),
and hence
\begin{equation*}
	| S^e_\init | \leq p (d + 1)^2.
\end{equation*}
In turn, Hölder's inequality yields
\begin{align}
	\label{eq:br}
	2 \lambda_\init \sum_{e} \| \Delta^e_{S^e_\init} \|_1
	\leq {} & 2 \lambda_\init \sqrt{p (d+1)^2 E} \| \Delta \|_F.
\end{align}

\textbf{Bounds on \( \| \Delta_\init \|_F^2 \):}
If \( \| \Delta_\init \|_F \geq 1 \), by \eqref{eq:bn} and \eqref{eq:br}, we have
\begin{align*}
	\| \Delta \|_F
	\leq {} & 2 \frac{\lambda_\init}{c_1} \sqrt{p (d + 1)^2 E} \| \Delta \|_F,
\end{align*}
which yields a contradiction if
\begin{equation*}
	\lambda_\init \leq \frac{c_1}{4 \sqrt{p (d+1)^2 E}}.
\end{equation*}
By the assumption that \( \lambda_\init \asymp \psi_n \) and the value of \( \psi_n \) in \ref{lem:trace-est-combined}, this  holds if
\begin{equation*}
	n \gtrsim p (d+1)^2 E^2 \log(e p E/\delta).
\end{equation*}

If \( \| \Delta_\init \|_F \leq 1 \), again by combining \eqref{eq:bn} and \eqref{eq:br}, we have
\begin{align*}
	\| \Delta_\init \|_F^2
	\leq {} & 2 \frac{\lambda_\init}{c_1} \sqrt{p (d+1)^2 E} \| \Delta_\init \|_F.
\end{align*}
Dividing by \( \| \Delta_\init \|_F \) and squaring then implies
\begin{equation*}
	\| \Delta_\init \|_F^2
	\leq 4 \frac{\lambda_\init^2}{c_1^2} \sqrt{p (d+1)^2 E}.
\end{equation*}
By Lemma \ref{lem:trace-est-combined} and the choice of \( \lambda_\init \asymp \psi_n \), this leads to
\begin{equation}
	\label{eq:bs}
	\| \Delta_\init \|_F^2
	\lesssim \frac{p (d+1)^2 E^2 \log(e p E/\delta)}{n}.
\end{equation}

\textbf{Bounds on \( \| H_\init \|_F \):}
In order to relate \( \| \Delta_{\init} \|_F \) to \( \| H_{\init} \|_F \) we appeal to Lemma \ref{lem:relate-h}.
If \( n  \) is large enough for \eqref{eq:bs} to hold, then by the lower bound \eqref{eq:bm} in Lemma \ref{lem:relate-h},
\begin{equation}
	\label{eq:bt}
	\frac{\eta^4}{pE} \| H_\init \|_F^4 \lesssim \| \Delta_\init \|_F^2 \lesssim \frac{p (d+1)^2 E^2 \log(e p E/\delta)}{n},
\end{equation}
and hence
\begin{equation}
	\label{eq:bu}
	\| H_\init \|_F^4 \lesssim \frac{p^2 (d+1)^2 E^3 \log(e p E/\delta)}{\eta^4 n},
\end{equation}
which concludes the analysis for the initialization estimator.

\textbf{Step 2, basic inequality:}
We have
\begin{equation*}
	\hB_\loc \in \argmin_{\| B - \hB_\init \|_F \leq \localradius} \cT_\loc(B).
\end{equation*}
Suppose \( \localradius \geq \| H_\init \|_F \), which we  achieve by \eqref{eq:bu} and choosing \( n \) large enough later, once \( \localradius \) has been chosen.
Then, comparing to the ground truth \( B^\ast \) yields the basic inequality
\begin{equation*}
	\cT_\loc(\hB_\loc) \leq \cT_\loc(B^\ast),
\end{equation*}
which implies
\begin{align}
	\cL(\hB_\loc) - \cL(B^\ast) \leq \pen_\loc(B^\ast) - \pen_\loc(\hB_\loc).
\end{align}
Applying the lower bound on the negative log-likelihood \eqref{eq:bq} in Lemma \ref{lem:bound-lh} yields
\begin{equation}
	c_1 \| \Delta_\loc \|_F \wedge c_1 \| \Delta_\loc \|_F^2 \leq \sum_{e \in \cE} \tr((\Sigma^{\ast,e} - \hat \Sigma^e) (\Theta_\loc^e - \Theta^{\ast, e})) + \pen_\loc(B^\ast) - \pen_\loc(\hB_\loc).
\end{equation}

\textbf{Step 2, estimate error term:}
We resort to Lemma \ref{lem:trace-est-combined}, this time in the form of \eqref{eq:bx}, which yields
\begin{equation}
	\label{eq:ca}
	c_1 \| \Delta_\loc \|_F \wedge c_1 \| \Delta_\loc \|_F^2 \leq (2 + \| H_\loc \|_{\infty, \infty}) \phi_n \| H_\loc \|_1 + \pen_\loc(B^\ast) - \pen_\loc(\hB_\loc).
\end{equation}


First, we want to ensure \( \| \Delta_\loc \|_F \leq 1 \).
The upper bound on \( \Delta \) in Lemma \ref{lem:relate-h}, \eqref{eq:bp}, achieves this if
\begin{equation*}
	\| H_\loc \|_F \le \frac{c_2}{\sqrt{E}},
\end{equation*}
for a small enough constant \( c_2 \le 1 \).
By the triangle inequality and \eqref{eq:bu}, this is true if
\begin{align*}
	\localradius \le {} & \frac{1}{2} \frac{c_2}{\sqrt{E}}, \quad \text{and} \quad  n \gtrsim \frac{p^2 (d + 1)^2 E^5}{\eta^4} \log(e p E/\delta).
\end{align*}

Second, since we want the bound \eqref{eq:bo} to be effective within the ball \( \| B - \hB_\init \|_F \leq \localradius \) over which the optimization in step 2 is constrained, we choose \( n \) large enough to guarantee
\begin{equation*}
	\| H_\loc \|_F^4 \leq \frac{\eta^4}{4}.
\end{equation*}
This again follows from triangle inequality and \eqref{eq:bu} if
\begin{align}
	\label{eq:ia}
	\localradius \le {} & \frac{\eta}{2 \sqrt{2}}, \quad \text{and}\quad
	n \gtrsim {}  \frac{p^2 (d+1)^2 E^3}{\eta^8} \log(e p E/\delta).
\end{align}


Third, to control the \( \| H_\loc \|_{\infty, \infty} \) term in \eqref{eq:ca}, observe that
\begin{align*}
	\| H_\loc \|_{\infty, \infty} \leq \sqrt{p} \| H_\loc \|_F
\end{align*}
by Hölder inequality.
To guarantee \( \| H_\loc \|_{\infty, \infty} \leq 2 \), it is enough to ask for \( \| H_\init \|_F^4 \leq 1/p^2 \) and \( \localradius \leq 1/\sqrt{p} \) by triangle inequality.
By \eqref{eq:bu}, the former is be satisfied if
\begin{equation*}
	n \gtrsim \frac{p^4 (d+1)^2 E^3}{\eta^4} \log(e pE/\delta).
\end{equation*}

Combined, in addition to the assumptions made in step 1, if
\begin{align*}
	\localradius \leq {} & c_3 \left[ \frac{1}{\sqrt{E}} \wedge \eta \wedge \frac{1}{\sqrt{p}} \right] \quad \text{ and } \quad
	n \gtrsim \left[ E^2 \vee \frac{1}{\eta^4} \vee p^2 \right] \frac{p^2 (d+1)^2 E^3}{\eta^4} \log(e p E/\delta),
\end{align*}
then
\begin{equation}
	\label{eq:bz}
	\| \Delta_\loc \|_F \leq 1, \quad \| H_\loc \|_{\infty, \infty} \leq 2 \quad \text{and} \quad \| H_\loc \|_F^4 \leq \frac{\eta^4}{4}.
\end{equation}
In turn, from \eqref{eq:bo}, we obtain
\begin{equation*}
	\| H_\loc \|_F^2 \leq 2 \eta^4 \| \Delta_\loc \|_F^2.
\end{equation*}
Writing \( S := \supp (I - B^\ast) \), we then see that
\begin{equation*}
	\| H_\loc \|_1 = \| (H_\loc)_S \|_1 + \| (H_\loc)_{S^c} \|_1,
\end{equation*}
and by triangle inequality,
\begin{equation*}
	\| B^\ast \|_1 - \| \hB_\loc \|_1 \leq \| (H_\loc)_{S} \|_1 - \| (H_\loc)_{S^c} \|_1.
\end{equation*}
Together with \eqref{eq:ca} and observing that we can assume \( \lambda_\loc \ge 4 \phi_n \), it follows that
\begin{equation*}
	\| H_\loc \|_F^2 \lesssim \frac{\lambda_\loc}{\eta^4} \| (H_\loc)_S \|_1.
\end{equation*}
Applying the Cauchy-Schwarz inequality gives
\begin{equation*}
	\| H_\loc \|_F^2 \lesssim \frac{\lambda_\loc}{\eta^4} \sqrt{| S |} \| H_\loc \|_F.
\end{equation*}
Finally, we divide by \( \| H_\loc \|_F \), take squares, observe that $	| S | \leq p (d+1)$
use \( \lambda_\loc \asymp \phi_n\), and plug in the value of \( \phi_n \) in Lemma \ref{lem:trace-est-combined} to obtain
\begin{equation*}
	\| H_\loc \|_F^2 \lesssim \frac{p (d + 1) E^2}{\eta^8 \, n} \log(pE/\delta),
\end{equation*}
which concludes the proof.


\subsection{Proof of Lemma \ref{lem:bound-lh}}


Let \( R_1 > 0 \) and recall the notation
\begin{equation*}
	\ell(\Theta, \Sigma) = \tr(\Sigma \Theta) - \log \det (\Theta)
\end{equation*}
for the negative log-likelihood of a centered multivariate Gaussian distribution.
Let \( \Theta^\ast, \hat \Sigma \) be a positive definite matrix and a positive semi-definite matrix, respectively, and set \( \Sigma^\ast = (\Theta^\ast)^{-1} \).
Noting that the first derivative of \( \Theta \mapsto - \log \det \Theta \) is \( -\Theta^{-1} \) and the second derivative is \( \Theta^{-1} \otimes \Theta^{-1} \), by computing a Taylor expansion of \( \ell \) with differential remainder term about \( \Theta^\ast \), we have that
\begin{align}
	\label{eq:al}
	\ell(\Theta, \hat \Sigma) - \ell(\Theta^\ast, \hat \Sigma) = \tr(\hat \Sigma (\Theta - \Theta^\ast)) - \tr(\Sigma^\ast (\Theta - \Theta^\ast))+ \frac{1}{2} \tr(\tilde{\Theta}^{-1} (\Theta - \Theta^\ast) \tilde{\Theta}^{-1} (\Theta - \Theta^\ast))
\end{align}
for some \( t \in [0,1] \) and \( \tilde{\Theta} = \Theta^\ast + t (\Theta - \Theta^\ast) \).

Denote the matrix square root of \( \tilde{\Theta}^{-1} \) by \( \tilde{\Theta}^{-1/2} \).
Then, we can further lower bound the quadratic term by
\begin{align}
	\tr(\tilde{\Theta}^{-1} (\Theta - \Theta^\ast) \tilde{\Theta}^{-1} (\Theta - \Theta^\ast))
	= {} & \tr(\tilde{\Theta}^{-1/2} (\Theta - \Theta^\ast) \tilde{\Theta}^{-1/2} \tilde{\Theta}^{-1/2}(\Theta - \Theta^\ast) \tilde{\Theta}^{-1/2}) \nonumber \\
	= {} & \| \tilde{\Theta}^{-1/2} (\Theta - \Theta^\ast) \tilde{\Theta}^{-1/2} \|_F^2 \nonumber\\
	\geq {} & \lambda_{\mathrm{min}}(\tilde{\Theta}^{-1/2})^4 \| \Theta - \Theta^\ast \|_F^2.
	\label{eq:am}
\end{align}
By the spectral theorem, we can express the smallest eigenvalue of \( \tilde{\Theta}^{-1/2} \) in terms of the largest eigenvalue of \( \tilde{\Theta} \),
\begin{equation*}
	\lambda_{\mathrm{min}}(\tilde{\Theta}^{-1/2}) = (\lambda_{\mathrm{max}}(\tilde{\Theta}))^{-1/2}.
\end{equation*}
Now, recall 
\begin{align*}
	\mathcal{L}(B) = \sum_{e \in \cE} \ell(\Theta^e(B), \hat \Sigma^e),
\end{align*}
where
\begin{equation*}
	\Theta^e = \Theta^e(B) = (I - U_e B)^\top (I - U_e B),
\end{equation*}
and introduce
\begin{equation*}
	\tilde{\Delta}^e := \tilde{\Theta}^e - \Theta^{\ast, e}
\end{equation*}
and denote by 
\begin{equation*}
	\| \tilde{\Delta} \|_F = \sqrt{ \sum_{e \in \cE} \| \tilde{\Delta}^e \|_F^2}
\end{equation*} 
the Frobenius norm of the collection of \( \tilde{\Delta}^e \) when viewed as a tensor.
We now apply the expansion \eqref{eq:al} and the estimate \eqref{eq:am} to each of the summands, distinguishing two cases.

First, if \( \| \Delta \|_F \leq R_1 \), then also \( \| \Theta^e - \Theta^{\ast, e} \|_F \leq R_1 \) for all \( e \in \cE \) and we get
\begin{align*}
	\lambda_{\mathrm{max}}(\tilde{\Theta}^e)
	= {} & \| \tilde{\Theta}^e \|_\mathrm{op} = \| \Theta^{\ast,e} + \tilde{\Delta}^e \|_{\mathrm{op}}
	\leq \| \Theta^{\ast,e} \|_{\mathrm{op}} + \| \tilde{\Delta}^e \|_{\mathrm{op}}\\
	\leq {} & \| \Theta^{\ast,e} \|_{\mathrm{op}} + \| \tilde{\Delta}^e \|_{F}
	\leq \| \Theta^{\ast,e} \|_{\mathrm{op}} + \| \Delta^e \|_{F}
	\leq \| \Theta^{\ast,e} \|_\mathrm{op} + R_1.
\end{align*}
Therefore, from \eqref{eq:am} we get a lower bound of the form
\begin{equation}
	\label{eq:hf}
	\cL(B) - \cL(B^\ast) \geq \sum_{e \in \cE} \tr((\hat \Sigma^e - \Sigma^{\ast,e}) (\Theta^e - \Theta^{\ast, e})) + c_1 \| \Delta \|_F^2,
\end{equation}
with \( c_1 = (\max_{e \in \cE} \| \Theta^{\ast,e} \|_\mathrm{op} + R_1)^{-2}/2 \).

Second, if \( \| \Delta \|_F > R_1 \), we can leverage the convexity of \( \Theta \mapsto - \log \det \Theta \) to again obtain lower bounds.
Define \( g(s) \) for \( s \in [0,1] \) by
\begin{equation*}
	g(s) = \sum_{e \in \cE} \left[ \ell(\Theta^{\ast,e} + s \Delta^e, \hat \Sigma^e) - \ell(\Theta^{\ast, e}, \hat \Sigma^e) \right].
\end{equation*}
Since \( \ell \) is convex in \( \Theta \), \( g \) is convex in \( s \), and we obtain
\begin{equation*}
	\frac{g(1) - g(0)}{1} \geq \frac{g(s) - g(0)}{s}, \quad \text{for all } s \in (0, 1].
\end{equation*}
Plugging in \( t = R_1/\| \Delta \|_F \), we are in the first case that was discussed and can appeal to \eqref{eq:hf}, which yields
\begin{align*}
	\sum_{e \in \cE} \left[ \ell(\Theta^{e}(B), \hat \Sigma^e) - \ell(\Theta^{\ast,e}, \hat \Sigma^e) \right]
	\ge {} & \frac{\| \Delta \|_F}{R_1} \sum_{e \in \cE} \left( \ell(\Theta^{\ast,e} +\frac{R_1}{\| \Delta \|_F} \Delta^e, \hat \Sigma^e) - \ell(\Theta^{\ast,e}, \hat \Sigma^e) \right)\\
	\geq {} & \frac{\| \Delta \|_F}{R_1} \sum_{e \in \cE} \left( \tr((\hat \Sigma^e - \Sigma^{\ast,e}) \frac{R_1}{\| \Delta \|_F} \Delta^e) + R_1^2 c_1 \right)\\
	= {} & \sum_{e \in \cE} \tr((\hat \Sigma^e - \Sigma^{\ast,e}) \Delta^e) + R_1 c_1 \| \Delta \|_F.
\end{align*}

Combined, we get
\begin{equation*}
	\cL(B) - \cL(B^\ast) \geq \sum_{e \in \cE} \tr((\hat \Sigma^e - \Sigma^{\ast,e}) (\Theta^e - \Theta^{\ast, e})) + (R_1 c_1 \| \Delta \|_F \wedge c_1 \| \Delta \|_F^2) \qedhere
\end{equation*}
Finally, setting \( R_1 = 1 \) and observing that \( \max_{e} \| \Theta^{\ast, e} \|_{\op} \le 2 \) by Lemma \ref{lem:boundedness} yields the claim.



\subsection{Proof of Lemma \ref{lem:relate-h}}

In this section, we abbreviate
\begin{equation}
	\label{eq:hh}
	H = B - B^\ast, \quad A = (I - B^\ast)^{-1}, \quad A_e = (I - U_e B^\ast)^{-1}.
\end{equation}
We  also need the following linear transformation of \( H \), which we denote by \( G \),
\begin{equation}
	\label{eq:hg}
	G = H A = H (I - B^\ast)^{-1}.
\end{equation}

First, we give a lemma that allows us to estimate the Frobenius norm of \( G \) by its off-diagonal elements.

\begin{lemma}
	\label{lem:off-diag}
	Let \( B \in \R^{p \times p} \) and denote by \( H \) and \( G \) the matrices in \eqref{eq:hh} and \eqref{eq:hg}, respectively.
	Moreover, write \( G_D \) and \( G_{D^c} \) for the restriction of the matrix \( G \) to its diagonal indices and off-diagonal elements, respectively.
	If $\| B^\ast\|_{\mathrm{op}}<1$ and \( H_D = 0 \), then
	\begin{equation*}
		\| G \|_F^2 \le 2 \| G_{D^c} \|_F^2.
	\end{equation*}
\end{lemma}

\begin{proof}
	By the definition of \( G \),	we know that \( H = G A^{-1} = G (I - B^\ast) \).
	The restriction \( H_D = 0 \) implies
	\begin{equation*}
		\sum_{k = 1}^p G_{ik} (I - B^\ast)_{ki}	= 0, \quad \text{for all } i \in [p].
	\end{equation*}
	Since \( B^\ast \) has zero diagonal, for each \( i \in [p] \), we can solve for \( G_{i i} \) and obtain
	\begin{equation*}
		G_{i i} = \sum_{k \neq i} G_{i k} B^\ast_{k i}.
	\end{equation*}
	By the Cauchy-Schwarz inequality,
	\begin{equation*}
		G_{i i}^2 \leq \left( \sum_{k \neq i} G_{i k}^2 \right) \left( \sum_{k \neq i} (B^\ast_{k i})^2 \right).
	\end{equation*}
	Finally, summing over all \( i \) gives
	\begin{equation*}
		\| G_D \|_F^2 = \sum_{i} G_{i i}^2 \leq \left[ \max_i \sum_{k} (B^\ast_{k i})^2 \right] \sum_{i} \sum_{k \neq i} G_{i k}^2.
	\end{equation*}
	Since $\|B^\ast\|_{\mathrm{op}}<1$ and by Lemma \ref{lem:boundedness},
	\begin{equation*}
		\max_i \sum_{k} (B^\ast_{k i})^2 \le 1,
	\end{equation*}
	we have the claim,
	\begin{equation*}
		\| G \|_F^2 = \| G_{D^c} \|_F^2 + \| G_{D} \|_F^2 \leq 2 \| G_{D^c} \|_F^2. \qedhere
	\end{equation*}
\end{proof}

With this, we proceed to prove Lemma \ref{lem:relate-h}.

	To start, let \( e \in \cE \).
	We have
\begin{align}
	\Delta^e = {} & \Theta^e - \Theta^{\ast,e}\\
	= {} & (I - U_e \hB)^\top (I - U_e \hB) - (I - U_e B^\ast)^\top (I - U_e B^\ast) \nonumber \\
	= {} & (I - U_e (B^\ast + H))^\top (I - U_e (B^\ast + H)) - (I - U_e B^\ast)^\top (I - U_e B^\ast) \nonumber \\
	\label{eq:ak}
	= {} & - (U_e H)^\top A_e^{-1} - A_e^{-\top} (U_e H) + (U_e H)^\top (U_e H).
\end{align}
Since \( U_e^\top U_e = U_e^2 = U_e \), we can simplify the terms in the above expression as
\begin{equation}
	\label{eq:hj}
	(I - U_e B^\ast)^{\top} (U_e H) = U_e H - (B^\ast)^\top U_e^\top U_e H = (I - B^\ast)^\top (U_e H),
\end{equation}
which leads to
\begin{align}
	\label{eq:bj}
	\Delta^e
	= {} & -(U_e H)^\top A^{-1} - A^{-\top} (U_e H) + A^{-\top} A^\top (U_e H)^\top (U_e H) A A^{-1}\\
	= {} & A^{-\top} \left( -A^\top (U_e H)^\top - (U_e H) A + A^\top (U_e H)^\top (U_e H) A \right) A^{-1} \nonumber \\
	= {} & A^{-\top} \left( -(U_e H A)^\top - (U_e H A) + (U_e H A)^\top (U_e H A) \right) A^{-1}.
\end{align}
Hence, by Lemma \ref{lem:boundedness},
\begin{align}
	\| \Delta^e \|_F
	\geq {} & \sigma_{\mathrm{min}}^2(A^{-1}) \| -(U_e H A)^\top - (U_e H A) + (U_e H A)^\top (U_e H A) \|_F\\
	\geq {} & \eta^2 \| -(U_e H A)^\top - (U_e H A) + (U_e H A)^\top (U_e H A) \|_F.
	\label{eq:hi}
\end{align}
Write \( G = HA \).

First, to further lower bound the above expression, consider the diagonal of the \( \cJ_e \times \cJ_e \) block of the matrix. There, we have \( (U_e G)_{\cJ_e, \cJ_e} = 0 \) and \( (U_e G)^\top_{\cJ_e, \cJ_e} = 0 \), and thus
\begin{align*}
	\| -(U_e G)^\top - (U_e G) + (U_e G)^\top (U_e G) \|_F^2
	\geq {} & \sum_{i \in \cJ_e} \left( \sum_{u \in \cU_e} G_{u, i}^2 \right)^2 
	\geq  \frac{1}{p} \left( \sum_{i \in \cJ_e} \sum_{u \in \cU_e} G_{u, i}^2 \right)^2,
\end{align*}
where we used \( \| h \|_1 \leq \sqrt{p} \| h \|_2  \) for a vector \( h \in \R^p \), which follows from H\"older's inquality.
Summing over the experiments \( \mathcal{E} \), together with the assumption of \( \mathcal{E} \) being completely separating, Hölder's inequality, Lemma \ref{lem:off-diag}, and Lemma \ref{lem:boundedness}, we get
\begin{align}
	\| \Delta \|_F^2
	\geq {} &
		\eta^4 \sum_{e \in \cE} \| (U_e G)^\top + (U_e G) + (U_e G)^\top (U_e G) \|_F^2\\
	\label{eq:ax}
		\geq {} & \frac{\eta^4}{pE} \| G_{D^c} \|_F^4
		\gtrsim \frac{\eta^4}{4 pE} \| G \|_F^4
		\gtrsim \frac{\eta^4}{pE} \| H \|_F^4.
\end{align}

Second, focusing on the \( \cU_e \times \cJ_e \) block of the matrix
\begin{equation}
	\label{eq:kb}
	-(U_e G)^\top - (U_e G) + (U_e G)^\top (U_e G),
\end{equation}
we note that by the Cauchy-Schwarz inequality and the elementary inequality \( (a+b)^2 \ge \frac{1}{2}a^2 - b^2 \) for \( a, b \in \R \),
\begin{align*}
	\leadeq{\| -(U_e G)^\top - (U_e G) + (U_e G)^\top (U_e G) \|_F^2}\\
	\ge {} & \sum_{i \in \cU_e}  \sum_{j \in \cJ_e} \left( - G_{i, j} + \sum_{k \in \cU_e} G_{k, i} G_{k, j} \right)^2\\
	\geq {} & \sum_{i \in \cU_e}  \sum_{j \in \cJ_e} \left[ \frac{1}{2} G_{i, j}^2 - \left( \sum_{k \in \cU_e} G_{k, i} G_{k, j} \right)^2 \right]\\
	\geq {} & \sum_{i \in \cU_e}  \sum_{j \in \cJ_e} \left[ \frac{1}{2} G_{i, j}^2 - \left( \sum_{k \in \cU_e} G_{k, i}^2\right) \left( \sum_{k \in \cU_e} G_{k, j}^2 \right) \right]\\
	\geq {} & \left( \sum_{j \in \cJ_e} \sum_{i \in \cU_e} G_{i, j}^2 \right)\left(\frac{1}{2} - \| G \|_F^4\right).
\end{align*}
Summing over the experiments, taking into account that by symmetry the same estimate holds for the \( \cJ_e \times \cU_e \) block, and bounding maximum and minimum singular values by Lemma \ref{lem:boundedness}, we obtain a lower bound of
\begin{align}
	\| \Delta \|_F^2 \geq \eta^4 \| G_{D^c} \|_F^2 (1 - 2 \| G \|_F^4)
	\gtrsim {} & \eta^4 \| G \|_F^2 (1 - 2 \| G \|_F^4) \nonumber \\
	\label{eq:ba}
	\gtrsim {} & \eta^4 \| H \|_F^2 (1 - 2 \| G \|_F^4) \\
	\geq {} & \eta^4 \| H \|_F^2 (1 - 2 \eta^{-4} \| H \|_F^4).
\end{align}

Finally, we can upper bound \( \| \Delta \|_F \) in terms of \( \| H \|_F \), starting from \eqref{eq:ak}, by
\begin{align}
	\| \Delta \|_F^2
	&= \sum_{e \in \cE}  \| - (U_e H)^\top A^{-1} - A^{-\top} (U_e H) + (U_e H)^\top (U_e H) \|_F^2\\
	&\lesssim  \sum_{e \in \cE} (\| H \|_F^2 + \| H \|_F^4) \le E (\| H \|_F^2 + \| H \|_F^4).
\end{align}

\subsection{Proof of Lemma \ref{lem:trace-est-combined}}

In this section, we abbreviate
\begin{align*}
	T_1 := {} & \max_{e \in \cE} \left\| (\Sigma^{\ast, e} - \hat \Sigma^e) \right\|_\infty, \\
	T_2 := {} & \left\| \sum_{e \in \cE} U_e A^{-1} (\Sigma^{\ast,e} - \hat \Sigma^e) \right\|_\infty, \\
	T_3 := {} & \max_{k} \left\| \sum_{e \in \cE} \1_{k \in U_e} (\Sigma^{\ast, e} - \hat \Sigma^e) \right\|_\infty
\end{align*}
and introduce the events
\begin{align}
	\label{eq:hy}
	\cA_1 = \{ T_1 \leq \psi_n \}, \quad	\cA_2 = \{ T_2 \leq \phi_n \}, \quad \cA_3 = \{ T_3 \leq \phi_n \}, \quad \cA = \cA_1 \cap \cA_2 \cap \cA_3,
\end{align}
where the terms \( T_1, T_2, T_3 \) are upper bounded by rates \( \phi_n \) and \( \psi_n \) to be made precise in Lemma \ref{lem:stoch-err}, while Lemma \ref{lem:trace-est} shows how \( \phi_n \) and \( \psi_n \) can be used to estimate the trace term.
\begin{lemma}[Trace term estimates]
	\label{lem:trace-est}
		On the event \( \cA_1 \cap \cA_2 \cap \cA_3 \), it holds for any \( B \in \R^{p \times p} \) that
	\begin{equation}
		\label{eq:by}
		\sum_{e} \tr \left( (\Sigma^{\ast,e} - \hat \Sigma^e)(\Theta^e - \Theta^{\ast,e})\right)
		\leq \psi_n \sum_{e} \| \Delta^e \|_1
	\end{equation}
	and
	\begin{equation}
		\label{eq:bx}
		\sum_{e} \tr \left( (\Sigma^{\ast,e} - \hat \Sigma^e)(\Theta^e - \Theta^{\ast,e})\right)
		\leq (2 + \| H \|_{\infty, \infty}) \phi_n \| H \|_1.
	\end{equation}
\end{lemma}

\begin{lemma}[Control on stochastic error]
	\label{lem:stoch-err}
	Let \( \delta \in (0, 1) \).
	There exists an absolute constant \( C \) such that if
	\begin{align}
		\label{eq:bv}
		\psi_n = {} & C \sqrt{\frac{E \log(e p E/\delta)}{n}},\qquad 
		\phi_n = {}  C \sqrt{\frac{E^2 \log(e p/\delta)}{n}},
	\end{align}
	and
	\begin{equation*}
		n \geq C E \log (e p E/\delta),
	\end{equation*}
	then with probability at least \( 1 - \delta \), it holds that
	\begin{equation*}
		\p(\cA_1 \cap \cA_2 \cap \cA_3) \geq 1 - \delta,
	\end{equation*}
	where \( \cA_1, \cA_2, \cA_3 \) are defined as in \eqref{eq:hy}.
\end{lemma}

Combined, Lemmas \ref{lem:trace-est} and \ref{lem:stoch-err} yield Lemma \ref{lem:trace-est-combined}.



\begin{proof}
	[{Proof of Lemma \ref{lem:trace-est}}]
	First, by Hölder's inequality,
\begin{align*}
	\sum_{e \in \cE} \tr \left( (\Sigma^{\ast,e} - \hat \Sigma^e)(\Theta^e - \Theta^{\ast,e})\right)
	\leq {} & \max_{e, i, j} \left| (\Sigma^{\ast, e} - \Sigma^e)_{i, j} \right| \sum_{e, i, j} | \Delta^e_{i, j} |.
\end{align*}
Identifying the first term as \( T_1 \) and using the estimate \( T_1 \leq \psi_n \) yields \eqref{eq:by}.

Second, by the same calculation that led to \eqref{eq:hi}, we decompose the trace term as
\begin{align}
	\leadeq{\sum_{e \in \cE} \tr \left( (\Sigma^{\ast,e} - \hat \Sigma^e)(\Theta^e - \Theta^{\ast,e})\right)}\\
	= {} & \sum_{e \in \cE} \tr \left( (\Sigma^{\ast,e} - \hat \Sigma^e) \left[ - (U_e H)^\top A^{-1} - A^{-\top} (U_e H) + (U_e H)^\top (U_e H) \right] \right)\\
	= {} & - 2 \sum_{e \in \cE} \tr \left( H^\top U_e A^{-1} (\Sigma^{\ast,e} - \hat \Sigma^e) \right) + \sum_{e \in \cE} \tr \left( (\Sigma^{\ast,e} - \hat \Sigma^e) H^\top U_e H \right).
	\label{eq:hk}
\end{align}
The first term in \eqref{eq:hk} can be bounded by
\begin{align}
	\left| -2 \sum_{e \in \cE} \tr \left( H^\top U_e A^{-1} (\Sigma^{\ast,e} - \hat \Sigma^e) \right) \right|
	\leq {} & 2 \| H \|_1 \| \sum_{e \in \cE} U_e A^{-1} (\Sigma^{\ast,e} - \hat \Sigma^e) \|_\infty \leq 2 \phi_n \| H \|_1,
	\label{eq:hm}
\end{align}
while the second term can be controlled by
\begin{align}
	\sum_{e \in \cE} \tr \left( (\Sigma^{\ast,e} - \hat \Sigma^e) H^\top U_e H \right)
	\leq \| H \|_1 \left\| \sum_{e \in \cE} U_e H (\Sigma^{\ast,e} - \hat \Sigma^e) \right\|_\infty.
	\label{eq:hl}
\end{align}
For each entry of the matrix on the right of \eqref{eq:hl}, indexed by \( i, j \in [p] \), we have
\begin{align*}
	\left| \left[\sum_{e \in \cE} U_e H (\Sigma^{\ast,e} - \hat \Sigma^e) \right]_{i, j} \right|
	\leq {} & \left| \sum_{e \in \cE} \sum_{k \in [p]} \1_{i \in U_e} (H)_{ik} (\hat{\Sigma}^e_{kj} - \Sigma^{\ast, e}_{kj}) \right|\\
	= {} & \left| \sum_{k \in [p]} (H)_{ik} \sum_{e \in \cE} \1_{i \in U_e} (\hat{\Sigma}^e_{kj} - \Sigma^{\ast, e}_{kj}) \right|\\
	\leq {} & \left(\sum_{k \in [p]} \left| (H)_{ik} \right| \right) \left(\max_{k \in [p]} \left| \sum_{e \in \cE} \1_{i \in U_e} (\hat{\Sigma}^e_{kj} - \Sigma^{\ast, e}_{kj}) \right|\right),
\end{align*}
so that
\begin{align*}
	\left\| \sum_{e \in \cE} U_e H (\Sigma^{\ast,e} - \hat \Sigma^e) \right\|_\infty
	\leq {} & \left( \max_{i \in [p]} \sum_{k \in [p]} | (H)_{ik} | \right) \left(\max_{i, j, k \in [p]} \left| \sum_{e} \1_{i \in U_e} (\hat{\Sigma}^e_{kj} - \Sigma^{\ast, e}_{kj}) \right|\right)\\
	\leq {} & \phi_n \max_{i \in [p]} \sum_{k \in [p]} | (H)_{ik} | = \phi_n \| H \|_{\infty, \infty}.
\end{align*}
Combined with the estimate \eqref{eq:hk}, this yields the second claim, \eqref{eq:bx}.
\end{proof}


	\begin{proof}
		[{Proof of Lemma \ref{lem:stoch-err}}]
	To begin, recall the definition of \( \hat\Sigma^e \), as a sum of \( n/E \) \iid samples, that is, for \( i, j \in [p] \),
	\begin{equation}
		\label{eq:hs}
		(\hat\Sigma^e)_{i, j} = \frac{E}{n} \sum_{\ell = 1}^n (X^e_\ell)_i (X^e_\ell)_j.
	\end{equation}
	By the definition of the sample distribution, we can write
	\begin{align*}
		(X^e_\ell)_i (X^e_\ell)_j
		= {} & \underbrace{\mathfrak{e}_i^\top (I - U_e B^\ast)^{-1} Z^e_{\ell}}_{=: Y_1} \underbrace{(Z^e_{\ell})^\top (I - U_e B^\ast)^{-\top} \mathfrak{e}_j}_{=: Y_2},
	\end{align*}
	where	the \( Z^e_\ell \) follow a \( \cN(0,1) \) distribution and are \iid, and Lemma \ref{lem:subgaussian-vector} ensures that both \( Y_1 \) and \( Y_2 \) are \( \sg(\sigma_{\max}(I - U_e B^\ast)) \) random variables.
	By Lemma \ref{lem:product-subgaussian}, we obtain that \( (X^e_\ell)_i (X^e_\ell)_j \sim \subE(\sigma_{\max}(I - U_e B^\ast)^2) \).

	Similarly,
	\begin{align}
		(A_e^{-1} X^e_{\ell})_i X^e_{\ell, j}
		= {} & \mathfrak{e}_i^\top (I - U_e B^\ast) (I - U_e B^\ast)^{-1} Z^e_{\ell} (Z^e_{\ell})^\top (I - U_e B^\ast)^{-\top} \mathfrak{e}_j\\
		= {} & \underbrace{\mathfrak{e}_i^\top Z^e_\ell}_{=: \tilde Y_1} \underbrace{(Z^e_\ell)^\top (I - U_e B^\ast)^\top \mathfrak{e}_j}_{=: \tilde Y_2}.
	\end{align}
	Here, we have \( \tilde{Y_1} \sim \sg(1) \) and \( \tilde{Y_2} = \mathfrak{e}_j^\top (I - U_e B^\ast) Z^e_\ell \sim \sg(\sigma_{\max}(I - U_e B^\ast)) \).
	By again applying Lemma \ref{lem:product-subgaussian}, this means that \( (A_e^{-1} X^e_{\ell})_i X^e_{\ell, j} \sim \subE(\sigma_{\max}(I - U_e B^\ast)) \).

	Having established this, to obtain an estimate for \( T_1 \), we employ Bernstein's inequality, Lemma~\ref{lem:bernstein} to the sum in \eqref{eq:hs} for each \( e \in \cE, i, j \in [p] \) to see
	\begin{align*}
		\p \left( \left| (\Sigma^{\ast, e} - \hat \Sigma^e)_{i, j} \right| \geq t_1 \right)
		\leq 2 \exp \left[ -c_B \left( \left( \frac{n t_1^2}{E K_1^2} \right) \wedge \left( \frac{n t_1}{E K_1} \right) \right) \right],
	\end{align*}
	for \( t_1 > 0 \), with an absolute constant \( c_B \) and \( K_1 = \max_e \sigma_{\max}(I - U_eB)^2 \).
	Here, we made use of the fact that subtracting \( \Sigma^{\ast, e} \) centers the variables in the sum and that there are \( n/E \) independent summands in \eqref{eq:hi}.
	By a union bound,
	\begin{align}
		\leadeq{\p \left( \max_{e, i, j} \left| (\Sigma^{\ast, e} - \hat \Sigma^e)_{i, j} \right| \geq t_1 \right)}\\
		\leq {} & 2 p^2 E \exp \left[ -c_B \left( \frac{n t_1^2}{E K_1^2} \right) \wedge \left( \frac{n t_1}{E K_1} \right)\right]\\
		\leq {} & \exp \left[ -c_B \left( \frac{n t_1^2}{E K_1^2} \right) \wedge \left( \frac{n t_1}{E K_1} \right) + 2\log(2pE) \right].
		\label{eq:hu}
	\end{align}

	To bound \( T_2 \), for \( i, j \in [p] \), we write
	\begin{align*}
		\left[ \sum_{e \in \cE} U_e A^{-1} (\Sigma^{\ast,e} - \hat \Sigma^e) \right]_{i, j}
		= {} & \sum_{e} \1_{i \in \cU_e} (A^{-1} (\Sigma^{\ast,e} - \hat \Sigma^e))_{i, j} \\
		= {} & \sum_{e} \sum_{\ell = 1}^{n/E} a_{e, \ell} ((A_e^{-1} X^e_\ell)_i (X^{e}_\ell)_j - \E [(A_e^{-1} X^e_\ell)_i (X^{e}_\ell)_j])
	\end{align*}
	with
	\begin{equation*}
		a_{e, \ell} = \1_{i \in \cU_e} \frac{E}{n}.
	\end{equation*}
	By Bernstein's inequality, Lemma \ref{lem:bernstein}, for \( t_2 > 0 \),
	\begin{align}
		\label{eq:ht}
		\p \left( \left| \left[ \sum_{e \in \cE} U_e A^{-1} (\Sigma^{\ast,e} - \hat \Sigma^e) \right]_{i, j} \right| \geq t_2 \right)
		\leq 2 \exp \left[ -c_B \left( \frac{t_2^2}{K_2^2 \| a \|_2^2} \right) \wedge \left( \frac{t_2}{K_2 \| a \|_\infty} \right) \right],
	\end{align}
	where \( c_B \) is an absolute constant and
	\begin{equation*}
		K_2 = \max_e \sigma_{\max}(I - U_eB), \quad \| a \|_2^2 = \sum_{e} \frac{E}{n} = \frac{E^2}{n}, \quad \| a \|_\infty = \max_e \left\{ \frac{E}{n} \right\} = \frac{E}{n}.
	\end{equation*}
	A union bound then yields
	\begin{align}
		\label{eq:hv}
		\p \left( \max_{i, j} \left| \left[ \sum_{e \in \cE} U_e A^{-1} (\Sigma^{\ast,e} - \hat \Sigma^e) \right]_{i, j} \right| \geq t_2 \right)
		\leq {} & 2 p^2 \exp \left[ -c_B \left( \frac{n t_2^2}{K_2^2 E^2} \right) \wedge \left( \frac{n t_2}{E K_2} \right) \right]\\
		\leq {} & \exp \left[ -c_B \left( \frac{n t_2^2}{K_2^2 E^2} \right) \wedge \left( \frac{n t_2}{E K_2} \right) + 2 \log(2 p) \right].
	\end{align}

	To bound \( T_3 \), we proceed similarly.
	Using \( (X^e_\ell)_i (X^e_\ell)_j \sim \subE(\sigma_{\max}(I - U_eB)^2) \) instead of \( (A_e^{-1} X^e_{\ell})_i X^e_{\ell, j} \sim \subE(\sigma_{\max}(I - U_eB)) \), for \( t_3 > 0 \), we have
	\begin{align}
		\label{eq:hw}
		\p \left( \max_{i, j, k} \left| \sum_{e} \1_{k \in U_e} (\hat{\Sigma}^e_{ij} - \Sigma^{\ast, e}_{ij}) \right|	\geq t_3 \right)
		\leq {} & 2 p^3 \exp \left[ -c_B \left( \frac{n t_3^2}{K_3^2 E^2} \right) \wedge \left( \frac{n t_3}{E K_3} \right) \right]\\
		\leq {} & \exp \left[ -c_B \left( \frac{n t^2}{K_3^2 E^2} \right) \wedge \left( \frac{n t_3}{E K_3} \right) + 3 \log(2 p) \right],
	\end{align}
	where \( K_3 = \max_e \sigma_{\max}(I - U_e B^\ast)^2 \).

	Combined, recalling that by Lemma \ref{lem:boundedness}, \( \sigma_{\max}(I - U_e B^\ast) \le 2 \) and applying a union bound, we see that the union of the events in \eqref{eq:hu}, \eqref{eq:hv}, and \eqref{eq:hw} occurs at most with probability \( \delta \) if
	\begin{align*}
		t_1 \geq {} & C \left[ \sqrt{\frac{E \log(e p E/\delta)}{n}} \vee \frac{E \log(e p E/\delta)}{n} \right],\\
		t_2 \wedge t_3 \geq {} & C \left[ \sqrt{\frac{E^2 \log(e p/\delta)}{n}} \vee \frac{E \log(e p/\delta)}{n} \right].
	\end{align*}
	Restricting \( n \) to be large enough so that the effective part of the bound is the square root term in both cases then yields the claim.
\end{proof}

\section{Technical lemmas}
\label{sec:general-lemmas}

\begin{lemma}
	\label{lem:boundedness}
	If \( B \in \R^{p \times p} \) is such that $\| B \|_{\mathrm{op}} \leq 1 - \eta$ for some $\eta>0$,
	then we have
	\begin{align}
		\label{eq:gs}
		\max_{i \in [p]} \sum_{k = 1}^{p} B_{ki}^2 \vee \max_{i \in [p]} \sum_{k = 1}^{p} B_{ik}^2 \leq {} & 1.
	\end{align}
	Moreover, for any diagonal matrix \( U = \diag u \) with \( u \in \{0,1\}^{p} \),
	\begin{align}
		\label{eq:ha}
		\sigma_{\max}(I - U B) \le {} & 2,\qquad \text{and} \qquad 
		(\sigma_{\min}(I - U B))^{-1} \le \frac{1}{\eta}.
	\end{align}
\end{lemma}

\begin{proof}
	Let \( B \) and \( U \) as in the assumptions above.
	First, note that by its diagonal structure,
	\begin{equation}
		\label{eq:hb}
		\| U \|_{\op} = \max_{i \in [p]} | u_i | \le 1.
	\end{equation}
	Next, We can relate the maximum and minimum singular values to the operator norm and employ sub-additivity and sub-multiplicativity as follows:
	\begin{align}
		\label{eq:gq}
		\sigma_{\max}(I - U B) = {} & \| I - U B \|_{\mathrm{op}} \leq 1 + \| U \|_{\op} \| B \|_{\mathrm{op}} \leq 2,\\
		(\sigma_{\min}(I - U B))^{-1} = {} & \| (I - U B)^{-1} \|_{\op} = \left\| \sum_{k \geq 0} (U B)^{k} \right\|_{\mathrm{op}}
		\leq \sum_{k \geq 0} \| U \|_{\op}^k \| B \|_{\mathrm{op}}^k \leq \frac{1}{1 - (1-\eta)} = \frac{1}{\eta}.
	\end{align}
	Moreover, for \( i \in [p] \), denoting the standard unit vector with \( 1 \) in the \( i \)th coordinate by \( \mathfrak{e}_i \), we have
	\begin{align}
		\label{eq:hc}
		\sum_{k = 1}^{p} B_{k,i}^2 = \| B_{:, i} \|_2^2 = \| B \mathfrak{e}_i \|_{2}^2 \le \| B \|_{\op}^2 \| \mathfrak{e}_i \|_2^2
		\le (1-\eta)^2 \le 1
	\end{align}
	and the same argument yields the bound for \( \sum_{k = 1}^p B_{i, k}^2 \) by transposing the matrix and \( \| B \|_{\op} = \| B^\top \|_{\op} \).
\end{proof}

\begin{definition}
	[Sub-Gaussian and sub-Exponential random variables]
	\label{def:subg-sube}
	We call a random variable \( X \) sub-Gaussian with variance proxy \( \sigma^2 \), written \( X \sim \sg(\sigma^2) \), if
	\begin{equation}
		\label{eq:ho}
		\E[\exp(X^2/\sigma^2)] \le 2.
	\end{equation}
	We call a random variable sub-exponential with parameter \( \lambda \), written \( X \sim \subE(\lambda) \), if
	\begin{equation}
		\label{eq:hp}
		\E[\exp(|X|/\lambda)] \le 2.
	\end{equation}
\end{definition}

\begin{lemma}
	[{Product of \( \sg \) random variables is \( \subE \), \cite[Lemma 2.7.7]{Ver}}]
	\label{lem:product-subgaussian}
	~\\
	If \( X \sim \sg(\sigma_X^2) \) and \( Y \sim \sg(\sigma_Y^2) \), then
	\begin{equation}
		\label{eq:hn}
		X Y \sim \subE(\sigma_X \sigma_Y).
	\end{equation}
\end{lemma}

\begin{lemma}
	[{Sum of independent sub-Gaussian variables, \cite[Proposition 2.6.1]{Ver}}]
	\label{lem:subgaussian-vector}
	~\\
	If \( X_1, \dots, X_n \) are \( n \) independent mean-zero random variables such that \( X_i \sim \sg(\sigma_i^2) \), then
	\begin{equation}
		\label{eq:hq}
		\sum_{i=1}^n X_i \sim \sg(\sigma^2), \quad \text{with } \sigma^2 = \sum_{i=1}^n \sigma_i^2.
	\end{equation}
\end{lemma}

\begin{lemma}
	[{Bernstein's inequality, \cite[Theorem 2.8.1]{Ver}}]
	\label{lem:bernstein}
	Let \( X_1, \dots, X_n \) be \( n \) independent mean-zero random variables such that \( X_i \sim \subE(\lambda_i) \).
	Then, there is an absolute constant \( c_B \) such that for \( t > 0 \), 
	\begin{equation}
		\label{eq:hr}
		\p\left( \left| \sum_{i=1}^{n} X_i \right| \ge t \right)
		\le 2 \exp \left( -c_B \min \left( \frac{t^2}{\sum_{i=1} \lambda_i^2}, \frac{t}{\max_{i \in [n]} \lambda_i} \right) \right).
	\end{equation}
\end{lemma}

\bibliographystyle{alphaabbr}
\bibliography{CyclicCausal}

\newcommand{\etalchar}[1]{$^{#1}$}
\begin{thebibliography}{HJM{\etalchar{+}}09}

\bibitem[AR18]{AbrRig18}
N.~Abrahamsen and P.~Rigollet.
\newblock Sparse {{Gaussian ICA}}.
\newblock {\em arXiv preprint arXiv:1804.00408}, 2018.

\bibitem[BH77]{BieHau77}
W.~T. Bielby and R.~M. Hauser.
\newblock Structural equation models.
\newblock {\em Annual review of sociology}, 3(1):137--161, 1977.

\bibitem[BKSV15]{BenKnoSch15}
M.~Benning, F.~Knoll, C.-B. Sch\"onlieb, and T.~Valkonen.
\newblock Preconditioned {{ADMM}} with nonlinear operator constraint.
\newblock {\em arXiv:1511.00425 [math]}, November 2015.

\bibitem[BLT18]{BelLecTsy18}
P.~C. Bellec, G.~Lecu\'e, and A.~B. Tsybakov.
\newblock Slope meets lasso: Improved oracle bounds and optimality.
\newblock {\em The Annals of Statistics}, 46(6B):3603--3642, 2018.

\bibitem[Bol83]{Bol83}
K.~Bollen.
\newblock A.(1989). {{Structural}} equations with latent variables.
\newblock {\em new york, ny: wiley. doi}, 10:9781118619179, 1983.

\bibitem[Boo86]{Boo86}
W.~M. Boothby.
\newblock {\em An Introduction to Differentiable Manifolds and {{Riemannian}}
  Geometry}, volume 120.
\newblock {Academic press}, 1986.

\bibitem[BPC{\etalchar{+}}11]{BoyParChu11}
S.~Boyd, N.~Parikh, E.~Chu, B.~Peleato, and J.~Eckstein.
\newblock Distributed optimization and statistical learning via the alternating
  direction method of multipliers.
\newblock {\em Foundations and Trends\textregistered{} in Machine Learning},
  3(1):1--122, 2011.

\bibitem[Cai84]{Mao84}
M.~Cai.
\newblock On a problem of {{Katona}} on minimal completely separating systems
  with restrictions.
\newblock {\em Discrete Mathematics}, 48(1):121--123, January 1984.

\bibitem[CBG13]{CaiBazGia13}
X.~Cai, J.~A. Bazerque, and G.~B. Giannakis.
\newblock Inference of gene regulatory networks with sparse structural equation
  models exploiting genetic perturbations.
\newblock {\em PLoS computational biology}, 9(5):e1003068, 2013.

\bibitem[Chi02]{Chi02a}
D.~M. Chickering.
\newblock Learning equivalence classes of {{Bayesian}}-network structures.
\newblock {\em Journal of machine learning research}, 2(Feb):445--498, 2002.

\bibitem[Dic69]{Dic69}
T.~J. Dickson.
\newblock On a problem concerning separating systems of a finite set.
\newblock {\em Journal of Combinatorial Theory}, 7(3):191--196, November 1969.

\bibitem[Dun66]{Dun66}
O.~D. Duncan.
\newblock Path analysis: {{Sociological}} examples.
\newblock {\em American journal of Sociology}, 72(1):1--16, 1966.

\bibitem[EB92]{EckBer92}
J.~Eckstein and D.~P. Bertsekas.
\newblock On the {{Douglas}}\textemdash{{Rachford}} splitting method and the
  proximal point algorithm for maximal monotone operators.
\newblock {\em Mathematical Programming}, 55(1):293--318, 1992.

\bibitem[FHT08]{FriHasTib08}
J.~Friedman, T.~Hastie, and R.~Tibshirani.
\newblock Sparse inverse covariance estimation with the graphical lasso.
\newblock {\em Biostatistics}, 9(3):432--441, July 2008.

\bibitem[FLNP00]{FriLinNac00}
N.~Friedman, M.~Linial, I.~Nachman, and D.~Pe'er.
\newblock Using {{Bayesian}} networks to analyze expression data.
\newblock {\em Journal of computational biology}, 7(3-4):601--620, 2000.

\bibitem[Fra12]{Fra12}
J.~N. Franklin.
\newblock {\em Matrix Theory}.
\newblock {Courier Corporation}, 2012.

\bibitem[Gab83]{Gab83}
D.~Gabay.
\newblock Chapter ix applications of the method of multipliers to variational
  inequalities.
\newblock {\em Studies in mathematics and its applications}, 15:299--331, 1983.

\bibitem[GM75]{GloMar75}
R.~Glowinski and A.~Marroco.
\newblock Sur l'approximation, par \'el\'ements finis d'ordre un, et la
  r\'esolution, par p\'enalisation-dualit\'e d'une classe de probl\`emes de
  {{Dirichlet}} non lin\'eaires.
\newblock {\em Revue fran{\c c}aise d'automatique, informatique, recherche
  op\'erationnelle. Analyse num\'erique}, 9(R2):41--76, 1975.

\bibitem[GM76]{GabMer76}
D.~Gabay and B.~Mercier.
\newblock A dual algorithm for the solution of nonlinear variational problems
  via finite element approximation.
\newblock {\em Computers \& Mathematics with Applications}, 2(1):17--40, 1976.

\bibitem[HB12]{HauBuh12}
A.~Hauser and P.~B\"uhlmann.
\newblock Characterization and greedy learning of interventional {{Markov}}
  equivalence classes of directed acyclic graphs.
\newblock {\em Journal of Machine Learning Research}, 13(Aug):2409--2464, 2012.

\bibitem[HDRS11]{HsiDhiRav11}
C.-J. Hsieh, I.~S. Dhillon, P.~K. Ravikumar, and M.~A. Sustik.
\newblock Sparse inverse covariance matrix estimation using quadratic
  approximation.
\newblock In {\em Advances in Neural Information Processing Systems}, pages
  2330--2338, 2011.

\bibitem[HEH12]{HytEbeHoy12}
A.~Hyttinen, F.~Eberhardt, and P.~O. Hoyer.
\newblock Learning linear cyclic causal models with latent variables.
\newblock {\em Journal of Machine Learning Research}, 13(Nov):3387--3439, 2012.

\bibitem[HEH13]{HytEbeHoy13}
A.~Hyttinen, F.~Eberhardt, and P.~O. Hoyer.
\newblock Experiment selection for causal discovery.
\newblock {\em The Journal of Machine Learning Research}, 14(1):3041--3071,
  2013.

\bibitem[HJM{\etalchar{+}}09]{HoyJanMoo09}
P.~O. Hoyer, D.~Janzing, J.~M. Mooij, J.~Peters, and B.~Sch\"olkopf.
\newblock Nonlinear causal discovery with additive noise models.
\newblock In {\em Advances in Neural Information Processing Systems}, pages
  689--696, 2009.

\bibitem[HYW00]{HeYanWan00}
B.~S. He, H.~Yang, and S.~L. Wang.
\newblock Alternating direction method with self-adaptive penalty parameters
  for monotone variational inequalities.
\newblock {\em Journal of Optimization Theory and applications},
  106(2):337--356, 2000.

\bibitem[IOS{\etalchar{+}}10]{ItaOhaSac10}
S.~Itani, M.~Ohannessian, K.~Sachs, G.~P. Nolan, and M.~A. Dahleh.
\newblock Structure learning in causal cyclic networks.
\newblock In {\em Causality: {{Objectives}} and {{Assessment}}}, pages
  165--176, 2010.

\bibitem[KB07]{KalBuh07}
M.~Kalisch and P.~B\"uhlmann.
\newblock Estimating high-dimensional directed acyclic graphs with the
  {{PC}}-algorithm.
\newblock {\em Journal of Machine Learning Research}, 8(Mar):613--636, 2007.

\bibitem[KDV17]{KocDimVis17}
M.~Kocaoglu, A.~Dimakis, and S.~Vishwanath.
\newblock Cost-optimal learning of causal graphs.
\newblock In {\em Proceedings of the 34th International Conference on Machine
  Learning-Volume 70}, pages 1875--1884. JMLR. org, 2017.

\bibitem[KH88]{KeaHit88}
B.~W. Keats and M.~A. Hitt.
\newblock A causal model of linkages among environmental dimensions, macro
  organizational characteristics, and performance.
\newblock {\em Academy of management journal}, 31(3):570--598, 1988.

\bibitem[LB14]{LohBuh14}
P.-L. Loh and P.~B\"uhlmann.
\newblock High-dimensional learning of linear causal networks via inverse
  covariance estimation.
\newblock {\em Journal of Machine Learning Research}, 15(1):3065--3105, 2014.

\bibitem[LN89]{LiuNoc89}
D.~C. Liu and J.~Nocedal.
\newblock On the limited memory {{BFGS}} method for large scale optimization.
\newblock {\em Mathematical programming}, 45(1):503--528, 1989.

\bibitem[LSRH12]{LacSpiRam12}
G.~Lacerda, P.~L. Spirtes, J.~Ramsey, and P.~O. Hoyer.
\newblock Discovering cyclic causal models by independent components analysis.
\newblock {\em arXiv preprint arXiv:1206.3273}, 2012.

\bibitem[LW11]{LohWai11a}
P.-L. Loh and M.~J. Wainwright.
\newblock High-dimensional regression with noisy and missing data: {{Provable}}
  guarantees with non-convexity.
\newblock In {\em Advances in {{Neural Information Processing Systems}}}, pages
  2726--2734, 2011.

\bibitem[LW13]{LohWai13}
P.-L. Loh and M.~J. Wainwright.
\newblock Regularized {{M}}-estimators with nonconvexity: {{Statistical}} and
  algorithmic theory for local optima.
\newblock In {\em Advances in {{Neural Information Processing Systems}}}, pages
  476--484, 2013.

\bibitem[MKB09]{MaaKalBuh09}
M.~H. Maathuis, M.~Kalisch, and P.~B{\"u}hlmann.
\newblock Estimating high-dimensional intervention effects from observational
  data.
\newblock {\em Ann. Statist.}, 37(6A):3133--3164, 12 2009.

\bibitem[NW06]{NocWri06}
J.~Nocedal and S.~J. Wright.
\newblock {\em Numerical Optimization 2nd}.
\newblock {Springer}, 2006.

\bibitem[PB14]{PetBuh14}
J.~Peters and P.~B\"uhlmann.
\newblock Identifiability of {{Gaussian}} structural equation models with equal
  error variances.
\newblock {\em Biometrika}, 101(1):219--228, January 2014.

\bibitem[Pea09]{Pea09}
J.~Pearl.
\newblock {\em Causality: {{Models}}, {{Reasoning}} and {{Inference}}}.
\newblock {Cambridge University Press}, second edition, 2009.

\bibitem[RBLZ08]{RotBicLev08}
A.~J. Rothman, P.~J. Bickel, E.~Levina, and J.~Zhu.
\newblock Sparse permutation invariant covariance estimation.
\newblock {\em Electronic Journal of Statistics}, 2:494--515, 2008.

\bibitem[Ric96]{Ric96a}
T.~Richardson.
\newblock {\em Feedback Models: {{Interpretation}} and Discovery}.
\newblock PhD thesis, Ph. D. thesis, Carnegie Mellon, 1996.

\bibitem[RS96]{RicSpi96}
T.~Richardson and P.~Spirtes.
\newblock Automated discovery of linear feedback models.
\newblock manuscript, 1996.

\bibitem[See04]{See04}
M.~Seeger.
\newblock Low rank updates for the {{Cholesky}} decomposition.
\newblock {\em Infoscience, EPFL Scientific Publications}, 2004.

\bibitem[SGS00]{SpiGlySch00}
P.~Spirtes, C.~N. Glymour, and R.~Scheines.
\newblock {\em Causation, Prediction, and Search}.
\newblock {MIT press}, 2000.

\bibitem[SHHK06]{ShiHoyHyv06}
S.~Shimizu, P.~O. Hoyer, A.~Hyv\"arinen, and A.~Kerminen.
\newblock A linear non-{{Gaussian}} acyclic model for causal discovery.
\newblock {\em Journal of Machine Learning Research}, 7(Oct):2003--2030, 2006.

\bibitem[SKDV15]{ShaKocDim15}
K.~Shanmugam, M.~Kocaoglu, A.~G. Dimakis, and S.~Vishwanath.
\newblock Learning {{Causal Graphs}} with {{Small Interventions}}.
\newblock In {\em Advances in {{Neural Information Processing Systems}}}, pages
  3195--3203, 2015.

\bibitem[SM09]{SchMur09}
M.~Schmidt and K.~Murphy.
\newblock Modeling discrete interventional data using directed cyclic graphical
  models.
\newblock In {\em Proceedings of the {{Twenty}}-{{Fifth Conference}} on
  {{Uncertainty}} in {{Artificial Intelligence}}}, pages 487--495. {AUAI
  Press}, 2009.

\bibitem[SNM07]{SchNicMur07}
M.~Schmidt, A.~{Niculescu-Mizil}, and K.~Murphy.
\newblock Learning graphical model structure using {{L1}}-regularization paths.
\newblock In {\em {{AAAI}}}, volume~7, pages 1278--1283, 2007.

\bibitem[TBA06]{TsaBroAli06}
I.~Tsamardinos, L.~E. Brown, and C.~F. Aliferis.
\newblock The max-min hill-climbing {{Bayesian}} network structure learning
  algorithm.
\newblock {\em Machine learning}, 65(1):31--78, 2006.

\bibitem[Tsy09]{Tsy09}
A.~B. Tsybakov.
\newblock {\em Introduction to Nonparametric Estimation. {{Revised}} and
  Extended from the 2004 {{French}} Original. {{Translated}} by {{Vladimir
  Zaiats}}}.
\newblock {Springer Series in Statistics. Springer, New York}, 2009.

\bibitem[vdGB13]{GeeBuh13}
S.~van~de Geer and P.~B\"{u}hlmann.
\newblock $\ell_{0}$-penalized maximum likelihood for sparse directed acyclic
  graphs.
\newblock {\em Ann. Statist.}, 41(2):536--567, 04 2013.

\bibitem[Ver18]{Ver}
R.~Vershynin.
\newblock {\em High-dimensional probability: An introduction with applications
  in data science}, volume~47.
\newblock Cambridge University Press, 2018.

\bibitem[WL01]{WanLia01}
S.~L. Wang and L.~Z. Liao.
\newblock Decomposition method with a variable parameter for a class of
  monotone variational inequality problems.
\newblock {\em Journal of optimization theory and applications},
  109(2):415--429, 2001.

\bibitem[WSU18]{WanSegUhl18}
Y.~Wang, S.~Segarra, and C.~Uhler.
\newblock High-{{Dimensional Joint Estimation}} of {{Multiple Directed Gaussian
  Graphical Models}}.
\newblock {\em arXiv preprint arXiv:1804.00778}, 2018.

\bibitem[WYZ15]{WanYinZen15}
Y.~Wang, W.~Yin, and J.~Zeng.
\newblock Global {{Convergence}} of {{ADMM}} in {{Nonconvex Nonsmooth
  Optimization}}.
\newblock {\em arXiv:1511.06324 [cs, math]}, November 2015.

\bibitem[ZBLN97]{ZhuByrLu97}
C.~Zhu, R.~H. Byrd, P.~Lu, and J.~Nocedal.
\newblock Algorithm 778: {{L}}-{{BFGS}}-{{B}}: {{Fortran}} subroutines for
  large-scale bound-constrained optimization.
\newblock {\em ACM Transactions on Mathematical Software (TOMS)},
  23(4):550--560, 1997.

\end{thebibliography}

\end{document}